\documentclass[10pt]{article}
\usepackage[english]{babel}
\usepackage[latin1]{inputenc}
\usepackage{amsmath}
\usepackage{amsfonts}
\usepackage{amssymb}
\usepackage{amsthm}
\usepackage{mathrsfs}

\newtheorem{thm}{Theorem}
\newtheorem{defn}[thm]{Definition}
\newtheorem{rem}[thm]{Remark}
\newtheorem{cor}[thm]{Corollary}
\newtheorem{prop}[thm]{Proposition}
\newtheorem{lem}[thm]{Lemma}

\begin{document}
\title{ {\large Random Walks in the Quarter Plane Absorbed at the Boundary : Exact and Asymptotic}}
\author{{\large Kilian Raschel\footnote{Laboratoire de Probabilit\'es et Mod\`eles Al\'eatoires, Universit\'e Pierre et Marie Curie,
4 Place Jussieu 75252 Paris Cedex 05, France. E-mail : \texttt{kilian.raschel@upmc.fr}}}}
\date{\today}
\maketitle
\begin{abstract}
Nearest neighbor random walks in the quarter plane
that are absorbed when reaching the boundary are studied.
The cases of positive and zero drift are considered.
Absorption probabilities at a given time 
and at a given site are made explicit. The following asymptotics
for these random walks 
starting from a given point $(n_0,m_0)$ 
are computed : that
of probabilities of being absorbed at a given site $(i,0)$
[resp. $(0,j)$] as $i\to \infty$
[resp. $j \to \infty$], that of the distribution's tail 
of absorption time at $x$-axis
[resp. $y$-axis], that of the Green functions 
at site $(i,j)$ when $i,j\to \infty$ and $j/i \to \tan \gamma$ for 
$\gamma \in [0, \pi/2]$. These results give 
the Martin boundary of the process and in particular
the suitable Doob $h$-transform in order to condition the process
never to reach the boundary. They also show that this $h$-transformed
process is equal in distribution to the limit as $n\to \infty$
of the process conditioned by not being absorbed at time $n$.
The main tool used here is complex analysis.    
\end{abstract}


\noindent {\it AMS $2000$ Subject Classification : primary 60G50, 60G40, 31C35 ; secondary 30E20, 30E25.}

\section{Introduction}\label{Intro}

 The interest in random processes in open domains of $\mathbb{Z}^{2}$
 conditioned in the sense of Doob $h$-transform
 never to reach the boundary dates back to Dyson~\cite{Dy62}.
 He looked at a process
 version of the famous Gaussian Unitary Ensemble (GUE) and observed
 that the process of  vectors of eigenvalues of that matrix process
 is equal in distribution to a family of standard Brownian motions
 conditioned on never colliding.

 After quiet years there was renewed interest in the 90s.
 An important class of such processes is since then studied,
 the so called ``non-colliding'' random
 walks, also called  ``vicious walkers'' or ``non-intersecting paths''.
 These walks are the processes $Z(n)=(Z_{1}(n), \ldots, Z_{k}(n))_{n\geq 0}$
 composed of $k$ independent and identically distributed
 random walks that never leave the Weyl chamber
 $W=\{z\in \mathbb{R}^{k} : z_{1}<\cdots <z_{k} \}$.
 The distances between these random walks 
 $(Z_{2}(n)-Z_{1}(n),\ldots, Z_{k}(n)-Z_{k-1}(n))_{n\geq 0}$
 give a $k-1$-dimensional random process whose
 components are positive.
 It turned out that these processes
 appear in the eigenvalues description of interesting
 matrix-valued stochastic processes (see
 e.g.~\cite{Bru91},~\cite{KO01},~\cite{KT4},~\cite{Gr99},~\cite{HW96})
 and in the analysis of corner-growth model
 (see~\cite{Jo00} and~\cite{Jo02}).
 Moreover, interesting connections between non-colliding walks, random
 matrices and queues in tandem are the subject of \cite{OC03}.
 Chapter~4 of~\cite{K05} gives besides a survey on this topic.

 It turns out that it is possible to construct
 these processes thanks to a suitable Doob
 $h$-transform.
 Paper~\cite{KP} reveals the general mechanism
 of this construction~:
 the authors find there --under rather general assumptions--
 a positive regular function $h$, namely 
 $h(z)=\prod_{1\leq i<j\leq k}(z_{j}-z_{i})$, such that the Doob $h$-transformed of $Z$,
 defined by $\hat{P}_{u}^{h}(Z(n)\in \text{d}v)=\mathbb{P}_{u}(\tau>n , Z(n) \in
 \text{d}v)h(v)/h(u)$,  where $\tau=\inf\{n>0 : Z(n) \notin W \} $,
 is equal to the conditional version of $Z$ given
 never exiting the Weyl chamber $W$.
 \textit{Prima facie}, it is not only the existence
 of such functions $h$ that is far from clear, but also
 the fact that the corresponding process $\hat{P}^{h}$
 has anything to do with the limit as $n\to \infty$
 of the conditional version of $Z$ given $\{\tau>n\}$.
 To prove these results, the authors compute the asymptotic behavior of the probabilities
 $\mathbb{P}_{u}(\tau>n)$.
 In their paper, they also show that the rescaled conditional process
 $(n^{-1/2}Z(\lfloor t n\rfloor ))_{t\geq 0}$ converges towards Dyson's Brownian motion.

 Most of the previous results concern only distances between independent
 random walks.
   In~\cite{KOR02}, random walks with exchangeable
   increments and conditioned never to exit the Weyl chamber are
   considered. In~\cite{OCY},
   the authors study a certain class of random
   walks, namely $(X_{i}(n))_{1\leq i\leq k}=
   (|\{1\leq m \leq n : \xi_{m}=i\}|)_{1\leq i\leq k}$, where
   $(\xi_{m}, m\geq 1)$ is a sequence of independent 
   and identically distributed random variables with
   common distribution on $\{1,2,\ldots, k\}$, and
   identify in law their conditional version with a certain
   path-transformation of the initial process.
   In~\cite{ORSK1} and~\cite{ORSK2}, O'Connell relates
   these objects to the Robinson-Schensted algorithm.

 Another important area where random processes
 in angles of $\mathbb{Z}^{d}$
 conditioned never to reach the boundary
 appear is the domain of ``quantum random walks''. In~\cite{Bi2}, Biane
 constructs a quantum Markov chain on the von Neumann
 algebra of ${\rm SU}(n)$ and interprets the restriction
 of this quantum Markov chain to the algebra of a maximal
 torus of ${\rm SU}(n)$ as a random walk on the lattice of integral forms on ${\rm SU}(n)$
 with respect to this maximal torus. He proves that the restriction of the quantum
 Markov chain to the center of the von Neumann algebra
 is a Markov chain on the same lattice  obtained from the
 preceding by conditioning it in Doob's sense
 to exit a Weyl chamber at infinity.
 In case $n=3$, the Weyl chamber
 of the corresponding Lie algebra $\mathfrak{sl}_{3}(\mathbb{C})$
 is the domain of $(\mathbb{R}_{+})^{2}$ delimited on the one hand
 by the $x$-axis and on the other by the axis making an
 angle equal to $\pi/3$ with the $x$-axis.
 One gets a spatially homogeneous random walk
 in the interior of the weights lattice,
 with three transition probabilities $1/3$
 as in the left side of Figure~\ref{Flatto_Lie}~;
 random walk which can be
 of course thought as a walk in $(\mathbb{Z}_{+})^2$
 with transition probabilities
 described in the second picture of Figure~\ref{Flatto_Lie}.
 Biane shows that for this walk, a proper $h$-transform
 $h(x,y)$ is the dimension of the representation
 of $\mathfrak{sl}_{3}(\mathbb{C})$ with highest weight $(x-1,y-1)$,
 equal to $x y (x+y)/2$.

 \begin{figure}[!h]
 \begin{picture}(100.00,80.00)
 \includegraphics{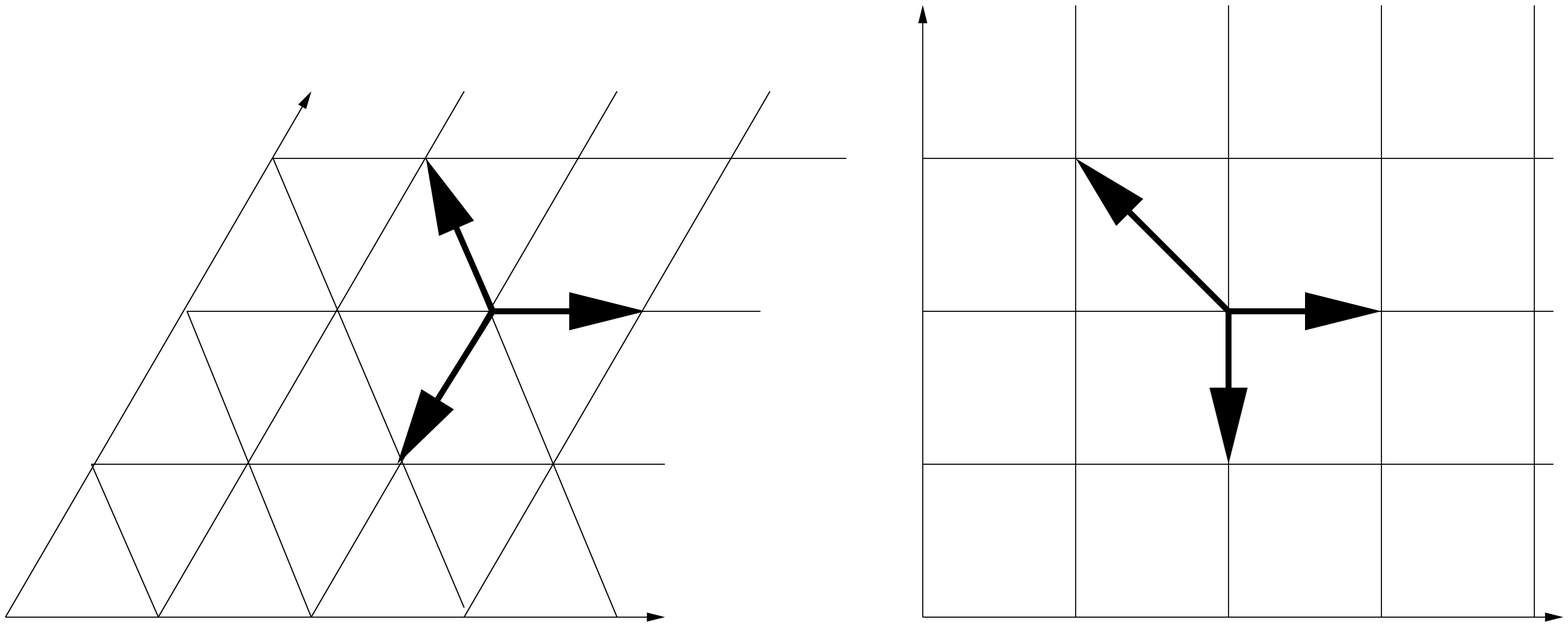}\hspace{82mm}
 \includegraphics{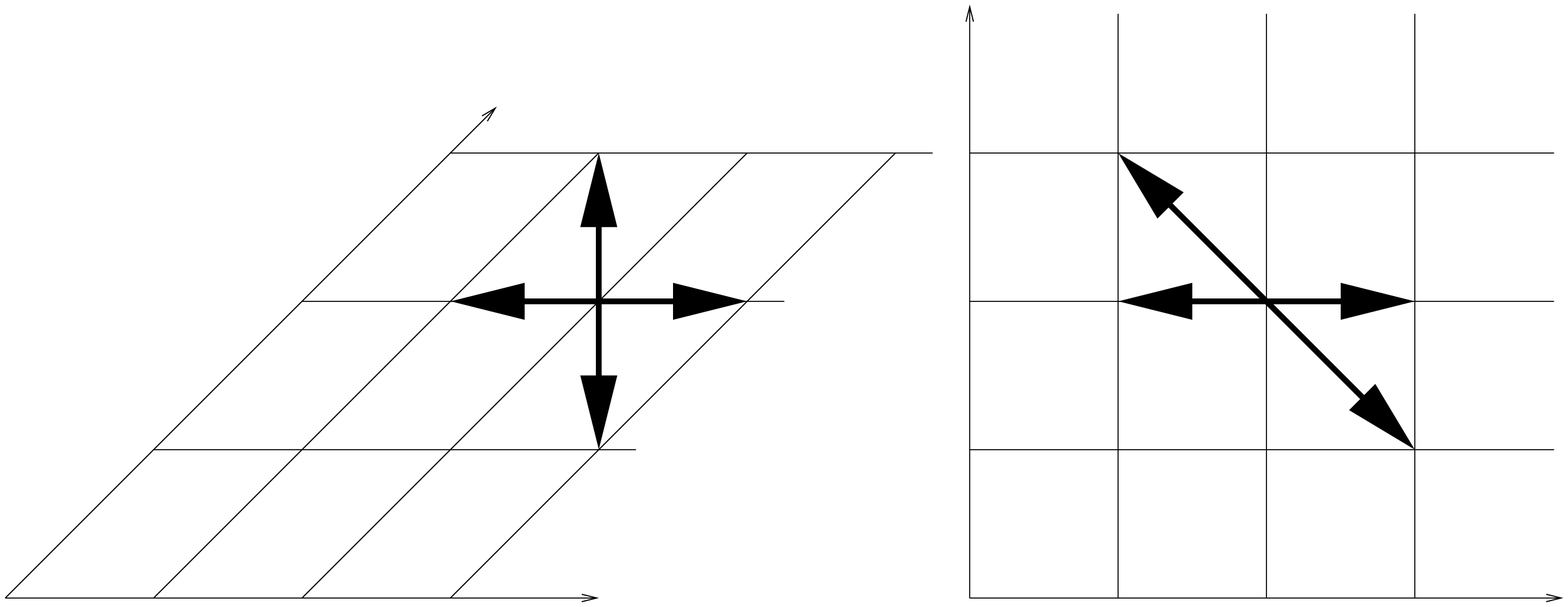}
 \end{picture}
 \caption{Walks in Weyl chambers of weights lattices of Lie
 algebras --above, $\mathfrak{sl}_{3}(\mathbb{C})$ and
 $\mathfrak{sp}_{4}(\mathbb{C})$-- can be viewed
 as random walks on $({\mathbb{Z}_{+}})^{d}$.}
 \label{Flatto_Lie}
 \end{figure}

 Then, in~\cite{Bi3}, Biane extends these results
 to the case of general semi-simple connected and simply connected
 compact Lie groups, the basic notion being that of minuscule weight.
 The corresponding random walk on the weights lattice in the
 interior of the Weyl chamber can be obtained as follows~:
 one draws the vector corresponding to the minuscule weight
 and its images under
 the Weyl group~; then one translates these  vectors
 to each point of the weight lattice in the interior of
 the Weyl chamber and we assign to them equal probabilities
 $2/l$, $l$ being the order of the Weyl group,
 these probabilities are the transitions of the walk.
  For example, in case of the Lie algebras
  $\mathfrak{sp}_{4}(\mathbb{C})$
  or $\mathfrak{so}_{5}(\mathbb{C})$,
  the associated Weyl chamber and
  random walks are drawn in the right side of Figure~\ref{Flatto_Lie}.
  In~\cite{Bi3}, Biane also makes some generalizations
  to non-centered random walks.
  Nevertheless these algebraic methods
  do not allow the computation of the distribution of random
  time $\tau$ to reach the boundary~;  they neither allow
  to relate the limit as $n \to \infty$ of
  the process conditioned by $\{\tau>n\}$ to a
  $h$-transformed process.

  In~\cite{Bi1}, Biane also
  computes the asymptotic of the Green functions
  $G_{x,y}$ for the first random walk on Figure~\ref{Flatto_Lie},
  asymptotic as $x,y \to \infty$ and $y/x \to \tan( \gamma)$, for
  $\gamma $ in $[\epsilon, \pi/2-\epsilon]$,
  $\epsilon>0$.
  The description of the Martin boundary for this random
  walk could not be completed since
  the asymptotic of the Green functions
  as $y/x \to 0$ or $y/x \to \infty$ could not be found.

  Once again with a view to applying to Lie algebras,
  Varopoulos studies in~\cite{VA1} and~\cite{VA2}
  random walks in general conical domains of $\mathbb{Z}^{d}$,
  inside of which the processes are supposed to be spatially homogeneous
  and to have non-correlated components.
  He estimates the distribution of time
  $\tau$ to reach the boundary, he shows more precisely that
  $\mathbb{P}(\tau >n)$ is bounded
  from above and below by $n^{-\alpha}$
  --up to some multiplicative constant-- with a proper
  exponent $\alpha$ depending on the conical domain
  and on the dimension $d$.

   In~\cite{ASP} and~\cite{ASP2}, the authors are interested in
   passage-times moments in centered balls for reflected random walks 
   in a quadrant homogeneous and with zero drift in the interior. 
   They find a critical exponent, depending only 
   on the transition probabilities,
   under (resp.\ above) which the passage-time moments 
   are finite (resp.\ infinite). They also give
   lower bounds for the tails of the distributions of the first-passage 
   times in centered balls for these walks,
   in terms of the same critical exponent as before.

  As for non-homogeneous random walks in $\mathbb{Z}^{d}$,
  a recent paper by~\cite{MMcPhW} studies
  the exit time moments of cones
  in case of an asymptotically zero drift.

   In~\cite{I}, Ignatiuk obtains, under general assumptions and
   for all $d\geq 2$, the Martin boundary of some
   random walks in the half-space $\mathbb{Z}^{d-1} \times \mathbb{Z}_{+}$
   killed on the boundary.  Her method can unfortunately not be
   generalized to random walks on $(\mathbb{Z}_{+})^{d}$, even
   for $d=2$.

   In this paper we restrict ourselves
   to spatially homogeneous random walks  $(X(n), Y(n))_{n\geq 0}$
   in $(\mathbb{Z}_{+})^2$ with jumps at distance at most $1$. We denote by
   $\mathbb{P}(X(n+1)=i_{0}+i, Y(n+1)=j_{0}+j \mid X(n)=i_{0},
   Y(n)=j_{0})=p_{(i_{0},j_{0}), (i+i_{0}, j+j_{0})}$
   the transition probabilities.
   So we do the hypothesis~:

\begin{itemize}
     \item[(H1)] {\it For all $\left(i_{0},j_{0}\right)$ such that $i_0>0,
                 j_0>0$, $p_{\left(i_{0},j_{0}\right) ,\left(i_{0},j_{0}\right)+(i,j)}$
                 does not depend on $\left(i_{0},j_{0}\right)$
                 and can thus be denoted by $p_{i j}$.}
     \item[(H2)] {\it  $p_{i,j}=0$ if $|i|>1$ or $|j|>1$.}
     \item[(H3)] {\it The boundary
                 $\left\{\left(0,0\right)\right\} \cup \left\{ \left(i,0\right):
                 \hspace{1mm}i\geq 1 \right\}
                 \cup \left\{ \left(0,j\right):\,j\geq 1 \right\}$
                 is absorbing.}
\end{itemize}
Let
     \begin{equation}\label{absorption_probabilities}
           h_{i,n}^{n_{0},m_{0}}  =  \mathbb{P}_{\left(n_{0} , m_{0}\right)}
               \left(\text{to hit}\, \left(i,0\right)\,\text{at time $n$}\right),\hspace{5mm}
               \widetilde{h}_{j,n}^{n_{0},m_{0}}  = \mathbb{P}_{\left(n_{0} , m_{0}\right)}
               \left(\text{to hit}\, \left(0,j\right)\,\text{at time $n$}\right)
     \end{equation}          
 be the probabilities of being absorbed at points $(i,0)$ and $(0,j)$ at time $n$.
 Let $h(x,z)$ and $\tilde{h}(y,z)$
 be their generating functions, initially defined for $|x|,|z|\leq 1$~:
     \begin{equation}\label{def_generating_functions}
               h\left(x,z\right)  = 
               \sum_{i\geq 1,n\geq0 } h_{i,n}^{n_{0},m_{0}} x^{i} z^{n} ,\hspace{5mm}
               \widetilde{h}\left(y,z\right)  = 
               \sum_{j \geq 1 ,n\geq 0} \widetilde{h}_{j,n}^{n_{0},m_{0}} y^{j} z^{n}.
     \end{equation}
Book~\cite{FIM} studies the random walks in
$(\mathbb{Z}_{+})^{2}$ under assumptions (H1) and (H2) but not (H3)~:
the jump probabilities from the boundaries to the interior of
$(\mathbb{Z}_{+})^2$ are there not zero and the $x$-axis, the $y$-axis and
$(0,0)$ are three other domains of spatial homogeneity. Moreover, the jumps
from the boundaries are supposed such that the Markov chain is
ergodic. The authors elaborate in this book
a profound and ingenious analytic
approach to compute the generating functions of
stationary probabilities of these random walks.
This approach serves as a starting point for our investigation
and therefore plays a crucial role~:
Subsections~\ref{Introduction}--\ref{Riemann_Carleman_problem} of this paper
leading to the first integral representation of the functions
$h(x,z)$ and $\tilde h(y,z)$ are inspired from~\cite{FIM}.
Indeed, we reduce, as there,
the computation of these functions to the resolve of a Riemann boundary value
problem with shift. Then, we use the classical way to study this kind of problem,
namely we transform it into a Riemann-Hilbert problem for which there exists
a suitable and complete theory~; the conversion between Riemann
problems with shift and Riemann-Hilbert problem being done thanks to the
use of conformal gluing function. On closer analysis we observed that
the conformal gluing function has an explicit and particularly nice
form under the simplifying hypothesis
\begin{itemize}
\item[(H2')] $p_{01}+p_{10}+p_{-10}+p_{0-1}=1$.
\end{itemize}
  Then we found it instructive to carry out first our analysis
  under this simplifying hypothesis (H2') that
  makes  all investigations much more transparent. Therefore in this paper
  we restrict ourselves to the random walks under
  hypothesis (H1), (H2'), (H3) above and (H2''), (H4) below~:
\begin{itemize}
\item[(H2'')] $p_{01}, p_{10}, p_{0-1},p_{-10}\neq 0$.
\item[(H4)] The drifts are non negative~:
\begin{equation}
\label{drift} M_{x}=p_{10}-p_{-10}\geq 0 ,
 \ \ \ M_{y}=p_{01}-p_{0-1} \geq 0.
\end{equation}
\end{itemize}
 At the end of this paper 
 we consider some extensions of the hypothesis (H2') that
 keep the conformal gluing function in the same nice form
 as for (H2').

 We will be interested here in the following questions.

\begin{itemize}
     \item[(1)] What  are $h(x,z)$ and $\tilde{h}(y,z)$ starting from $(n_{0},m_{0})$~?
     \item[(2)] Let $S$ (resp.\ $T$) be the first time of reaching the $x$-axis (resp.\ $y$-axis).
                What are the distributions' tails of $S$ and $T$
                starting from $(n_{0},m_{0})$~?
                Let $\tau= T \wedge S$ be the time of absorption on the boundary
                --the absorption at $(0,0)$ starting from $(n_0,m_0)\neq (0,0)$
                under hypothesis (H2') is impossible--.
                What is the distribution's tail of $\tau$
                starting from $(n_{0},m_{0})$ ?
     \item[(3)] What are the absorption probabilities
                $h_i^{n_{0},m_{0}}=\sum_{n=0}^{\infty} h_{i,n}^{n_{0},m_{0}}$
                and $\tilde{h}_j^{n_{0},m_{0}}=\sum_{n=0}^{\infty} \tilde{h}_{j,n}^{n_{0},m_{0}}$
                at points $(i,0)$ and $(0,j)$~?  What are their asymptotic
                as $i$ and $j$ go to infinity~?
                What is the probability of absorption
                $h(1,1)+\tilde{h}(1,1)=\sum_{i \geq 1}h_{i}^{n_{0},m_{0}}+
                \sum_{j \geq 1}\tilde{h}_{j}^{n_{0},m_{0}}$~?
     \item[(4)] What is the suitable Doob $h$-transform to condition the process
                never to touch the boundary~? Is this Doob $h$-transformed
                process equal in distribution to the limit as $n \to \infty$ of the conditional
                process given $\{\tau>n\}$~?
     \item[(5)] What are the asymptotic
                of the Green functions $G_{x,y}^{n_{0},m_{0}}$
                of the mean number of visits to $(x,y)$ starting from
                $(n_{0},m_{0})$ as $x,y \to \infty$, $y/x \to \tan(\gamma)$ where
                $\gamma \in [0, \pi/2]$~?
                What is the Martin boundary of this random walk~?
\end{itemize}

In Section~\ref{h_x_z} we find $h(x,z)$ and $\tilde{h}(y,z)$ under four
different forms, all of which having an own interest
and also an usefulness in the sequel~; what answers to Question~(1).

 The analysis of Question~(2) is performed
 in Section~\ref{h_1_z},
 using $h(x,z)$ and $\tilde{h}(y,z)$ with $x=1$ and $y=1$.
 The behavior of the process is of course notably different
 in cases $M_{x}>0,M_{y}>0$ and
 $M_{x}=M_{y}=0$.
 In first case the process is not absorbed
 with positive probability and
     \begin{equation}
          \label{tailt}
          \mathbb{P}_{\left(n_{0},m_{0}\right)}\left(\tau>n\right)
          \to 1-h\left(1,1\right)-\widetilde{h}\left(1,1\right),\
          \ n\to \infty.
     \end{equation}
 In second case
 the process is  almost surely absorbed and we will find that~:
     \begin{equation}
          \label{tailtt}
          \mathbb{P}_{\left(n_{0},m_{0}\right)}\left(\tau>n\right)  \sim
          \frac{n_{0}m_{0}}{\pi \sqrt{p_{10}p_{01}}}\frac{1}{n},\ \
          n\to \infty.
     \end{equation}

   Question~(3) is studied in Section~\ref{h_1_x} using
   $h(x,z)$ and $\tilde{h}(y,z)$ with $z=1$.
   In particular we get the probability of non absorption on the
   boundary that we denote by $A(n_0,m_0)$~:
          \begin{equation}
               \label{zzz}
               A\left(n_{0},m_{0}\right)=1-h\left(1,1\right)-\widetilde{h}\left(1,1\right)=
               \left(1-\left(\frac{p_{-10}}{p_{10}}\right)^{n_{0}}\right)
               \left(1-\left(\frac{p_{0-1}}{p_{01}}\right)^{m_{0}}\right) .
          \end{equation}

   We can then reply to Question~(4).
   If $M_{x}>0$ and $M_{y}>0$, the harmonic function in order to condition the
   process never to reach the boundary in Doob's sense is of course
   $A(n_{0},m_{0})$. If $M_{x}=M_{y}=0$, then the harmonic 
   function is $n_{0}m_{0}$. Moreover for
   all $x,y>0$ and $n>m>0$,
          \begin{equation*}
               \label{zlz}
               \mathbb{P}_{\left(n_{0},m_{0}\right)}
               \left(\left(X\left(m\right),Y\left(m\right)\right)=
               \left(x,y\right) \mid \tau>n\right)=
               \frac{\mathbb{P}_{\left(n_{0},m_{0}\right)}
               \left(\left(X\left(m\right),Y\left(m\right)\right)=
               \left(x,y\right)\right)
               \mathbb{P}_{\left(x,y\right)}\left(\tau>n-m\right)}
              {\mathbb{P}_{\left(n_{0},m_{0}\right)}\left(\tau>n\right)}.
          \end{equation*}
  If $M_{x}>0$ and $M_{y}>0$, this quantity converges as $n \to \infty$ to
  $\mathbb{P}_{n_{0},m_{0}}((X(m),Y(m))=(x,y))A(x,y)/A(n_{0},m_{0})$,
  thanks to~(\ref{tailt}) and~(\ref{zzz}).
  If $M_{x}=M_{y}=0$ it converges as $n \to \infty$ to
  $\mathbb{P}_{n_{0},m_{0}}((X(m),Y(m))=(x,y))x y/(n_{0}m_{0})$ by~(\ref{tailtt}).
  Consequently, the Doob $h$-transformed
  process is equal in distribution to the limit as $n \to \infty$
  of the process conditioned by $\{\tau>n\}$.

  The harmonic function $A(n_{0},m_{0})$ in case
  $M_{x}>0$, $M_{y}>0$ (resp. $n_{0}m_{0}$ in case $M_{x}=M_{y}=0$)
  provides us with a point of the
  Martin boundary. To complete the study of the Martin boundary, we
  should find the asymptotic of the Martin kernel
  along all different infinite paths of the random walk.
  We analyze for that the asymptotic of the Green functions.
  In Section~\ref{h_1_x} we find the asymptotic of 
  $h_{x}^{n_{0},m_{0}}$ and  of $
  \tilde{h}_{y}^{n_{0},m_{0}}$ as $x \to \infty$ and $y \to \infty$ and
  in Section~\ref{Martin_boundary} we compute the asymptotic of
  $G_{x,y}^{n_{0},m_{0}}$ as $x,y>0$, $y/x \to \tan( \gamma)$,
  where $\gamma$ is a given angle in $[0, \pi/2]$.
  In~\cite{KV}, using the approach of~\cite{MA1},
  it has already been done when $M_{x}>0$ and $M_{y}>0$ but
  in case of non-zero jump probabilities from the boundaries to the interior of
  $(\mathbb{Z}_{+})^2$ --so that the interior of the quadrant, the
  $x$-axis, the $y$-axis
  and $(0,0)$ are four domains of spatial homogeneity--~;
  also, in~\cite{KS}, this approach has been successfully applied
  to the analysis of JS-queues.
  In our case of an absorbing boundary, it can be done by exactly
  the same methods, and even easier,
  using the explicit representations of
  functions $h(x,1)$ and $\tilde{h}(y,1)$ obtained in Section~\ref{h_x_z}.
  In particular we will deduce that in case $M_{x}>0$ and $M_{y}>0$,
  all angles $\gamma \in [0, \pi/2]$ will correspond to
  different points of the Martin boundary, which will be
  therefore homeomorphic to the segment $[0, \pi/2]$.
  Note that for the angle $\gamma $ of the drift
  (i.e.\ $\tan( \gamma)=M_{y}/M_{x})$
  the asymptotic of the Martin kernel is proportional to $A(n_{0},m_{0})$.
  As for case $M_{x}=M_{y}=0$, we prove in
  Section~\ref{Martin_boundary} that
  $G_{x,y}^{n_0,m_0}\sim 4\sqrt{p_{10}p_{01}} n_{0} m_{0} x y/(\pi (p_{01}x^2+p_{10}y^2)^2)$
  for any angle $\gamma \in [0,
  \pi/2]$. In other words, the function $n_{0}m_{0}$ is in this case
  the unique point of the Martin boundary.

In a next work, we will answer Questions~(1)--(5)
without making hypothesis (H2') but only under (H2)
and supposing that
$M_{x}=\sum_{i,j}i p_{i j}>0$ and
$M_{y}=\sum_{i,j}j p_{i j}>0$.
Furthermore, motivated by
Biane's works on ``quantum random walks'', we
will also, in an other work, analyze these questions
for the walks with
transition probabilities drawn in Figure~\ref{Flatto_Lie},
that is to say random walks in the Weyl chambers
of $\mathfrak{sl}_{3}(\mathbb{C})$ and
$\mathfrak{sp}_{4}(\mathbb{C})$ verifying
$M_{x}=M_{y}=0$.
One of the main difference between these walks and those
studied here is the fact that
the underlying conformal gluing functions
become quite elaborated.

\medskip

\noindent{\bf Acknowledgments.}

This project would never have taken shape 
without the advice of Professor Bougerol.
I thank him very much for his encouragements and
strong commitment all along, 

I also would like to thank M.\ Defosseux and F.\ Chapon for 
the interesting discussions we had together 
concerning the topic of this article.

I am also very grateful to 
Professor Bidaut-V\'eron for having initiated me to the depth of complex analysis.

Finally, I must thank Professor Kurkova who introduced me to this field of 
research and also for her constant help and support during the elaboration of this project.

\section{Generating functions of absorption probabilities}\label{h_x_z}

\subsection{A functional equation}\label{Introduction}

We start here by establishing a functional equation that the
generating functions of the absorption probabilities verify. Let~:
     \begin{equation}
          G\left(x,y,z\right) = \sum_{i,j\geq 1,n\geq 0}
          \mathbb{P}_{\left(n_{0},m_{0}\right)}
          \left((X(n), Y(n))=\left(i,j\right)\right) x^{i-1} y^{j-1} z^{n}.
     \end{equation}
We write now the following functional equation~(\ref{functional_equation}),
on which all our study is based~:
     \begin{equation}\label{functional_equation}
          Q\left(x,y,z\right)G\left(x,y,z\right) =
          h\left(x,z\right)+\widetilde{h}\left( y ,z\right) -x^{n_{0}} y^{m_{0}},
     \end{equation}
where $h$ and $\tilde{h}$ are defined in~(\ref{def_generating_functions})
and $Q$ is the following polynomial, depending only on the walk's
transition probabilities~:
     \begin{equation}\label{def_Q}
          Q\left(x,y,z\right)=
          x y z\Big( p_{10}x+p_{-10}x^{-1}+p_{01}y+p_{0-1}y^{-1} -z^{-1}
          \Big).
     \end{equation}
The functions $h(x,z)$, $\tilde{h}(y,z)$ for $|x|,|y|, |z|\leq
1$, and $G(x,y,z)$ for $|x|, |y|<1$, $|z|\leq 1$, are unknown.
Equation~(\ref{functional_equation}) has a meaning in
(at least) $\left\{x,y\in \mathbb{C} : \left|x\right|<1,
\left|y\right|<1,\left|z\right|\leq 1 \right\}$.
Note that the proof of~(\ref{functional_equation}) comes from
writing that for $k,l,p\in (\mathbb{Z}_{+})^{2}$~:
               \begin{eqnarray*}
                    \lefteqn{\mathbb{P}_{\left(n_{0},m_{0}\right)}
                    \left(\left(X\left(p+1\right), Y\left(p+1\right)\right)=\left(k,l\right) \right)=
                    \sum_{i,j\geq 1} \mathbb{P}_{\left(n_{0},m_{0}\right)}
                    \left(\left(X\left(p\right), Y\left(p\right)\right)=\left(i,j\right) \right)
                    p_{\left(i,j\right),\left(k,l\right)}} \\&+&
                    \sum_{i\geq 1}\mathbb{P}_{\left(n_{0},m_{0}\right)}
                    \left(\left(X\left(p\right), Y\left(p\right)\right)=
                    \left(i,0\right) \right)
                    \delta_{\left(k,l\right)}^{\left(i,0\right)}+
                    \sum_{j\geq 1}\mathbb{P}_{\left(n_{0},m_{0}\right)}
                    \left(\left(X\left(p\right), Y\left(p\right)\right)=\left(0,j\right) \right)
                    \delta_{\left(k,l\right)}^{\left(0,j\right)},
               \end{eqnarray*}
          where $\delta_{(a,b)}^{(c,d)}=1$ if $a=c$ and $b=d$,
          otherwise $0$. It remains to multiply by
          $x^{k} y^{l} z^{p+1}$ and then to sum with respect to $k,l,p$.

\subsection{The algebraic curve $Q\left(x,y,z\right)=0$.}\label{The_algebraic_curve_Q}

The polynomial~(\ref{def_Q}) can be written alternatively~:
     \begin{equation*}
          Q\left(x,y,z\right) = a\left(x,z\right) y^{2}+ b\left(x,z\right) y
          + c\left(x,z\right) = \widetilde{a}\left(y,z\right) x^{2}+
          \widetilde{b}\left(y,z\right) x + \widetilde{c}\left(y,z\right),
     \end{equation*}
where
     \begin{equation*}
          \left.\begin{array}{cccccc}
               a\left(x,z\right) &=& z p_{01} x,&
               \widetilde{a}\left(y,z\right)&=&z p_{10} y, \\
               b\left(x,z\right) &=& z p_{10} x^2-x+z p_{-10}, &
               \widetilde{b}\left(y,z\right)&=&z p_{01} y^2-y+z p_{0-1}, \\
               c\left(x,z\right) &=& z p_{0-1} x, &
               \widetilde{c}\left(y,z\right) &=& z p_{-10} y.
          \end{array}\right.
     \end{equation*}

With these notations, building the algebraic function
$Y\left(x,z\right)$ defined by $Q\left(x,y,z\right)=0$ is tantamount
to the construction of the square root of the four degree polynomial
$d\left(x,z\right) = b\left(x,z\right)^2-4a\left(x,z\right)
c\left(x,z\right)$. Indeed, $Q\left(x,y,z\right)=0$ is equivalent to
$\left(b\left(x,z\right)+2a\left(x,z\right)y\right)^{2}=d\left(x,z\right)$.
As for any non zero polynomial, there are two branches of the square
root of $d$. Each determination leads to a well defined (i.e.\ single
valued) and meromorphic function on the complex plane $\mathbb{C}$
appropriately cut. As usual, these cuts are constructed using the
roots of $d$, called the branched points. In our case, we have an
explicit expression for these branched points, some properties of
which are collected in the following~:

     \begin{lem}\label{lemma_branched_points}
          Define $z_{1}=1/(2(p_{10}p_{-10})^{1/2}+2(p_{01}p_{0-1})^{1/2})$
          (and note that $z_{1}\geq 1$,
          with equality if and only if the two drifts are equal to zero).
          For $z\in [0,z_{1}]$, the four roots of $d\left(x,z\right)=0$
          are real and non negative.
          For $z\in ]0,z_{1}[$, these roots are mutually different and positive.
          We call them in such a way that for $z\in ]0,z_{1}[$,
          $0<x_{1}(z)<x_{2}(z)<x_{3}(z)<x_{4}(z)$.
          Their explicit expressions are~:
               \begin{eqnarray*}
                    x_{2,3}\left(z\right)&=&\frac{z^{-1}-2\sqrt{p_{01}p_{0-1}}}{2p_{10}}\pm
                    \sqrt{\left(\frac{z^{-1}-2\sqrt{p_{01}p_{0-1}}}{2p_{10}}\right)^{2}-\frac{p_{-10}}{p_{10}}},\\
                    x_{1,4}\left(z\right)&=&\frac{z^{-1}+2\sqrt{p_{01}p_{0-1}}}{2p_{10}}\pm
                    \sqrt{\left(\frac{z^{-1}+2\sqrt{p_{01}p_{0-1}}}{2p_{10}}\right)^{2}-\frac{p_{-10}}{p_{10}}}.
               \end{eqnarray*}
          They are tied together by $x_{1}(z)x_{4}(z)=
          x_{2}(z)x_{3}(z)=p_{-10}/p_{10}$
          and verify $x_{1}(z),x_{2}(z)\in ]0,(p_{-10}/p_{10})^{1/2}[$,
          $x_{3}(z),x_{4}(z)\in ](p_{-10}/p_{10})^{1/2},+\infty [$
          for $z\in ]0,z_{1}[$. Also, $x_{1}(0)=x_{2}(0)=0$
          and $x_{3}(0)=x_{4}(0)=\infty $. Moreover, if $M_{y}>0$ then
          $0<x_{1}(1)<x_{2}(1)<1< x_{3}(1)<x_{4}(1)$
          and if $M_{y}=0$ then $0<x_{1}(1)<x_{2}(1)=1=x_{3}(1)<x_{4}(1)$.
          At last,
          the $x_{i}$ vary continuously and monotonously with respect to $z$.
     \end{lem}

\begin{proof}
All the facts described in Lemma~\ref{lemma_branched_points} are
based on the explicit expression of the branched points  that we get
by solving $d(x,z)=0$. Here $d$ is a four degree polynomial that we
can split in two polynomials of degree two~:
$d(x,z)=(b(x,z)-2z x (p_{01}p_{0-1})^{1/2})
(b(x,z)+2z x(p_{01}p_{0-1})^{1/2})$, since $a$ and $c$ are
proportional.
\end{proof}

This precise knowledge of the branched points allows us to complete
the construction of the algebraic function $Y$~: this function has
two branches, each of them being well defined and meromorphic on
$\mathbb{C}\setminus
[x_{1}(z),x_{2}(z)] \cup
[x_{3}(z),x_{4}(z)]$. We can write
the analytic expression of these two branches $Y_{0}$ and $Y_{1}$ of
$Y$~: $Y_{0}\left(x,z\right)=Y_{-}\left(x,z\right)$ and
$Y_{1}\left(x,z\right)=Y_{+}\left(x,z\right)$ where~:
     \begin{equation*}
          Y_{\pm }\left(x,z\right) = \frac{ -b\left(x,z\right)\pm
          \sqrt{d\left(x,z\right)}}{2 a\left(x,z\right) }.
     \end{equation*}

Note that just above, and in fact throughout
the whole paper, we chose the principal determination of
the logarithm as soon as we use the complex logarithm~;
in this case to define the square root.

For more details about the
construction of algebraic functions, see for instance 
Book~\cite{SG2}.

In a similar way, the functional
equation~(\ref{functional_equation}) defines also an algebraic
function $X\left(y,z\right)$. All the results concerning
$X\left(y,z\right)$ can be deduced from those obtained for
$Y\left(x,z\right)$ after a proper change of the parameters, namely
$p_{-10}$ (resp.\ $p_{10},p_{0-1},p_{01}$) in $p_{0-1}$ 
(resp.\ $p_{01},p_{-10},p_{10}$).

To conclude this part, we give a lemma that clarifies some
properties of the functions $X$ and $Y$, that will be useful in the
sequel. This lemma is an adaptation of results of~\cite{FIM}, so we
refer to this book for the proof.

\begin{lem}\label{properties_X_Y}
     The two following equalities hold~: $Y_{1}(1,1)=1$ and $Y_{0}(1,1)=p_{0-1}/p_{01}$.
     Moreover, the next properties are valid for all $z\in ]0,z_{1}]$.

     {\rm (i)}  For all $x\in \mathbb{C}\setminus
                [x_{1}(z),x_{2}(z)] \cup
                [x_{3}(z),x_{4}(z)]$, $Y_{0}(x,z)Y_{1}(x,z)=p_{0-1}/p_{01}$,
                $|Y_{0}(x,z)|\leq (p_{0-1}/p_{01})^{1/2}\leq |Y_{1}(x,z)|$.
       
     {\rm (ii)} If $|x|<(p_{-10}/p_{10})^{1/2}$, then
                $X_{0}(Y_{0}(x,z),z)=X_{0}(Y_{1}(x,z),z)=x$ and
                $X_{1}(Y_{0}(x,z),z)=X_{1}(Y_{1}(x,z),z)=p_{-10}/(p_{10} x)$.
       
     {\rm (iii)} If $|x|>(p_{-10}/p_{10})^{1/2}$, then
                 $X_{0}(Y_{0}(x,z),z)=X_{0}(Y_{1}(x,z),z)=p_{-10}/(p_{10} x)$ and
                 $X_{1}(Y_{0}(x,z),z)=X_{1}(Y_{1}(x,z),z)=x$.

%

     {\rm  (iv)} As $x\to \infty $, $Y_{0}\left(x,z\right)=-p_{0-1}/(p_{10}x)-p_{0-1}/(p_{10}x)^{2}
                 +\mathcal{O}(1/x^{3})$ and $Y_{1}\left(x,z\right)=
                 -p_{10}x/p_{01}+1/p_{01}+\mathcal{O}(1/x)$.
     \end{lem}


\subsection{Riemann boundary problem and conformal gluing functions}
\label{Riemann_Carleman_problem}

Throughout all Subsection~\ref{Riemann_Carleman_problem}, 
$z$ lies in $]0,z_{1}]$, unless otherwise specified.
Using the notations of Subsections~\ref{Introduction}
and~\ref{The_algebraic_curve_Q}, we define the two following curves~:
     \begin{equation}\label{def_curves_L_M}
          \mathcal{L}_{z} = Y_{0}\left(\left[\overrightarrow{\underleftarrow{x_{1}\left(z\right)
          ,x_{2}\left(z\right)}}\right],z\right),\hspace{5mm} \mathcal{M}_{z} =
          X_{0}\left(\left[\overrightarrow{\underleftarrow{y_{1}\left(z\right)
          ,y_{2}\left(z\right)}}\right],z\right).
     \end{equation}
Just above, we use the notation
$[\overrightarrow{\underleftarrow{u,v}}]$
for the contour
$[u,v]$ traversed from $u$ to $v$ along the upper edge of
the slit $[u,v]$ and then back to $u$ along the lower
edge of the slit.

A worthwhile sight is that under the hypothesis (H2'),
these two curves are quite
simple since they are in fact just two circles, centered at the
origin and of radius $\tilde{r}=(p_{0-1}/p_{01})^{1/2}
\leq 1$ and $r=(p_{-10}/p_{10})^{1/2} \leq 1$
respectively.
One verifies these facts directly~: if $t\in [x_{1}\left(z\right),
x_{2}\left(z\right)]$, then $d\left(t,z\right)\in \mathbb{R}_{-}$
and so $|-b(t,z)\pm d(t,z)^{1/2}|^{2}=4a(t,z)c(t,z)$. Thus,
$|Y_{0,1}\left(t,z\right)|^2=c(t,z)/a(t,z)=p_{0-1}/p_{01}$
and $\mathcal{L}_{z}=\mathcal{C}(0,\tilde{r})$~;
likewise, we prove that $\mathcal{M}_{z}=\mathcal{C}(0,r)$.

The reason why we have introduced these curves appears now~:
the functions $h$ (of the argument $x$) and $\tilde{h}$ (of the
argument $y$), defined in~(\ref{def_generating_functions}), verify
the following boundary conditions on
$\mathcal{M}_{z}=\mathcal{C}\left(0,r\right)$ and
$\mathcal{L}_{z}=\mathcal{C}\left(0,\tilde{r}\right)$~:
     \begin{eqnarray}
          \forall t \in \mathcal{C}\left(0,r\right) :\hspace{5mm}
          h\left( t ,z\right) - h\left( \overline{t},z \right)&=&
          t^{n_{0}}Y_{0}\left(t,z\right)^{m_{0}}-\overline{t}^{n_{0}}
          Y_{0}\left(\overline{t},z\right)^{m_{0}},\label{SR_problem_h}\\
          \forall t \in \mathcal{C}\left(0,\widetilde{r}\right):\hspace{5mm}
          \widetilde{h}\left( t ,z\right) - \widetilde{h}\left( \overline{t} ,z\right)&=&
          X_{0}\left(t,z\right)^{n_{0}}t^{m_{0}}-
          X_{0}\left(\overline{t},z\right)^{n_{0}}\overline{t}^{m_{0}}.\nonumber
     \end{eqnarray}

The way to obtain~(\ref{SR_problem_h}) and the analogue for
$\tilde{h}$ is described in~\cite{FIM}, so we refer to this book for
the details~; nevertheless, we recall here briefly the
explanations~: taking $|y|\leq 1$ and
$x=X_{0}(y,z)$ (whose modulus is less than one thanks to
Lemma~\ref{properties_X_Y}) in~(\ref{functional_equation}) leads to~:
     \begin{equation*}
          h\left(X_{0}\left(y,z\right),z\right)
          +\widetilde{h}\left(y,z\right)-
          X_{0}\left(y,z\right)^{n_{0}}y^{m_{0}}=0.
     \end{equation*}

We let now $y$ go successively to the upper and lower edge of
$[y_{1}(z),y_{2}(z)]$ and we make
the difference of these two equations so obtained. Since the 
slit $[y_{1}(z),y_{2}(z)]$ is included in
the unit disc where $\tilde{h}$ is holomorphic, $\tilde{h}$ vanishes
and we find that for $y\in [y_{1}(z),y_{2}(z)]$,
     \begin{equation*}
          h\left(X_{0}\left(y,z\right),z\right)-
          h\left(X_{1}\left(y,z\right),z\right)=
          X_{0}\left(y,z\right)^{n_{0}}y^{m_{0}}-
          X_{1}\left(y,z\right)^{n_{0}}y^{m_{0}}.
     \end{equation*}
According to Lemma~\ref{properties_X_Y}, for any $y\in
[y_{1}(z),y_{2}(z)]$ we have $Y_{0}(X_{0}(y,z),z)=y$~; so we obtain
that for any $t\in \mathcal{M}_{z}$,
$h(t,z)-h(\overline{t},z)=Y_{0}\left(t,z\right)^{m_{0}}(t^{n_{0}}-
\overline{t}^{n_{0}})$. To complete the proof of~(\ref{SR_problem_h}),
it remains to show that $Y_{0}(t,z)=Y_{0}(\overline{t},z)$ for $t\in
\mathcal{M}_{z}$~; but this is once again a consequence of
Lemma~\ref{properties_X_Y} since we proved there that for $y\in
[y_{1}(z),y_{2}(z)]$, $Y_{0}(X_{1}(y,z),z)=y$.

For any $z \in [0,1]$, the function $h$ of the argument $x$, as a
generating function of probabilities,
is well defined on the closed unit disc  $|x|\leq 1$, holomorphic
inside it and continuous up to its boundary.
With Lemma~\ref{properties_X_Y}, the curve $\mathcal{M}_{z}$
is included in the closed unit disc. Now we have the problem
{\it to find $h$ holomorphic inside $\mathcal{M}_{z}$,
continuous up to the boundary and verifying
the boundary condition~(\ref{SR_problem_h}).
In addition $h(0,z)=0$ for all $z \in [0,1]$.}

Problems with boundary conditions like~(\ref{SR_problem_h}) are
called Riemann boundary value problems with shift. The classical way
to study this kind of problems is to reduce them to Riemann-Hilbert
problems, for which there exists a suitable and complete theory. The
conversion between Riemann problems with shift and Riemann-Hilbert
problems is done thanks to the use of conformal gluing functions,
notion defined just below. For details about boundary value
problems, we refer to~\cite{LU}.

     \begin{defn}\label{def_CGF}
     Let $\mathcal{C}$ be a simple closed curve, symmetrical
     with respect to the real axis.
     Denote by $\mathcal{G}_{\mathcal{C}}$
     the interior of the bounded domain delimited by $\mathcal{C}$.
     $w$ is called a conformal gluing function (CGF) for
     the curve $\mathcal{C}$ if {\rm (i)} $w$ is meromorphic in
     $\mathcal{G}_{\mathcal{C}}$, continuous up
     to its boundary {\rm (ii)} $w$
     establishes a conformal mapping of
     $\mathcal{G}_{\mathcal{C}}$ onto the complex plane
     cut along a smooth arc $U$ {\rm (iii)}
     for all $t\in \mathcal{C}$,
     $w\left(t\right)=w( \overline{t})$.
     \end{defn}

In the general case, finding a CGF associated to some curve
without strong hypothesis on this curve is a quite difficult problem,
we will besides discuss this fact in Section~\ref{Extension_results}~;
however, in our case, the curves $\mathcal{M}_{z}$
and $\mathcal{L}_{z}$ are circles and
we have therefore an explicit expression of possible CGF~:

     \begin{prop}\label{explicit_form_CGF}
     CGF for the curves $\mathcal{M}_{z}=\mathcal{C}\left(0,r\right)$ and
     $\mathcal{L}_{z}=\mathcal{C}\left(0,\tilde{r}\right)$ are equal to~:
          \begin{equation*}
               w\left(t\right)=\frac{1}{2}\left(t+\frac{r^{2}}{t}\right),
               \hspace{5mm}  \widetilde{w}\left(t\right)=
               \frac{1}{2}\left(t+\frac{\widetilde{r}^{2}}{t}\right).
          \end{equation*}
     \end{prop}

\begin{proof}
We verify easily that $w$ and $\tilde{w}$ are indeed CGF,
following the different \textit{item} of Definition~\ref{def_CGF}. First,
$w$ is manifestly holomorphic on $\mathbb{C}\setminus \{0\}$ and
has a simple pole at $0$, that proves {\rm (i)}. Moreover, $w$ is a
conformal mapping from $\mathcal{D}(0,r)$ onto $\mathbb{C}\setminus
U$, where $U$ is the segment $[-r,r]$, hence {\rm (ii)}. At last,~{\rm (iii)}
comes from remarking that $w(re^{i \theta
})=r\cos(\theta)=w(re^{-i \theta })$. Of course, the proof is
similar for $\tilde{w}$, so we omit it.
\end{proof}

\subsection{Integral representations of the generating functions}
\label{Integral_representation}

In Subsection~\ref{Riemann_Carleman_problem}, we have showed that
$h$ verifies a Riemann problem with shift with boundary
condition~(\ref{SR_problem_h})~; we will now see how
we can deduce from this an explicit expression for $h$.
In fact, we will obtain four different explicit formulations for
$h$, in Propositions~\ref{explicit_h(x,z)},~\ref{explicit_h(x,z)_second}
,~\ref{explicit_h(x,z)_third} and~\ref{lemma_final_form}~; each of
these expressions will have an own interest and will serve in the
sequel --for instance Proposition~\ref{explicit_h(x,z)_third} will be
the starting point of Section~\ref{h_1_x},
Proposition~\ref{lemma_final_form} the one of Section~\ref{h_1_z}--.







     \begin{prop}\label{explicit_h(x,z)}
     Let $\mathcal{C}(0,r)$ be the circle of the radius
     $r=(p_{-10}/p_{10})^{1/2}$.
     The function $h$ admits the following integral expression~:
          \begin{equation*}
               h\left(x,z\right) = \frac{x}{2\pi i}
               \int_{\mathcal{C}\left(0,r\right)}
               t^{n_{0}}Y_{0}\left(t,z\right)^{m_{0}}
               \left( \frac{1}{t\left(t-x\right)}+
               \frac{1}{x t-r^{2}}\right) \textnormal{d}t.
          \end{equation*}
     Above, $x$ belongs to the open centered disc of radius $r$,
     and $z$ to $]0,z_{1}]$,
     $z_{1}$ being defined in Lemma~\ref{lemma_branched_points}.
     \end{prop}

\begin{proof}

This proposition corresponds to the standard way to obtain a 
Riemann-Hilbert problem starting from a Riemann problem with shift.

We recall
from Definition~\ref{def_CGF} and
Proposition~\ref{explicit_form_CGF} that $w$ is a conformal
mapping from
the open disc $D(0,r)$
onto $\mathbb{C}\setminus U$, where $U$ is the segment
$[w(X(y_{1}(z),z)), w(X(y_{2}(z),z))]=[-r,r]$.
Therefore, the function $w$ admits an inverse from
$\mathbb{C}\setminus [-r,r]$ onto $D(0,r)$,
inverse that we call $v$.
We can here give the explicit expression of $v$~:
it is equal to $v(w)=w-(w^{2}-r^{2})^{1/2}$.

If we denote by $v^{+}(w)$ (resp.\ $v^{-}(w)$)
the limit value of $v(y)$ when
$y\to w$ from the upper half plane $\{s\in \mathbb{C} :
\text{Im} (s)>0 \}$ (resp.\ lower half plane $\{s\in \mathbb{C} :
\text{Im}(s)<0 \}$), then
$v^{+}(U)=\mathcal{C}(0,r)\cap \{t\in \mathbb{C} :
\text{Im} (t)<0\}$ and
$v^{-}(U)=\mathcal{C}(0,r)\cap \{t\in \mathbb{C} :
\text{Im} (t)>0\}$.
%
%
%
%
This is why we can,
thanks to the function $v$, rewrite the boundary
condition~(\ref{SR_problem_h}) in terms of $\phi=h\circ v$ as follows~:
     \begin{equation}\label{RH_problem_h}
          \phi^{+}\left(w\right)-\phi^{-}\left(w\right)=
          v^{+}\left(w\right)^{n_{0}}Y_{0}
          \left(v^{+}\left(w\right),z\right)^{m_{0}}-
          v^{-}\left(w\right)^{n_{0}}Y_{0}
          \left(v^{-}\left(w\right),z\right)^{m_{0}},
          \;\;\;\;w\in U,
     \end{equation}
the advantage of this new formulation being that we have now to
solve a more classical Riemann-Hilbert problem. The properties of
$h$ (as a generating function of probabilities) and $v$ are such
that $\phi$ has to be sought among the functions 
holomorphic on $\mathbb{C}\setminus U$ with a finite
limit at infinity and bounded near the ends of $U$. Thus, the 
index (see e.g.\ \cite{LU})
of this Riemann-Hilbert problem is equal to zero, what in concrete
terms means that two solutions of the boundary value
problem with boundary condition~(\ref{RH_problem_h}), or
equivalently~(\ref{SR_problem_h}), differ by a constant~; the
constant will be fixed by using the fact that $h\left(0,z\right)=0$.

Using the theory of Riemann-Hilbert
problems developed for instance in~\cite{LU}, we obtain that $h$ is
equal, up to an additive constant, to~:
     \begin{equation*}
          \frac{1}{2\pi i} \int_{U}
          \left(v^{+}\left(w\right)^{n_{0}}Y_{0}
          \left(v^{+}\left(w\right),z\right)^{m_{0}}-
          v^{-}\left(w\right)^{n_{0}}Y_{0}
          \left(v^{-}\left(w\right),z\right)^{m_{0}}\right)
          \frac{1}{w-w\left(x\right)}
          \text{d}w.
     \end{equation*}
So that, taking account
of the equality $h(0,z)=0$, we obtain that $h$ is equal to~:
\begin{equation}\label{solutions_RH_problem_h}
          \frac{1}{2\pi i} \int_{U}
          \left(v^{+}\left(w\right)^{n_{0}}Y_{0}
          \left(v^{+}\left(w\right),z\right)^{m_{0}}-
          v^{-}\left(w\right)^{n_{0}}Y_{0}
          \left(v^{-}\left(w\right),z\right)^{m_{0}}\right)
          \left(\frac{1}{w-w\left(x\right)}-\frac{1}{w-w\left(0\right)}\right)
          \text{d}w.
     \end{equation}
Then, we take the notation
$\phi(t,x)=w'(t)/(w(t)-w(x))
-w'(t)/(w(t)-w(0))$ and we
make the change of variable
$w=w\left(t\right)$ in~(\ref{solutions_RH_problem_h})~:
     \begin{eqnarray*}
          &&\frac{1}{2\pi i} \int_{U}
          \left(v^{+}\left(w\right)^{n_{0}}Y_{0}\left(v^{+}
          \left(w\right),z\right)^{m_{0}}-
          v^{-}\left(w\right)^{n_{0}}Y_{0}
          \left(v^{-}\left(w\right),z\right)^{m_{0}}\right)
          \left(\frac{1}{w-w\left(x\right)}-
          \frac{1}{w-w\left(0\right)}\right)\text{d}w\\&=&\frac{1}{2\pi i}
          \int_{v^{+}\left(U\right)}t^{n_{0}}
          Y_{0}\left(t,z\right)^{m_{0}}\phi\left(t,x\right)
          \text{d}t-
          \left(- \frac{1}{2\pi i}\int_{v^{-}\left(U\right)}
          t^{n_{0}}Y_{0}\left(t,z\right)^{m_{0}}
          \phi\left(t,x\right)\text{d}t\right)\\&=&\frac{1}{2\pi i}
          \int_{\mathcal{C}(0,r)} t^{n_{0}}
          Y_{0}\left(t,z\right)^{m_{0}}\phi\left(t,x\right)\text{d}t,
     \end{eqnarray*}
since $v^{+}(U)\cup v^{-}(U)=\mathcal{C}(0,r)$, as written at 
the beginning of the proof.
To close the proof of Proposition~\ref{explicit_h(x,z)}, it suffices
to write the partial fraction expansion of $\phi$, namely
$\phi(t,x)=x/(t(t-x))+x/(t x-r^{2})$.
\end{proof}

We transform now the integral on $\mathcal{M}_{z}=\mathcal{C}(0,r)$ obtained in
Proposition~\ref{explicit_h(x,z)} into an integral on the cut
$\left[x_{1}\left(z\right),x_{2}\left(z\right)\right]$. We start by
giving the definition~:
     \begin{equation}\label{def_mu}
          \mu_{m_{0}}\left(t,z\right) = \frac{1}{\left(2 a\left(t,z\right)\right)^{m_{0}}}
          \sum_{k=0}^{\left\lfloor \left(m_{0}-1\right)/2
          \right\rfloor}{\genfrac(){0cm}{0}{m_{0}}{2k+1}  }
          d\left(t,z\right)^{k} \left( - b\left(t,z\right) \right)^{m_{0}-\left(2k+1\right)}.
     \end{equation}
The function $\mu_{m_{0}}$ is such that for all $t$ in $
\left[x_{1}\left(z\right),x_{2}\left(z\right)\right]\pm 0 \cdot i$,
${Y_{0}(t,z)}^{m_{0}}-\overline{Y_{0}(t,z)}^{m_{0}}= \mp 2
i (-d(t,z))^{1/2} \mu_{m_{0}}(t,z)$.

An application of residue theorem in
Proposition~\ref{explicit_h(x,z)}
and the use of the definition~(\ref{def_mu})
of $\mu_{m_{0}}$ allow
to obtain~:

     \begin{prop}\label{explicit_h(x,z)_second}
     The function $h$ admits the following integral expression~:
          \begin{equation}\label{new_integral_form_h}
               h\left(x,z\right) = x^{n_{0}}Y_{0}\left(x,z\right)^{m_{0}}
               +\frac{x}{\pi }\int_{x_{1}\left(z\right)}^{x_{2}\left(z\right)}
               t^{n_{0}}\left(\frac{1}{t\left(t-x\right)}+\frac{1}{x t-r^{2}}\right)
               \mu_{m_{0}}\left(t,z\right)\sqrt{-d\left(t,z\right)}\textnormal{d}t.
          \end{equation}
     Above, $x$ belongs to the open centered disc of radius $r$,
     and $z\in ]0,z_{1}]$,
     $z_{1}$ being defined in Lemma~\ref{lemma_branched_points}.
     \end{prop}

\begin{proof}

We take here, as in the proof of
Proposition~\ref{explicit_h(x,z)}, the notation $\phi(t,x)=
x/(t(t-x))+x/(t x-r^{2})$.
Consider the contour $\mathcal{H}_{\epsilon }= \mathcal{M}_{\epsilon}
\cup \mathcal{S}^{1}_{\epsilon } \cup \mathcal{S}^{2}_{\epsilon }
\cup \mathcal{C}^{1}_{\epsilon } \cup \mathcal{C}^{2}_{\epsilon }
\cup \mathcal{D}^{1}_{\epsilon } \cup \mathcal{D}^{2}_{\epsilon }$,
drawn in Figure~\ref{pacman}.
\begin{figure}[!h]
\begin{center}
\begin{picture}(200.00,180.00)
\includegraphics{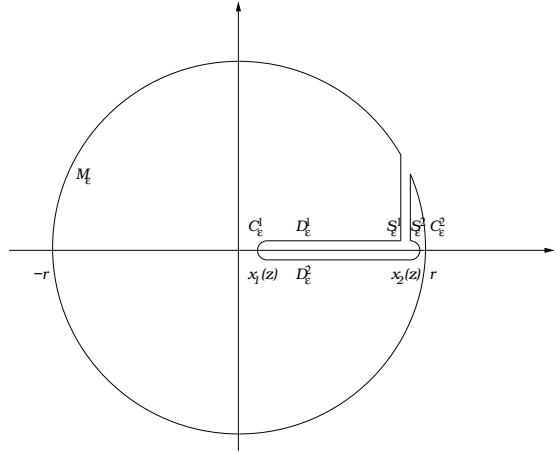}
\end{picture}
\end{center}
\caption{Contour of integration} \label{pacman}
\end{figure}
The following facts hold~:

     \begin{enumerate}

     \item $\int_{\mathcal{C}(0,r)}t^{n_{0}}Y_{0}(t,z)^{m_{0}}
     \phi(t,x)\text{d}t
     =\lim_{\epsilon \to 0}\int_{\mathcal{M}_{\epsilon }}
     t^{n_{0}}Y_{0}(t,z)^{m_{0}}
     \phi(t,x)\text{d}t$, thanks to the continuity of the integrand
     on the circle $\mathcal{C}(0,r)$,

     \item the residue theorem gives that for all $\epsilon >0$ sufficiently
     small and all
     $x\in D(0,r)\setminus [x_{1}(z),x_{2}(z)]$~:
     $\int_{\mathcal{H_{\epsilon }}}(t^{n_{0}-1}Y_{0}(t,z)^{m_{0}})/(t-x)
     \text{d}t
     =2\pi ix^{n_{0}-1}Y_{0}(x,z)^{m_{0}}$,

     \item the integral on
     $\mathcal{S}^{1}_{\epsilon }\cup \mathcal{S}^{2}_{\epsilon }$ 
     goes to zero as $\epsilon $ goes to zero.
     Indeed, the integrand is holomorphic in the neighborhood
     of $\mathcal{S}^{1}_{\epsilon }\cup \mathcal{S}^{2}_{\epsilon }$ and for this reason,
     $\lim_{\epsilon \to 0}\int_{\mathcal{S}^{1}_{\epsilon }}
     t^{n_{0}}Y_{0}(t,z)^{m_{0}}
     \phi(t,x)\text{d}t   =  -
     \lim_{\epsilon \to 0}\int_{\mathcal{S}^{2}_{\epsilon }}
     t^{n_{0}}Y_{0}(t,z)^{m_{0}}
     \phi(t,x)\text{d}t$. Also,
     for $k=1,2$, $\lim_{\epsilon \rightarrow 0}\int_{\mathcal{C}^{k}_{\epsilon }}
     t^{n_{0}}Y_{0}(t,z)^{m_{0}}
     \phi(t,x)\text{d}t=0$ since
     the integrand is integrable in the neighborhood of the
     branched points $x_{1}(z)$ and $x_{2}(z)$.

     \item $\lim_{\epsilon \to 0}
     \int_{\mathcal{D}^{1}_{\epsilon }\cup \mathcal{D}^{2}_{\epsilon }}
     t^{n_{0}}Y_{0}(t,z)^{m_{0}}
     \phi(t,x)\text{d}t =
     \int_{x_{1}(z)}^{x_{2}(z)}
     t^{n_{0}}( Y_{0}(t,z)^{m_{0}}-
     \overline{Y_{0}(t,z)}^{m_{0}})
     \phi(t,x)\text{d}t$, thanks to
     the algebricity of the function $Y_{0}$.

     \end{enumerate}
If we bring together all these facts, we obtain the equality~:
     \begin{equation*}
          \int_{\mathcal{C}(0,r)}
          t^{n_{0}}Y_{0}\left(t,z\right)^{m_{0}}
          \phi\left(t,x\right)\text{d}t
          = x^{n_{0}}Y_{0}\left(x,z\right)^{m_{0}}
          -\frac{1}{2\pi i}
          \int_{x_{1}\left(z\right)}^{x_{2}\left(z\right)}
          t^{n_{0}}\left( Y_{0}\left(t,z\right)^{m_{0}}-
          \overline{Y_{0}\left(t,z\right)}^{m_{0}}\right)
          \phi\left(t,x\right)\text{d}t,
     \end{equation*}
from which Proposition~\ref{explicit_h(x,z)_second} follows immediately,
using the definitions of $\mu_{m_{0}}$ and $\phi$.
\end{proof}

We carry on with the simplifications of the explicit expression of
the function $h$. The formulation~(\ref{new_integral_form_h}) is
nearly satisfactory but has yet a defect~: $h$ is a function
holomorphic in the neighborhood of $[x_{1}(z),x_{2}(z)]$, but is
written in~(\ref{new_integral_form_h}) as the sum of two functions
which are not holomorphic but algebraic in the neighborhood of
$[x_{1}(z),x_{2}(z)]$. The next lemma overcomes this fact~:

\begin{lem}\label{simplification_principal_part}
For $x\in \mathbb{C}\setminus [x_{1}(z),x_{2}(z)]\cup
[x_{3}(z),x_{4}(z)]$ and $z\in ]0,z_{1}]$, the following equality holds~:
     \begin{eqnarray*}
          x^{n_{0}}Y_{0}\left(x,z\right)^{m_{0}}+\frac{x}{\pi }
          \int_{x_{1}\left(z\right)}^{x_{2}\left(z\right)}
          \frac{t^{n_{0}-1}\mu_{m_{0}}\left(t,z\right)}
          {t-x}\sqrt{-d\left(t,z\right)}\textnormal{d}t&=&
          \frac{x}{\pi}\int_{x_{3}\left(z\right)}^{x_{4}\left(z\right)}
          \frac{t^{n_{0}-1}\mu_{m_{0}}\left(t,z\right)
          \sqrt{-d\left(t,z\right)}}{t-x}\textnormal{d}t\\&+&xP_{\infty }
          \left(x\mapsto x^{n_{0}-1}Y_{0}\left(x,z\right)^{m_{0}}\right)\left(x\right).
     \end{eqnarray*}
Above, $P_{\infty }(x\mapsto x^{n_{0}-1}Y_{0}(x,z)^{m_{0}})$
denotes the principal part at infinity
of the meromorphic function at infinity
$x\mapsto x^{n_{0}-1}Y_{0}(x,z)^{m_{0}}$~;
in other words, the polynomial
part of the Laurent expansion at infinity of this function.
In particular, $(x,z)\mapsto xP_{\infty }(x\mapsto x^{n_{0}-1}Y_{0}(x,z)^{m_{0}})(x)$ is a
polynomial in the two variables $(x,z)$. For more comments about
this quantity, see Remark~\ref{rem_principal_part}.
\end{lem}

\begin{proof}
Consider the contour $\mathcal{C}_{\epsilon }$ drawn in
Figure~\ref{contour_simplification}
     \begin{figure}[!h]
     \begin{center}
     \begin{picture}(300.00,50.00)
     \includegraphics{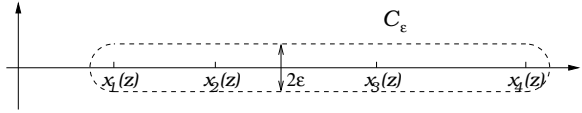}
     \end{picture}
     \end{center}
     \caption{Contour of residue theorem}
     \label{contour_simplification}
     \end{figure}
and apply on it the residue theorem at infinity (a precise statement of
this theorem
can be found e.g.\ in~\cite{Chat}) to the function $t\mapsto
t^{n_{0}-1}Y_{0} (t,z)^{m_{0}}$~; we obtain that  for all $x$ 
in the unbounded domain delimited by $\mathcal{C}_{\epsilon}$,
     \begin{equation}\label{residue_infinity}
          \frac{x}{2\pi i} \int_{\mathcal{C}_{\epsilon}}\frac{t^{n_{0}-1}Y_{0}
          \left(t,z\right)^{m_{0}}}{t-x}\text{d}t=x^{n_{0}}
          Y_{0}\left(x,z\right)^{m_{0}}-xP_{\infty }
          \left(x\mapsto x^{n_{0}-1}Y_{0}\left(x,z\right)^{m_{0}}\right)\left(x\right),
     \end{equation}
where, if $f$ is a function meromorphic at infinity,
$P_{\infty}\left(f\right)$ denotes its principal part at infinity. On the
other hand, using the algebricity of $Y_{0}$ and the definition of
$\mu_{m_{0}}$, we get~:
     \begin{eqnarray*}
          \lim_{\epsilon \to 0}\frac{x}{2\pi i}
          \int_{\mathcal{C}_{\epsilon}}\frac{t^{n_{0}-1}Y_{0}\left(t,z\right)^{m_{0}}}
          {t-x}\text{d}t = &-&\frac{x}{\pi}\int_{x_{1}\left(z\right)}^{x_{2}\left(z\right)}
          \frac{t^{n_{0}-1}\mu_{m_{0}}\left(t,z\right)\sqrt{-d\left(t,z\right)}}{t-x}\text{d}t
          \\&+&\frac{x}{\pi}\int_{x_{3}\left(z\right)}^{x_{4}\left(z\right)}
          \frac{t^{n_{0}-1}\mu_{m_{0}}\left(t,z\right)\sqrt{-d\left(t,z\right)}}{t-x}\text{d}t.
     \end{eqnarray*}
We conclude by taking the limit when $\epsilon$ goes to zero
in~(\ref{residue_infinity}).
\end{proof}

\begin{rem}\label{rem_principal_part}
{\rm Lemma~\ref{properties_X_Y} gives that
$x^{n_{0}-1}Y_{0}\left(x,z\right)^{m_{0}}\sim c x^{n_{0}-m_{0}-1}$
as $x\to \infty$,
where $c$ is a non zero constant.
Thus, if $n_0\leq m_0$, then the principal part is equal to zero, whereas if
$n_{0}>m_{0}$, then the degree of the principal part is
$n_{0}-m_{0}-1$ and the degree of $xP_{\infty }(x\mapsto
x^{n_{0}-1}Y_{0}\left(x,z\right)^{m_{0}})\left(x\right)$ equal to
$n_{0}-m_{0}$. For briefness, we set $P(x,z)=xP_{\infty}(x\mapsto
x^{n_{0}-1}Y_{0}\left(x,z\right)^{m_{0}})\left(x\right)$.}
\end{rem}
The Lemma~\ref{simplification_principal_part} allows to write, in
accordance with~(\ref{new_integral_form_h}),
     \begin{equation*}
               h\left(x,z\right)=\frac{x}{\pi }\left( \int_{x_{3}\left(z\right)}^{x_{4}\left(z\right)}
               \frac{t^{n_{0}}\mu_{m_{0}}\left(t,z\right)\sqrt{-d\left(t,z\right)}}{t\left(t-x\right)}
               \text{d}t-\int_{x_{1}\left(z\right)}^{x_{2}\left(z\right)} \frac{t^{n_{0}}\mu_{m_{0}}
               \left(t,z\right)\sqrt{-d\left(t,z\right)}}{r^{2}-t x}\text{d}t\right)+P(x,z),
     \end{equation*}
where the polynomial $P(x,z)$ is defined in
Remark~\ref{rem_principal_part}. Instead of the two integrals above,
we can write just one, making a simple change of variable based on the
two properties described
below~:
     \begin{enumerate}

          \item In conformity with Lemma~\ref{lemma_branched_points} ,
          the branched points verify
          $x_{2}\left(z\right)x_{3}\left(z\right)=x_{1}\left(z\right)x_{4}\left(z\right)=r^2$.
          Thus, by the change of variable $r^{2}/t$ in the second integral above, we
          can express the integral part of $h$ as a
          single integral between $x_{3}(z)$ and $x_{4}(z)$.
          \item The polynomials $a,b,c,d$ verify an interesting relationship
          with respect to the transformation $t\mapsto r^2/t$. Indeed,
          if $f$ stands for $a,b$ or $c$, we easily verify that $f(r^{2}/t,z)=
          (r^{2}/t^{2})f(t,z)$,
          so that we also have $d(r^2/t,z)=(r^4/t^4)d(t,z)$. In particular,
          we get $\mu_{m_{0}}(r^2/t,z)=(t^{2}/r^{2})\mu_{m_{0}}(t,z)$.

     \end{enumerate}
With these remarks we obtain~:
\begin{prop}\label{explicit_h(x,z)_third}
The function $h$ admits the following integral expression~:
     \begin{equation*}
          h\left(x,z\right)=\frac{x}{\pi } \int_{x_{3}\left(z\right)}^{x_{4}\left(z\right)}
          \left(t^{n_{0}}-\left(\frac{r^{2}}{t}\right)^{n_{0}}\right)
          \frac{\mu_{m_{0}}\left(t,z\right)\sqrt{-d\left(t,z\right)}}
          {t\left(t-x\right)}\textnormal{d}t+xP_{\infty}\left(x\mapsto
          x^{n_{0}-1}Y_{0}\left(x,z\right)^{m_{0}}\right)\left(x\right).
     \end{equation*}
     Above, $x$ belongs to the open centered disc of radius $r$,
     $z$ to $]0,z_{1}]$,
     $z_{1}$ being defined in Lemma~\ref{lemma_branched_points},
     and $P_{\infty}$ is the principal part at infinity, defined in
     Lemma~\ref{simplification_principal_part} and Remark~\ref{rem_principal_part}.
\end{prop}

\subsection{Chebyshev polynomials}\label{Chebychev_polynomials}

We close the study of the explicit expressions of $h$ by making a
--last but-- quite natural change of variable in the
integral~(\ref{new_integral_form_h}). Define
$\hat{b}(t,z)=b(t,z)/
(4a(t,z)c(t,z))^{1/2}$. Then $t\mapsto
\hat{b}(t,z)$ is clearly a diffeomorphism between
$]x_{1}(z),x_{2}(z)[$ (resp.\ $]x_{3}(z),x_{4}(z)[$) and
$]-1,1[$. Moreover, $\mu_{m_{0}}$ expresses oneself with
a more natural way in the variable $\hat{b}$ since the following
equality holds~:
     \begin{equation}\label{natural_expression_mu}
          \mu_{m_0}\left(t,z\right)\sqrt{-d\left(t,z\right)}=
          \left(\frac{c\left(t,z\right)}{a\left(t,z\right)}\right)^{m_{0}/2}
          U_{m_{0}-1}\left(-\widehat {b}\left(t,z\right)\right)
          \sqrt{1-\widehat{b}\left(t,z\right)^2},
     \end{equation}
where the $U_{n}$, $n\in \mathbb{N}$, are the Chebyshev polynomials
of the second kind. We recall that they are the orthogonal
polynomials associated to the weight $t\mapsto
(1-t^{2})^{1/2}1_{]-1,1[}(t)$ and that their
explicit expression is~:
     \begin{equation}\label{Chebyshev_second_kind}
          U_{n}\left(t\right)=
          \frac{ \left(t+\sqrt{t^2-1}\right)^{n+1}-\left(t-\sqrt{t^2-1}\right)^{n+1}}
          {2\sqrt{t^2-1}}=\sum_{k=0}^{\left\lfloor n/2\right\rfloor }{\genfrac(){0cm}{0}{n+1}{2k+1}  }
          \left(t^2-1\right)^k t^{n-2 k},\;\;\;\; n\in \mathbb{N}.
     \end{equation}
We also recall two properties of the Chebyshev polynomials of the second
kind that we will especially use~: first that they have the parity
of their order~: for all $n$ in $\mathbb{N}$ and all $u$ in 
$\mathbb{C}$, $U_{n}\left(-u\right)=
\left(-1\right)^{n}U_{n}\left(u\right)$~; second their expansion in
the neighborhood of $1$~: $U_{n}(u)=(n+1)(1+n(n+2)(u-1)/3
+\mathcal{O}((u-1)^2))$. For all the facts concerning the
Chebyshev polynomials used here, we refer to~\cite{SZ}.

We are now ready to make the change of variable mentioned above~:
to set $\hat{b}\left(t,z\right)=u$. But $\hat{b}(t,z)=u$
if and only if $b(t,z)-2 u r p_{01} z t =0$, in other
words if and only if $t=t_{1}(u,z)$ or
$t=t_{2}(u,z)$ where $t_{1}=t_{-}$, $t_{2}=t_{+}$ and~:
     \begin{equation}
     \label{tttt}
          t_{\pm}\left(u,z\right) = \frac{1+2\sqrt{p_{01}p_{0-1}}u z\pm
          \sqrt{\left(1+2\sqrt{p_{01}p_{0-1}}u z\right)^2-4p_{10}p_{-10}z^2}}{2p_{10}z}.
     \end{equation}
Of course, we find again the explicit expression of
the branched points as the values of $t_{1}$ and $t_{2}$
at $u=\pm 1$, more precisely
$t_{1}(1,z)=x_{1}(z)$,
$t_{1}(-1,z)=x_{2}(z)$,
$t_{2}(-1,z)=x_{3}(z)$ and
$t_{2}(1,z)=x_{4}(z)$, in accordance with
Lemma~\ref{lemma_branched_points}. This change of variable allows
us to write the following integral representation for the function $h$~:
\begin{prop}\label{lemma_final_form}
The function $h$ admits the following integral expression~:
     \begin{eqnarray}\label{final_form_SRW}
          h\left(x,z\right)&=&xP_{\infty}\left(x\mapsto
          x^{n_{0}-1}Y_{0}\left(x,z\right)\right)^{m_{0}}\left(x\right)\\&+&
          \frac{x}{\pi }\left(\frac{p_{0-1}}{p_{01}}\right)^{m_{0}/2}
          \int_{-1}^{1} \left(t_{2}\left(u,z\right)^{n_{0}}-\left(\frac{r^{2}}
          {t_{2}\left(u,z\right)}\right)^{n_{0}}\right)\frac{\partial_{u}
          t_{2}\left(u,z\right)U_{m_{0}-1}\left(-u\right)}
          {t_{2}\left(u,z\right)\left(t_{2}\left(u,z\right)-x\right)}
          \sqrt{1-u^{2}}\textnormal{d}u.\nonumber
    \end{eqnarray}
Above, $x$ belongs to the open centered disc of radius $r$,
and $z\in ]0,z_{1}]$,
$z_{1}$ being defined in Lemma~\ref{lemma_branched_points}.
\end{prop}

We take now some notations that will be useful in the sequel,
notably in Section~\ref{h_1_z}~: we set
$k_{1}(u) = -2(p_{01}p_{0-1})^{1/2}u+2(p_{10}p_{-10})^{1/2}$ and
$k_{2}(u) = -2(p_{01}p_{0-1})^{1/2}u-2(p_{10}p_{-10})^{1/2}$~; so that
for instance $t_{\pm}(u,z) = (1+2(p_{01}p_{0-1})^{1/2}u z\pm ((1-k_{1}(u)z)
(1-k_{2}(u)z))^{1/2})/(2p_{10}z)$. Moreover, we easily show the two
following facts~:
     \begin{equation}
     \label{kk}
          \left\{\begin{array}{ccc}
          \displaystyle  \partial_{u} t_{\pm}\left(u,z\right) & = &
          \displaystyle \pm t_{\pm}\left(u,z\right)
          \frac{2\sqrt{p_{01}p_{0-1}}z}
          {\sqrt{\left(1-k_{1}\left(u\right)z\right)
          \left(1-k_{2}\left(u\right)z\right)}},\\
          \displaystyle  \frac{1}{t_{2}\left(u,z\right)-x} & = &
          \displaystyle \frac{1}{2}
          \frac{ \sqrt{\left(1-k_{1}\left(u\right)z\right)
          \left(1-k_{2}\left(u\right)z\right)}-
          \left(1+\left(2\sqrt{p_{01}p_{0-1}}u-2p_{10}x\right)z\right)}
          {x+\left(2\sqrt{p_{01}p_{0-1}}u x -
          \left(p_{10}x^2+p_{-10}\right)\right)z}.
          \end{array}\right.
     \end{equation}

\begin{rem}[the gambler ruin]\label{remark_gambler_ruin}
{\rm Consider the explicit
expression~(\ref{final_form_SRW}), in which we do $p_{10},p_{-10}\to
0$, $x=1$. Moreover,  we take $n_{0}=1$ to lighten the technical details. This
expression becomes~:
     \begin{equation}\label{h1zruin}
      h\left(1,z\right)=\frac{1}{\pi }\left( \frac{p_{0-1}}{p_{01}}\right)^{m_{0}/2}
       \int_{-1}^{1}\frac{2\sqrt{p_{01}p_{0-1}}z}{1+2\sqrt{p_{01}p_{0-1}}z u}
        U_{m_{0}-1}\left(-u\right)\sqrt{1-u^{2}}\text{d}u.
     \end{equation}
Not unexpectedly this quantity is equal to~:
\begin{equation*}
          \lambda_{m_{0}}\left(z\right)=
          \left( \frac{1-\sqrt{1-4p_{01}p_{0-1}z^2}}{2p_{01}z}\right)^{m_{0}},
     \end{equation*}
which is the generating function of the ruin probabilities for the
gambler ruin problem~:
$\lambda_{m_{0}}(z)=\sum_{k=0}^{+\infty }
\mathbb{P}_{m_{0}}($to be ruined in a time $ k)z^{k}$, 
in accordance with~\cite{FEL}.
Let us sketch the proof of this fact. We start by
remarking that for $n\in \mathbb{N}$ and $t\in \mathbb{C}\setminus
[-1,1]$, we have~:
     \begin{equation}\label{application_residue_theorem_principal_part}
          \frac{1}{\pi }\int_{-1}^{1}
          \frac{u^{n}\sqrt{1-u^{2}}}{u-t}\text{d}u =
          \left(t^{k}\sqrt{t^{2}-1}-P_{\infty}
          \left(t\mapsto t^{k}\sqrt{t^{2}-1}\right)\left(t\right)\right),
     \end{equation}
where, if $f$ is a function meromorphic at infinity,
$P_{\infty}(f)$ denotes its principal part at infinity.
To get (\ref{application_residue_theorem_principal_part}), we
consider the function $u^n (1-u^2)^{1/2}/(u-t)$ and
we integer it  on a closed
contour that surrounds at a distance equal to $\epsilon$ the segment
$[-1,1]$, like the contour drawn in
Figure~\ref{contour_simplification} of
Subsection~\ref{Integral_representation}, then we apply the residue
theorem at infinity and
at last we do $\epsilon$ going to zero.

Next, we apply~(\ref{application_residue_theorem_principal_part}) at every
monomial that composes the polynomial $U_{m_{0}-1}$. Using the
linearity of the principal part, integral~(\ref{h1zruin}) becomes~:
     \begin{equation}\label{application_residue_theorem_principal_part_h_1_z}
          h\left(1,z\right) = -\left( \frac{p_{0-1}}{p_{01}}\right)^{m_{0}/2}
          \left.\left(  U_{m_{0}-1}\left( u\right)\sqrt{u^{2}-1}-P_{\infty }\left(u\mapsto U_{m_{0}-1}
          \left(u\right)\sqrt{u^{2}-1}\right)\left(u\right)\right)
          \right|_{u=\frac{1}{2\sqrt{p_{01}p_{0-1}}z} }.
     \end{equation}
Introduce now $(T_{n})_{n\in \mathbb{N}}$, the Chebyshev polynomials
of the first kind~; we recall that they are the orthogonal polynomials associated to
the weight $t\mapsto
1_{\left]-1,1\right[}\left(t\right)/(1-t^{2})^{1/2}$. As for the
Chebyshev polynomials of the second kind $U_{n}$ (defined
in~(\ref{Chebyshev_second_kind})), there exists an explicit
formulation for the Chebyshev polynomials of the first kind~:
     \begin{equation}\label{Chebyshev_first_kind}
          T_{n}\left(t\right)=\frac{1}{2}\left( \left(t+\sqrt{t^{2}-1}\right)^{n}
          +\left(t-\sqrt{t^{2}-1}\right)^n\right),\;\;\;\;n\in \mathbb{N}.
     \end{equation}
Moreover, as it is proved in~\cite{SZ}, there exists 
--among others-- the following link between the Chebyshev polynomials
of the first and second kind~:
     \begin{equation*}
          P_{\infty }\left(t\mapsto U_{n}\left(t\right)\sqrt{t^{2}-1}\right) =
          T_{n+1},\hspace{5mm}
          P_{\infty}\left(t\mapsto T_{n+1}\left(t\right)/\sqrt{t^{2}-1}\right) =
          U_{n},\hspace{5mm}n\in \mathbb{N}.
     \end{equation*}
These relations allow to
simplify~(\ref{application_residue_theorem_principal_part_h_1_z}).
We find~:
     \begin{equation*}
          h\left(1,z\right) = -\left( \frac{p_{0-1}}{p_{01}}\right)^{m_{0}/2}
          \left.\left(  U_{m_{0}-1}\left( u\right)
          \sqrt{u^{2}-1}-T_{m_{0}}\left(u\right)\right)
          \right|_{u=1/\left(2\sqrt{p_{01}p_{0-1}}z\right) }.
     \end{equation*}
But with~(\ref{Chebyshev_second_kind})
and~(\ref{Chebyshev_first_kind}),
 $U_{m_{0}-1}\left(u\right)(u^{2}-1)^{1/2}-T_{m_{0}}\left(u\right)
=-(u-(u^{2}-1)^{1/2})^{m_{0}}$ so that we find that
$h(1,z)=\lambda_{m_{0}}(z)$.}
\end{rem}

\subsection{Analytic continuation}
\label{Analytic_continuation}
We recall from Equation~(\ref{def_generating_functions})
of Section~\ref{Intro} that
$h$ is initially defined
on the closed unit disc,
inside of which it is holomorphic.

\begin{prop}
\label{proposition_continuation}
The function $h$ admits an
analytic continuation on $\mathbb{C}\setminus [x_{3}(z),x_{4}(z)]$.
Moreover, the set $\mathbb{C}\setminus [x_{3}(z),x_{4}(z)]$
is the biggest one where $h$ can be continued
into a single valued function.
\end{prop}

\begin{proof}
We could equally use one or the other
of the formulations obtained in
Propositions~\ref{explicit_h(x,z)},~\ref{explicit_h(x,z)_second},
\ref{explicit_h(x,z)_third} and~\ref{lemma_final_form},
but the simplest is perhaps to
make the use of Proposition~\ref{explicit_h(x,z)_third},
as well as the analytic properties of Cauchy integrals,
that can be found for instance in~\cite{LU}.
\end{proof}

\begin{rem}
  {\rm A nice property peculiar to the random walks we are studying
  is that the curve $\mathcal{M}_{z}=\mathcal{C}(0,r)$
  belongs to the closed unit disc.
  We can thus solve the boundary value problem
  with boundary condition~(\ref{SR_problem_h}) and then
  continue $h$ using its explicit expression.
  In Section~\ref{Extension_results},
  we will see that for
  walks under general hypothesis (H2) of the introduction, it is
  quite possible that a part --or even the whole--
  of the associated curve $\mathcal{M}_{z}$ belongs to the exterior of the
  closed unit disc. There we will have \emph{first} to continue $h$
  as a holomorphic function up to $\mathcal{M}_{z}$
  and \emph{after} to solve the boundary
  value problem. We will do this continuation using Galois
  automorphisms (notion that will be defined here in the
  proof of Proposition~\ref{prop_continuation_Delta_zero}),
  following the procedure which in the heart of
  Book~\cite{FIM}.}
\end{rem}



\section{Probability of being absorbed with a fixed time}\label{h_1_z}

\subsection{Absorption probabilities in the case of
a drift zero}\label{Simple_walk_drift_zero}

\begin{prop}\label{main_result_simple_walk_drift_zero}
We suppose here that the two drifts~(\ref{drift}) are equal to zero,
in other words that $p_{-10}=p_{10}$ and $p_{0-1}=p_{01}$.
Define $S=\inf\left\{n\in \mathbb{N}:\,
(X(n),Y(n))\,\textnormal{hits the $x$-axis}\right\}$
the hitting time of the $x$-axis.
The following asymptotic holds~:
     \begin{equation}\label{asymptotic_hitting_times_real_axis_drift_zero}
          \mathbb{P}_{\left(n_{0},m_{0}\right)}\left(S = k \right)
          \sim \frac{n_{0} m_{0}}{2 \pi\sqrt{p_{10}p_{01}}}
          \frac{1}{k^{2}}, \ \  k \to \infty.
     \end{equation}
\end{prop}

\begin{proof}
Setting $x=1$, $p_{-10}=p_{10}$, $p_{0-1}=p_{01}$
in~(\ref{final_form_SRW}) leads to~:
     \begin{equation*}
          h\left(1,z\right)=\frac{p_{01}z}{\pi}\int_{-1}^{1}
          \frac{t_{2}\left(u,z\right)^{n_{0}}
          -t_{1}\left(u,z\right)^{n_{0}}}{\sqrt{
          \left(1-k_{1}\left(u\right)z\right)
          \left(1-k_{1}\left(u\right)z\right)}}
          U_{m_{0}-1}\left(-u\right)
          \left(\sqrt{\frac{1-k_{2}\left(u\right)z}
          {1-k_{1}\left(u\right)z}}-1\right)
          \sqrt{1-u^{2}}\text{d}u.
     \end{equation*}
Using the explicit expressions of $t_{1}$ and $t_{2}$
given in~(\ref{tttt}), we immediately
notice that the function
     \begin{equation*}
          F\left(u,z\right)=\frac{p_{01}z}{\pi}
          \frac{t_{2}\left(u,z\right)^{n_{0}}
          -t_{1}\left(u,z\right)^{n_{0}}}{\sqrt{
          \left(1-k_{1}\left(u\right)z\right)
          \left(1-k_{1}\left(u\right)z\right)}}
          U_{m_{0}-1}\left(-u\right)
     \end{equation*}
is a polynomial in the two variables $(u,z)$.
So we can write $F$ as the
following finite sum~: $F(u,z)=\sum_{i,j}F_{i j}(u+1)^{i}(z-1)^{j}$
with coefficients $F_{i j}$ that can of course be computed explicitly, for
example $F_{0 0}=n_{0} m_{0} p_{01}/(\pi p_{10})$. Since adding a
polynomial does not change the asymptotic of the function's coefficients,
the coefficients of $h\left(1,z\right)$ have the same asymptotic
as those of the following function~:
     \begin{equation*}
          l\left(z\right)=\int_{-1}^{1}
          F\left(u,z\right)\sqrt{\frac{1-k_{2}\left(u\right)z}
          {1-k_{1}\left(u\right)z}}\sqrt{1-u^{2}}\text{d}u.
     \end{equation*}
Consider now the function
$G(u,z)=F(u,z)(1-k_{2}(u)z)^{1/2}$. Since
$k_{2}(-1)=2(p_{01}-p_{10})<1$, the function of two variables
$G$ is holomorphic in $\mathcal{D}(0,1+\epsilon )^2$,
where $\epsilon >0$ is sufficiently small.
For this reason, $G$ can be expanded
according to the powers $(u+1)^{i}(z-1)^{j}$~: $G(u,z)=
\sum_{i,j}G_{i j}(u+1)^{i}(z-1)^{j}$. As for $F$, all the
coefficients $G_{i j}$ can be explicited~; for instance $G_{0
0}=2n_{0}m_{0}p_{01}/(\pi\sqrt{p_{10}})$. With these notations, the
function $l$ becomes
     \begin{equation*}
          l\left(z\right)=\sum_{i,j}G_{i j}\left(z-1\right)^{j}
          \int_{-1}^{1}\left(1-u\right)^{i}\frac{\left(1-u^{2}\right)^{1/2}}
          {\left(1-k_{1}\left(-u\right)z\right)^{1/2}}\text{d}u.
     \end{equation*}
Thanks to Lemma~\ref{lem_log_singularity} below, we obtain that
$l(z)$ (and therefore $h(1,z)$) is a
function (i) holomorphic in the unit disc (ii) having a holomorphic
continuation up to every point of the unit circle except $1$ (iii)
having a logarithmic singularity at $1$.

As concerns the logarithmic singularity,
we can  be more precise~:
Lemma~\ref{lem_log_singularity}
asserts the existence of
$f(z)=\sum_{i,j}G_{i j}f_{i}(z)(z-1)^{i+j}$ and
$g(z)=\sum_{i,j}G_{i j}g_{i}(z)(z-1)^{j}$
such that $l(z)=f(z)(z-1)\ln(1-z)+g(z)$.

Moreover, using the fact that
$G_{0 0}=2n_{0}m_{0}p_{01}/(\pi\sqrt{p_{10}})$ and once
again with Lemma~\ref{lem_log_singularity},
we find $f(1)=-n_{0} m_{0}/(2\pi \sqrt{p_{01}p_{10}})$.

We can now easily find the asymptotic of the coefficients of the
Taylor series at $0$ of $l(z)$, thus also of $h\left(1,z\right)$,
following the principle explained hereunder~:
if $F(z)=\sum_{k}c_{k}z^k$ is a function (i) holomorphic
in the open unit disc (ii) having
a holomorphic continuation at every point
of the unit circle except $1$ (iii) having at $1$
a logarithmic singularity in the sense
that in the neighborhood of $1$,
$F$ can be written as $F(z)=F_{1}(z)+F_{2}(z)\ln(1-z)$
where $F_{1}$ and $F_{2}$ are holomorphic functions
at $1$, then the asymptotic of the coefficients
of the Taylor series can easily be calculated~:
if $q=\inf\{p\in \mathbb{N} : {F_{2}}^{(p)}(1)\neq 0\}$,
then $c_{k}\sim (-1)^q{F_{2}}^{(q)}(1)/k^{q+1}$ as $k\to +\infty $.

We use this result with
$q=1$ and $F_{2}'(1)=f(1)=-n_{0}m_{0}/(2\pi \sqrt{p_{01}p_{10}})$,
the
asymptotic~(\ref{asymptotic_hitting_times_real_axis_drift_zero})
comes immediately.

\end{proof}

\begin{lem}\label{lem_log_singularity}
Let $i$ be a non negative integer. The function $F_{i}$, defined by
     \begin{equation}\label{int_log_singularity}
          F_{i}\left(z\right)=\int_{-1}^{1}\left(1-u\right)^{i}
          \frac{\left(1-u^{2}\right)^{1/2}}
          {\left(1-k_{1}\left(-u\right)z\right)^{1/2}}\textnormal{d}u,
     \end{equation}
is holomorphic in the open unit disc. Moreover, it
can be continued into a holomorphic
function in the neighborhood
of any point of the unit circle except $1$.
At $z=1$, the function has a logarithmic singularity~; 
more precisely
there exist two
functions $f_{i}$ and $g_{i}$ holomorphic at $z=1$,
$f_{i}\left(1\right)\neq 0$, such that
$F_{i}(z)=(z-1)^{i+1}\ln(1-z)f_{i}(z)+g_{i}(z)$.
Moreover, $f_{0}(1)=-1/(4{p_{01}}^{3/2})$.
\end{lem}

\begin{proof}
The two facts that the integrals considered in
Lemma~\ref{lem_log_singularity} are holomorphic in the unit disc
and also that they can be continued into holomorphic functions
trough every point of the unit circle except $1$ come immediately
from the theory of integrals with parameters. We will therefore
concentrate the proof on the logarithmic singularity.

First, we replace
the lower bound $-1$ in the integrals~(\ref{int_log_singularity}) by
$-p_{10}/p_{01}$, which does not
change the singularity in the
neighborhood of $1$ of the functions
$F_{i}(z)$
since doing this is equivalent to add to
$F_{i}(z)$ a function with a radius of convergence
strictly larger
than $1$.
Then, the change of variable
$v^{2}=k_{1}(-u) =2(p_{01}u+p_{10})$ gives
     \begin{equation*}
          \int_{-p_{10}/p_{01}}^{1}\frac{
          \left(1-u\right)^{i}\left(1-u^{2}\right)^{1/2}}
          {\left(1-k_{1}\left(-u\right)z\right)^{1/2}}\text{d}u=
          \frac{2}{\left(2p_{01}\right)^{2+i}}
          \int_{0}^{1}\frac{\left(1-v^{2}\right)^{1/2+i}}{\left(1-z v^{2}\right)^{1/2}}
          v\sqrt{v^{2}+2\left(p_{01}-p_{10}\right)}\text{d}v.
     \end{equation*}
Next, using the expansion of $\sqrt{v}$ at $v=1$, we can expand the
function $v(v^{2}+2\left(p_{01}-p_{10}\right))^{1/2}$ according to
the powers of $(1-v^{2})$ :
$v(v^{2}+2(p_{01}-p_{10}))^{1/2}=\sum_{i}c_{i}(1-v^{2})^{i}$
with $c_{0}=2\sqrt{p_{01}}$, $c_{1}=(1+4p_{01})/(4\sqrt{p_{01}})$, etc.

We will now explain why there exist functions $\phi_{k}$ and
$\psi_{k}$ holomorphic in the neighborhood of $1$,
$\phi_{k}(1)\neq 0$ such that
     \begin{equation}\label{expression_terms_K_E}
          \int_{0}^{1}\frac{\left(1-v^{2}\right)^{1/2+k}}
          {\left(1-z v^{2}\right)^{1/2}}\text{d}v
          =\left(z-1\right)^{k+1}\ln\left(1-z\right)\phi_{k}\left(z\right)
          +\psi_{k}\left(z\right).
     \end{equation}
But before, we show how~(\ref{expression_terms_K_E})
allows to complete the proof of
Lemma~\ref{lem_log_singularity}~:
with the notations of~(\ref{expression_terms_K_E}) we set
$\tilde{g}_{i}(z)=2/(2p_{01})^{i+2}\sum_{k}c_{k}\psi_{k+i}(z)$ and
$f_{i}(z)=2/(2p_{01})^{i+2}\sum_{k}c_{k}\phi_{k+i}(z)(z-1)^{k}$,
we obtain that~:
     \begin{equation}
     \label{zdf}
          \int_{-p_{10}/p_{01}}^{1}\frac{\left(1-u\right)^{i}\left(1-u^{2}\right)^{1/2}}
          {\left(1-k_{1}\left(-u\right)z\right)^{1/2}}\text{d}u=
          \left(z-1\right)^{i+1}\ln\left(1-z\right)
          f_{i}\left(z\right)+\tilde{g}_{i}\left(z\right).
     \end{equation}
Then, we replace the lower bound
$-p_{10}/p_{01}$ by $-1$,
what changes $\tilde{g}_{i}$ in
a new function holomorphic in the
neighborhood of $1$, that we call $g_{i}$,
but what does not change the function $f_{i}$,
for the reasons already explained at the beginning of the proof.

So, it remains to prove~(\ref{expression_terms_K_E}).
The proof consists in expressing the
integrals~(\ref{expression_terms_K_E}) in terms of $K$ and $E$,
the two classical Legendre's complete elliptic integrals of the
first and second kind, defined by~:
     \begin{equation*}
          K\left(z\right)=\int_{0}^{1}\frac{\text{d}v}
          {\left(\left(1-v^{2}\right)\left(1-z v^{2}\right)
          \right)^{1/2}},\hspace{5mm}
          E\left(z\right)=\int_{0}^{1}\frac{\left(
          1-z v^{2}\right)^{1/2}}{\left(1-v^{2}\right)^{1/2}}
          \text{d}v,
     \end{equation*}
and next in using the well known results concerning these elliptic integrals,
notably their behavior in the neighborhood
of $1$ and in particular their --logarithmic-- singularity at $1$~;
all these properties can be
found e.g.\ in~\cite{SG2}.

The functions $K$ and $E$ are
manifestly holomorphic
in the open unit disc, continuable
trough any point of the unit circle except $1$,
and from
the so called Abel's identity (see~\cite{SG2})
it can be deduced that the functions
$K$ and $E$
have at $1$ a logarithmic singularity as follows~:
$K(z)=\rho_{K}(z)+\sigma_{K}(z)\ln(1-z)$ and
$E(z)=\rho_{E}(z)+\sigma_{E}(z)(z-1)\ln(1-z)$,
where the functions $\rho$ and $\sigma$ are holomorphic
in the neighborhood of $1$ and
$\sigma_{K}(1)=-1/2$ and $\sigma_{E}(1)=1/4$.

Moreover, for any non negative integer $k$,
we can find two polynomials
$P_{k}$ and $Q_{k}$ such that~:
     \begin{equation*}
          \int_{0}^{1}\frac{\left(1-v^{2}\right)^{1/2+k}}
          {\left(1-z v^{2}\right)^{1/2}}\text{d}v=\frac{
          P_{k}\left(z\right)E\left(z\right)+
          Q_{k}\left(z\right)K\left(z\right)}{z^{k+1}}.
     \end{equation*}
These polynomials could be
explicitly calculated, for instance, $P_{0}(z)=1$
and $Q_{0}(z)=z-1$.
Therefore, setting
$\psi_{k}(z)=(P_{k}(z)\rho_{E}(z)+Q_{k}(z)\rho_{K}(z))/z^{k+1}$ and
$\tilde{\phi}_{k}(z)=(P_{k}(z)\sigma_{E}(z)(z-1)+Q_{k}(z)\sigma_{K}(z))/z^{k+1}$,
we obtain that
     \begin{equation*}
          \int_{0}^{1}\frac{\left(1-v^{2}\right)^{1/2+k}}
          {\left(1-z v^{2}\right)^{1/2}}\text{d}v
          =\ln\left(1-z\right)\widetilde{\phi}_{k}\left(z\right)
          +\psi_{k}\left(z\right).
     \end{equation*}
Then, to prove~(\ref{expression_terms_K_E}) it suffices to verify that
we can write $\tilde{\phi}_{k}(z)$ as $(z-1)^{k+1}\phi_{k}(z)$, where
$\phi$ is holomorphic at $1$. We don't make this verification
in the general case, because the calculations are somewhat tedious
--the expressions of
the polynomials $P_{k}$ and $Q_{k}$ are rather complicated--,
but do it in case $k=0$~: thanks to
the explicit expression of $P_{0}$ and $Q_{0}$
given above we have  $\tilde{\phi}_{0}=(z-1)(\sigma_{E}(z)+\sigma_{K}(z))/z$,
hence the result by setting $\phi_{0}(z)=(\sigma_{E}(z)+\sigma_{K}(z))/z$.

To prove the last fact claimed in
Lemma~\ref{lem_log_singularity}, namely that
$f_{0}(1)=-1/(4{p_{01}}^{3/2})$,
we use the fact that $f_{0}(1)=2c_{0}(\sigma_{E}(1)+\sigma_{K}(1))/(2{p_{01}})^2
=-1/(4{p_{01}}^{3/2})$.
\end{proof}

\begin{cor}\label{hitting_times_definitions}

Take the following notations~:
     \begin{equation}\label{def_hitting_times}
          \left\{ \begin{array}{ccccc}
          S&=&\inf\left\{n\in \mathbb{N}:\, \left(X\left(n\right),Y\left(n\right)\right)
          \,\textnormal{hits the $x$-axis}\right\}, \\
          T&=&\inf\left\{n\in \mathbb{N}:\, \left(X\left(n\right),Y\left(n\right)\right)
          \,\textnormal{hits the $y$-axis}\right\}, \\
          \tau&=&\inf\left\{n\in \mathbb{N}:\, \left(X\left(n\right),Y\left(n\right)\right)
           \,\textnormal{hits the boundary}\right\}
          =S\wedge T.\end{array}\right.
     \end{equation}
Then the following equivalents hold for $\mathbb{P}_{(n_{0},m_{0})}(T = k ) $
and for the probability
of not being absorbed at time $k$~:
     \begin{equation*}\label{not_absorb_time_k}
          \mathbb{P}_{\left(n_{0},m_{0}\right)}\left(T = k \right)
          \sim \frac{n_{0} m_{0}}{2\pi \sqrt{p_{10}p_{01}}}\frac{1}{k^{2}}
          ,\hspace{5mm}
          \mathbb{P}_{\left(n_{0},m_{0}\right)}\left(\tau \geq k \right)
          \sim \frac{n_{0} m_{0}}{\pi \sqrt{p_{10}p_{01}}}\frac{1}{k},\hspace{5mm}k\to \infty.
     \end{equation*}
\end{cor}

\begin{proof}

We immediately obtain the first part of
Corollary~\ref{hitting_times_definitions} from
Proposition~\ref{main_result_simple_walk_drift_zero} by exchanging
the parameters $p_{10},p_{-10}$ and $p_{01},p_{0-1}$.
Moreover, since the random walk can be 
absorbed by at most one of the axes, we get
     \begin{equation}
          \label{kkk}
          \mathbb{P}_{\left(n_{0},m_{0}\right)}\left(\tau \geq k \right)=
          \mathbb{P}_{\left(n_{0},m_{0}\right)}\left(k\leq S <\infty
          \right)+\mathbb{P}_{\left(n_{0},m_{0}\right)}\left(k\leq
          T <\infty  \right)+
          \mathbb{P}_{\left(n_{0},m_{0}\right)}\left((S=\infty)
          \cap (T=\infty)\right).
     \end{equation}
This random walk being absorbed almost surely
(we recall that we have supposed $p_{-10}=p_{10}$
and $p_{0-1}=p_{01}$), we have
$\mathbb{P}_{\left(n_{0},m_{0}\right)}\left((S=\infty) \cap
(T=\infty )\right)=0$ and
Corollary~\ref{hitting_times_definitions} is immediate.
\end{proof}


\subsection{Absorption probabilities in the case of a non zero drift}
\label{simple_walk_drift_non_zero}

In Subsection~\ref{Simple_walk_drift_zero},
we were interested in the hitting time of the
boundary of $(\mathbb{Z}_{+})^{2}$ in the case of the two drifts
$M_{x}$ and $M_{y}$ equal to zero. Now, we state analogous results
when one (Proposition~\ref{zek}) or two (Proposition~\ref{zer}) of
$M_{x}$ and $M_{y}$ are not zero.


\begin{prop}
\label{zek}
Suppose that $M_{x}>0$, $M_{y}>0$ and
let $S$ be the hitting time of the $x$-axis, defined
in~(\ref{def_hitting_times}). Then
$\mathbb{P}_{(n_{0},m_{0})}(S = k)$, the
probability of being absorbed in the $x$-axis at time $k$,
admits the asymptotic as $k$ goes to infinity~:
     \begin{equation}\label{asymptotic_hitting_times_real_axis_two_drift_positive}
          \frac{m_{0}}{2\sqrt{\pi}}\sqrt{\frac{p_{10}+p_{-10}+
          2\sqrt{p_{01}p_{0-1}}}{\sqrt{p_{01}p_{0-1}}}}
          \left(\frac{p_{0-1}}{p_{01}}\right)^{m_{0}/2}
          \left(1-\left(\frac{p_{-10}}{p_{10}}\right)^{n_{0}}\right)
          \frac{\left(p_{10}+p_{-10}+2\sqrt{p_{01}p_{0-1}}\right)^{k}}{k^{3/2}}.
     \end{equation}
\end{prop}


\begin{proof}
Lemma~\ref{lemma_final_form} gives that
$h\left(1,z\right)$ is, up to a polynomial, equal to~:
     \begin{eqnarray}
          \frac{1}{\pi}\left( \frac{p_{0-1}}{p_{01}}\right)^{m_{0}/2}
          \int_{-1}^{1} \left(t_{2}\left(u,z\right)^{n_{0}}-
          \left(\frac{r^2}{t_{2}\left(u,z\right)}\right)^{n_{0}}\right)
          \frac{2\sqrt{p_{01}p_{0-1}}z}
          {\sqrt{\left(1-k_{1}\left(u\right)z\right)
          \left(1-k_{2}\left(u\right)z\right)}}
          \times \nonumber \\
          \times \frac{1}{2}
          \frac{\sqrt{\left(1-k_{1}\left(u\right)z\right)
          \left(1-k_{2}\left(u\right)z\right)}
          -\left(1-k_{4}\left(u\right)z\right)}
          {1-k_{3}\left(u\right)z}
          U_{m_{0}-1}\left(-u\right)\sqrt{1-u^{2}}\text{d}u,
          \label{k1_k2_k3_k4}
     \end{eqnarray}
where we have set
$k_{1}(u) = -2(p_{01}p_{0-1})^{1/2}u+2(p_{10}p_{-10})^{1/2}$,
$k_{2}(u) = -2(p_{01}p_{0-1})^{1/2}u-2(p_{10}p_{-10})^{1/2}$,
$k_{3}(u)=
-2(p_{01}p_{0-1})^{1/2}u+p_{10}+p_{-10}$ and
$k_{4}(u)=
-2(p_{01}p_{0-1})^{1/2}u+2p_{10}$.
Due to the inequalities $2(p_{10}p_{-10})^{1/2}<p_{10}+p_{-10}<2p_{10}$, the
integral~(\ref{k1_k2_k3_k4}) is holomorphic in
the open disc $\mathcal{D}(0,k_{3}(-1)^{-1})$,
continuable at every point of the boundary
$\mathcal{C}(0,k_{3}(-1)^{-1})$ except at
$k_{3}(-1)^{-1}$.
Now we set $F(u,z)=
(t_{2}(u,z)^{n_{0}}-(r^2/t_{2}(u,z))^{n_{0}})(p_{01}p_{0-1})^{1/2}z
U_{m_{0}-1}(-u)(((1-k_{1}u)z)((1-k_{2}u)z))^{1/2}-(1-k_{4}(u)z))/
((1-k_{1}u)z)((1-k_{2}u)z))^{1/2}$ in such a way that the
function~(\ref{k1_k2_k3_k4}) can be expressed as the integral~:
     \begin{equation}
     \label{int_Cauchy_k3}
          \frac{1}{\pi}\left( \frac{p_{0-1}}{p_{01}}\right)^{m_{0}/2}
          \int_{-1}^{1}\frac{F\left(u,z\right)}
          {1-k_{3}\left(u\right)z}\sqrt{1-u^{2}}\text{d}u.
     \end{equation}
The function of two variables $F$ is certainly not holomorphic on the whole
$\mathbb{C}^{2}$ but is holomorphic on
$\mathcal{D}(0,k_{3}(-1)^{-1}+\epsilon)\times \mathcal{D}(0,1+\epsilon )$,
where $\epsilon $, which depends on the $p_{i j}$, is sufficiently small~:
indeed, thanks to --once again--
the obvious inequalities
$2(p_{10}p_{-10})^{1/2}<p_{10}+p_{-10}<2p_{10}$,
we immediately notice that
$(u,z)\mapsto (1-k_{i}(u)z)^{1/2}$
is, for $i\in\{1,2,4\}$, holomorphic
in $\mathcal{D}(0,k_{3}(-1)^{-1}+\epsilon)\times \mathcal{D}(0,1+\epsilon )$,
for sufficiently small values of $\epsilon$.

Therefore, we can write the expansion of
$F(u,z)$ in the neighborhood of $(-1,k_{3}(-1)^{-1})$,
say
$F(u,z)=\sum_{i,j}F_{i j}(1+u)^{i}(1-k_{3}(-1)^{-1}z)^{j}$.
The coefficients of this expansion could be explicitly
calculated, for instance, using that
$(k_{3}(-1)-k_{1}(-1))(k_{3}(-1)-k_{2}(-1))=(p_{10}-p_{-10})^{2}$ and
$t_{2}(-1,k_{3}(-1))=1$ we find
$F_{00}=2m_{0}(1-(p_{-10}/p_{10})^{n_{0}})(p_{01}p_{0-1})^{1/2}/k_{3}(-1)$.

Then, in accordance with
Lemma~\ref{lemma_some_asymptotic} below we set
$f(z)=\sum_{i,j}F_{i j}
(1-k_{3}(-1)^{-1})^{i+j}f_{i}(z)$
$(p_{0-1}/p_{01})^{m_{0}/2}/\pi$
and $g(z)=\sum_{i,j}F_{i j}
(1-k_{3}(-1)^{-1})^{j}g_{i}(z)(p_{0-1}/p_{01})^{m_{0}/2}/\pi$.
With these notations,
the function defined in~(\ref{int_Cauchy_k3}) is
equal to $g(z)+f(z)(1-k_{3}(-1)^{-1})^{1/2}$.

We can now easily find the asymptotic of the coefficients of the
Taylor series at $0$ of 
function~(\ref{int_Cauchy_k3}), or equivalently
of $h(1,z)$,
following a similar principle as the one explained
in the proof of Proposition~\ref{main_result_simple_walk_drift_zero},
and summarized below~:
if $F(z)=\sum_{k}c_{k}z^{k}$ is a function (i) holomorphic
in the open disc of radius $r$ (ii) having
a holomorphic continuation at every point
of the circle of radius $r$ except $r$ (iii) having at $r$
an algebraic singularity in the sense
that in the neighborhood of $r$,
$F$ can be written as $F(z)=F_{0}(z)+
\sum_{i=1}^{d}F_{i}(z)(1-z/r)^{\theta_{i}}$
where the $F_{i}$, $i\geq 0$, are holomorphic functions
in the
neighborhood of $r$,
not vanishing at $r$ for $i\geq 1$, the $\theta_{1}<\cdots <\theta_{d}$ are rational
but not integer,
then the asymptotic of the coefficients
of the Taylor series at $0$ can easily be calculated~:
$c_{k}\sim F_{1}(r)r^{k}/(\Gamma(-\theta_{1})k^{\theta_{1}+1})$ as $k\to +\infty $.
This principle is known as Pringsheim theorem.

With the last part
of Lemma~\ref{lemma_some_asymptotic}, we obtain
     \begin{equation*}
          F_{1}(r)=F_{00}f_{0}\left(k_{3}\left(-1\right)^{-1}\right)=
          -m_{0}\sqrt{\frac{p_{10}+p_{-10}+
          2\sqrt{p_{01}p_{0-1}}}{\sqrt{p_{01}p_{0-1}}}}
          \left(\frac{p_{0-1}}{p_{01}}\right)^{m_{0}/2}
          \left(1-\left(\frac{p_{-10}}{p_{10}}\right)^{n_{0}}\right),
     \end{equation*}
so that, using Pringsheim result with
this value of $F_{1}(r)$,
$\theta_{1}=1/2$, $r=k_{3}(-1)^{-1}$
and using the fact that $\Gamma(-1/2)=-2\sqrt{\pi}$, we get
immediately the announced
asymptotic~(\ref{asymptotic_hitting_times_real_axis_two_drift_positive}).
\end{proof}

\begin{lem}\label{lemma_some_asymptotic}
Let $i$ be a non negative integer. The function $G_{i}$, defined by
     \begin{equation*}
          G_{i}\left(z\right)=\int_{-1}^{1}\left(1-u\right)^{i}
          \frac{\left(1-u^{2}\right)^{1/2}}
          {1-k_{3}\left(-u\right)z}\textnormal{d}u,
     \end{equation*}
where $k_{3}(u)=-2(p_{01}p_{0-1})^{1/2}u+p_{10}+p_{-10}$,
is holomorphic in the open disc $\mathcal{D}(0,k_{3}(-1)^{-1})$. Moreover, it
can be continued into a holomorphic
function in the neighborhood
of any point of the circle $\mathcal{C}(0,k_{3}(-1)^{-1})$,
except $k_{3}(-1)^{-1}$.
At $z=k_{3}(-1)^{-1}$, the function has an algebraic singularity~; 
more precisely
there exist two
functions $f_{i}$ and $g_{i}$ holomorphic at $z=k_{3}(-1)^{-1}$,
$f_{i}(k_{3}(-1)^{-1})\neq 0$, such that
$G_{i}(z)=(1-k_{3}(-1)z)^{i+1/2}f_{i}(z)+g_{i}(z)$.
Moreover,
$f_{0}(k_{3}(-1)^{-1})=-(\pi/2)(k_{3}(-1)/(p_{01}p_{0-1})^{1/2})^{3/2}$.
     \end{lem}

\begin{proof}
The proofs of all assertions
of Lemma~\ref{lemma_some_asymptotic}
are based on the fact that the functions
$G_{i}$ can be explicitly calculated~:
     \begin{equation}\label{explicit_calculation_G_i}
          G_{i}\left(z\right)=\frac{-\pi}{2\sqrt{p_{01}p_{0-1}}z}
          \left.\left(\left(1-Z\right)^{i}\sqrt{Z^{2}-1}-
          P_{\infty}\left(\left(1-Z\right)^{i}\sqrt{Z^{2}-1}\right)
          \right)\right|_{Z=\frac{1-z\left(p_{10}+p_{-10}\right)}
          {2\sqrt{p_{01}p_{0-1}}z}},
     \end{equation}
where $P_{\infty}$ is the principal part defined in
Lemma~\ref{simplification_principal_part}.

To prove~(\ref{explicit_calculation_G_i}), we
start by remarking that
$1-k_{3}(-u)z=-2(p_{01}p_{0-1})^{1/2}(u-Z)$,
where $Z=(1-z(p_{10}+p_{-10}))/(2(p_{01}p_{0-1})^{1/2}z)$.
Then, we consider the function
$(1-u)^{i}(u^{2}-1)^{1/2}/(u-Z)$, well defined
on $\mathbb{C}\setminus ]-1,1[\cup \{Z\}$,
at which we apply the residue theorem
at infinity, on the same contour
as the one used in the Remark~\ref{remark_gambler_ruin},
namely a closed
contour that surrounds at a distance equal to $\epsilon$ the segment
$[-1,1]$. After that $\epsilon$ has gone to zero, we get~:
     \begin{equation*}
          \int_{-1}^{1}\left(1-u\right)^{i}
          \frac{\left(1-u^{2}\right)^{1/2}}{u-Z}\text{d}u=
          \left(\left(1-Z\right)^{i}\sqrt{Z^{2}-1}-
          P_{\infty}\left(\left(1-Z\right)^{i}\sqrt{Z^{2}-1}\right)
          \right),
     \end{equation*}
from which~(\ref{explicit_calculation_G_i}) and thus
Lemma~\ref{lemma_some_asymptotic} are immediate consequences.
\end{proof}

\begin{prop}
\label{zer}
Suppose that $M_{x}=0$, $M_{y}>0$ and
let $S$ and $T$ be the hitting times of the $x$
and $y$-axis, defined in~(\ref{def_hitting_times}). Then
$\mathbb{P}_{(n_{0},m_{0})}(S = k)$ and
$\mathbb{P}_{(n_{0},m_{0})}(T = k )$ admit the following asymptotic
as $k$ goes to infinity~:
     \begin{eqnarray}
          \mathbb{P}_{\left(n_{0},m_{0}\right)}\left(S = k \right)
          &\sim &\frac{n_{0}m_{0}}{2 \pi \sqrt{p_{10}}\left( p_{01}p_{0-1}\right)^{1/4}}
          \left(\frac{p_{0-1}}{p_{01}}\right)^{m_{0}/2}
          \frac{ \left(2\left(p_{10}+\sqrt{p_{01}p_{0-1}}\right)\right)^{k}}{k^2},
          \label{asymptotic_hitting_times_real_axis_one_drift_positive}\\
          \mathbb{P}_{\left(n_{0},m_{0}\right)}\left(T = k \right)
          &\sim &\frac{n_{0}}{\sqrt{\pi p_{10}}}
          \left(1-\left(\frac{p_{0-1}}{p_{01}}\right)^{m_{0}}\right)
          \frac{1}{\left(2p_{01}\right)^{m_{0}}}
          \frac{1}{k^{3/2}}.
          \label{asymptotic_hitting_times_imaginary_axis_one_drift_positive}
     \end{eqnarray}
\end{prop}

\begin{proof}
The proof
of~(\ref{asymptotic_hitting_times_real_axis_one_drift_positive})
is quite similar to the one
of~(\ref{asymptotic_hitting_times_real_axis_drift_zero})
and the proof
of~(\ref{asymptotic_hitting_times_imaginary_axis_one_drift_positive})
is quite similar to the one
of~(\ref{asymptotic_hitting_times_real_axis_two_drift_positive}).
We omit the details.
\end{proof}

\begin{rem}
{\rm Note that equation~(\ref{asymptotic_hitting_times_real_axis_one_drift_positive})
formally implies~(\ref{asymptotic_hitting_times_real_axis_drift_zero}).
Also,~(\ref{asymptotic_hitting_times_imaginary_axis_one_drift_positive}) formally
follows from~(\ref{asymptotic_hitting_times_real_axis_two_drift_positive}) after
a proper change of the parameters. But one can not
obtain~(\ref{asymptotic_hitting_times_real_axis_one_drift_positive})
starting from~(\ref{asymptotic_hitting_times_real_axis_two_drift_positive})
and then making the drift go to zero.}
\end{rem}

\begin{rem}
{\rm Let $ \tau=\inf\left\{n\in \mathbb{N}:\,
(X\left(n\right),Y\left(n\right))\,\textnormal{hits the boundary}\right\}$
be the hitting time of the boundary of $(\mathbb{Z}_{+})^{2}$.
To find the tail's asymptotic of $\tau$, we can now
apply~(\ref{kkk}).
If at least one of the two drifts~(\ref{drift}) is zero,
then the last term in~(\ref{kkk})
is zero and the result comes immediately.
If both drifts are positive,
then we have to compute the probability of non absorption,
that will be done in
Proposition~\ref{probability_being_absorb} of
Subsection~\ref{Probability_being_absorbed}.}
\end{rem}

\section{Probability of being absorbed in a fixed site}\label{h_1_x}

\subsection{Explicit form and asymptotic}
\label{Explicit_form_and_asymptotic}

We recall from the very beginning of this paper
that taking $z=1$ in $h(x,z)$
(see~(\ref{def_generating_functions})),
leads to
$h(x,1)=\sum_{i=1}^{+\infty}\mathbb{P}_{(n_{0},m_{0})}
(\text{to be absorbed at}\, (i,0))x^{i}$.
In addition, putting $z=1$ in the
explicit expression of $h(x,z)$ obtained
in Proposition~\ref{explicit_h(x,z)_third}
yields
     \begin{equation}\label{explicit_h(x,z)_fourth}
          h\left(x,1\right)=\frac{x}{\pi }
          \int_{x_{3}\left(1\right)}^{x_{4}\left(1\right)}
          \left(t^{n_{0}}-\left(\frac{r^{2}}{t}\right)^{n_{0}}\right)
          \frac{\mu_{m_{0}}\left(t,1\right)\sqrt{-d\left(t,1\right)}}
          {t\left(t-x\right)}\textnormal{d}t+xP_{\infty}\left(x\mapsto
          x^{n_{0}-1}Y_{0}\left(x,1\right)^{m_{0}}\right)\left(x\right).
     \end{equation}
Above, $x_{3}(1)$ and $x_{4}(1)$ are defined in
Lemma~\ref{lemma_branched_points},
$\mu_{m_{0}}$ in~(\ref{def_mu}) and
$P_{\infty}$ in Lemma~\ref{simplification_principal_part}.
We recall about $xP_{\infty}(x\mapsto
x^{n_{0}-1}Y_{0}(x,1)^{m_{0}})(x)$
that it is simply a polynomial,
the null polynomial
if $n_{0}\leq m_{0}$,
of degree $n_{0}-m_{0}$ if
$n_{0}>m_{0}$.
Note that the equality~(\ref{explicit_h(x,z)_fourth})
is viable equally in the cases
$M_{x}>0$, $M_{x}=0$, $M_{y}>0$, $M_{y}=0$.
In particular, we immediately
deduce the explicit expression of the
coefficients $h_{i}=\mathbb{P}_{(n_{0},m_{0})}
(\text{to be absorbed at}\, (i,0))$~:

\begin{prop}
Suppose that $M_{x}\geq 0$, $M_{y} \geq 0$.
Then, for $i\geq \max (n_{0}-m_{0},1)$, the following equality holds~:
     \begin{equation}\label{explicit_probability_absorb_each_site_simple_walk}
          h_{i}=\frac{1}{\pi }
          \int_{x_{3}\left(1\right)}^{x_{4}\left(1\right)}
          \left(t^{n_{0}}-\left(\frac{r^{2}}{t}\right)^{n_{0}}\right)
          \frac{\mu_{m_{0}}\left(t,1\right)\sqrt{-d\left(t,1\right)}}
          {t^{i+1}}\textnormal{d}t.
     \end{equation}
For $i\in \{1,\max\left(n_{0}-m_{0},0\right)\}$, the
equality~(\ref{explicit_probability_absorb_each_site_simple_walk})
is still true if we add
the contribution of
the polynomial
$xP_{\infty}(x\mapsto
x^{n_{0}-1}Y_{0}(x,1)^{m_{0}})(x)$,
defined in
Lemma~\ref{simplification_principal_part}.
\end{prop}

We will now study the asymptotic of $h_{i}$, first
in case of a zero drift, then in case of a
non zero drift. We will see that the decrease of these probabilities
is respectively polynomial and exponential, with an exponential rate
equal to $1/x_{3}(1)$,
what we would have been able to anticipate
from Proposition~\ref{proposition_continuation} of
Subsection~\ref{Analytic_continuation},
where we have seen that $x_{3}(1)$ is the
first positive singularity of $h(x,1)$.

Among other things, we will see that
the asymptotic of $h_{i}$ in case of a drift zero
is not the limit,
when the drift goes to zero, of the
asymptotic in case of
a non zero drift,
thought $x_{3}(1)=1$.

Of course, the calculation of the asymptotic can be deduced
from the explicit
expression~(\ref{explicit_probability_absorb_each_site_simple_walk}),
using e.g. Laplace's method.
However, and since it will be useful later, we prefer,
like in Section~\ref{h_1_z}, deduce this
asymptotic from the study of singularities of the function $h$~;
singularities that will be of
two different types,
namely logarithmic and algebraic,
according to the drift
is zero or positive, see
Propositions~\ref{singularity_h_ln}
and~\ref{singularity_h_square}.

\begin{prop}\label{singularity_h_ln}
Suppose that $M_{y}=0$ and $M_x\geq 0$. The function $h\left(x,1\right)$
admits a singularity of a logarithmic type at $x=1$, where its
development is~:
     \begin{equation*}
          h\left(x,1\right)=h\left(1,1\right)+n_{0}\left(x-1\right)
          \left(1+\left(x-1\right)f\left(x\right)\right)-
          \frac{2n_{0}m_{0}}{\pi}\sqrt{\frac{p_{10}}{p_{01}}}
          \left(x-1\right)^{2}\ln\left(1-x\right)\left(1+\left(x-1\right)
          g\left(x\right)\right),
     \end{equation*}
where $f$ and $g$ are holomorphic in
the neighborhood of $1$, and could
be made explicit from the proof.
\end{prop}

\begin{proof}
The proof is lightly different
according to $n_{0}\leq m_{0}$
or $n_{0}> m_{0}$~;
indeed, as said in
Remark~\ref{rem_principal_part},
in first case the polynomial
$xP_{\infty}(x\mapsto
x^{n_{0}-1}Y_{0}(x,1)^{m_{0}})(x)$
is zero, whereas in second
it is of degree $n_{0}-m_{0}$.
We choose to do the proof in
case $n_{0}\leq m_{0}$,
knowing that in the other case,
it suffices to do
an induction on $n_{0}-m_{0}$
to show that the
Proposition~\ref{singularity_h_ln}
is still valid.

Under this assumption the expression of $h(x,1)$ written
in~(\ref{explicit_h(x,z)_fourth}),
Subsection~\ref{Explicit_form_and_asymptotic}, becomes
     \begin{equation}\label{assumption_n0<=m0}
          h\left(x,1\right)=\frac{x}{\pi}\int_{1}^{x_{4}\left(1\right)}
          \frac{t^{n_{0}}-t^{-n_{0}}}{t\left(t-x\right)}\mu_{m_{0}}\left(t,1\right)
          \sqrt{-d\left(t,1\right)}\text{d}t,
     \end{equation}
so that, using twice that $1/(t-x)=1/(t-1)
+(x-1)/((t-x)(t-1))$,
we get $h(x,1)/x=h(1,1)+(x-1)H_{1}+(x-1)^{2}H_{2}(x)$, where
     \begin{equation*}
          \left\{\begin{array}{ccc}
          \displaystyle H_{1}&=&\displaystyle
          \frac{1}{\pi}\int_{1}^{x_{4}\left(1\right)}
          \left(t^{n_{0}}-\frac{1}{t^{n_{0}}}\right)
          \frac{\mu_{m_{0}}\left(t,1\right)}{t\left(t-1\right)^{2}}
          \sqrt{-d\left(t,1\right)}\text{d}t, \\
          \displaystyle H_{2}\left(x\right)&=&\displaystyle
          \frac{1}{\pi}\int_{1}^{x_{4}\left(1\right)}
          \left(t^{n_{0}}-\frac{1}{t^{n_{0}}}\right)
          \frac{\mu_{m_{0}}\left(t,1\right)}{t\left(t-1\right)^{2}\left(t-x\right)}
          \sqrt{-d\left(t,1\right)}\text{d}t.
          \end{array}\right.
     \end{equation*}
The function $l(t)$, that we define by
$l(t)=(t^{n_{0}}-t^{-n_{0}})\mu_{m_{0}}(t,1)(-d(t,1))^{1/2}/(t(t-1)^{2})$,
which appears in $H_{1}$ and $H_{2}(x)$,
is continuable into a holomorphic
function in the neighborhood of $1$.
Indeed, we recall from
Lemma~\ref{lemma_branched_points}
that $x_{2}(1)=x_{3}(1)=1$, since $M_{y}=0$.
We still note $l(t)$ this continuation and write
$l(t)=\sum_{k=0}^{+\infty} l_{k}(t-1)^k$.
The $l_{k}$ could of course
be calculated, for instance
$l_{0}=2n_{0}\mu_{m_{0}}(1,1)p_{10}((x_{4}(1)
-1)(1-x_{1}(1)))^{1/2}$, that we can simplify by using that
$\mu_{m_{0}}(1,1)=m_{0}/ (2p_{01})$ and
$(x_{4}(1)-1)(1-x_{1}(1))=4p_{01}/p_{10}$,
we finally find $l_{0}=2n_{0}m_{0}(p_{10}/p_{01})^{1/2}$.

We will now study successively
$H_{2}(x)$ and $H_{1}$,
start with $H_{2}(x)$.
We split the integral
$H_{2}(x)$ in two terms~:
$\int_{1}^{1+\epsilon }l(t)/(t-x)\text{d}t+
\int_{1+\epsilon}^{x_{4}(1)}l(t)/(t-x)\text{d}t$,
where $\epsilon\in [0,x_{4}(1)-1]$. The fact that the second
term in the last sum is, as a function of $x$,
holomorphic on the open disc
$\mathcal{D}(0,1+\epsilon )$ is
clear. In addition, it is easily shown that
     \begin{equation}\label{equality_log_binom}
          \int_{1}^{1+\epsilon}\frac{\left(t-1\right)^{k}}
          {t-x}\text{d}t=P_{k}\left(x\right)+\left(x-1\right)^{k}
          \ln\left(\frac{1+\epsilon-x}{1-x}\right),
     \end{equation}
where $P_{0}$ is the null polynomial, and for $k\geq 1$,
$\deg(P_{k})=k-1$ --of course, $P_{k}$ could be
calculated in an explicit way--.
This leads to
     \begin{equation}\label{recap}
          \int_{1}^{1+\epsilon}
          \frac{l\left(t\right)}{t-x}\text{d}t=
          \sum_{k=0}^{+\infty}l_{k}P_{k}\left(x\right)+
          \ln\left(\frac{1+\epsilon-x}{1-x}\right)l\left(x\right).
     \end{equation}
This is here that having split the integral in two terms turns out
to be useful~: if we had left $x_{4}(1)$ as the upper bound of the
integral, it would have been quite possible that the function 
$\sum_{k}l_{k}P_{k}$ does not exist~: indeed, the radius of convergence of
$l$ is equal to $\inf\{1-x_{1}(1),x_{4}(1)-1\}$ and for $k\geq 1$,
$P_{k}(1)=\epsilon^k/k$ --as we show by taking $x=1$
in~(\ref{equality_log_binom}) for $k\geq 1$--. However, for
sufficiently small values of $\epsilon$, the function 
$\sum_{k}l_{k}P_{k}$ exists well and truly.

We have thus showed
that $H_{2}(x)$ is the sum of a function
holomorphic at $1$ and of a function having
at $1$ a logarithmic singularity,
see~(\ref{equality_log_binom})
and~(\ref{recap}).

To complete the proof of Lemma~\ref{singularity_h_ln},
it remains to study the term $H_{1}$, and in
particular to show that $H_{1}+h(1,1)=n_{0}$.
We recall that we have supposed $n_{0}\leq m_{0}$, so that
differentiating~(\ref{assumption_n0<=m0}) and taking $x=1$
yields~:
     \begin{equation*}
          \partial_{x}h\left(1,1\right)=H_{1}+h\left(1,1\right)=
          \frac{1}{\pi}\int_{1}^{x_{4}\left(1\right)}
          \frac{t^{n_{0}}-t^{-n_{0}}}{\left(t-1\right)^{2}}\mu_{m_{0}}\left(t,1\right)
          \sqrt{-d\left(t,1\right)}\text{d}t.
     \end{equation*}
After having made the change of variable $t=t_{2}\left(u,1\right)$,
see~(\ref{natural_expression_mu}),~(\ref{tttt})
and~(\ref{kk}), and after some simplifications, we
find~:
     \begin{equation*}
          \partial_{x}h\left(1,1\right)=\frac{p_{10}}{\pi}\int_{-1}^{1}
          \frac{t_{2}\left(u,1\right)^{n_{0}}-t_{1}\left(u,1\right)^{n_{0}}}
          {\sqrt{\left(1-k_{1}\left(u\right)\right)\left(1-k_{2}\left(u\right)\right)}}
          U_{m_{0}-1}\left(-u\right)\sqrt{\frac{1-u}{1+u}}\text{d}u.
     \end{equation*}
Using the explicit expressions of $t_{1}$ and $t_{2}=1/t_{1}$
written in~(\ref{tttt}), we notice that
$(t_{2}(u,1)^{n_{0}}-t_{1}(u,1)^{n_{0}})/
((1-k_{1}(u))(1-k_{2}(u)))^{1/2}$ is in fact a polynomial of degree
$n_{0}-1$, that we note $P_{n_{0}-1}(u)$. Moreover, it
turns out that $P_{n_{0}-1}(-1)=n_{0}/p_{10}$. Define now
$Q_{n_{0}-2}(u)$ the $n_{0}-2$ degree polynomial defined
by $P_{n_{0}-1}(u)=
P_{n_{0}-1}(-1)+(u+1)Q_{n_{0}-2}(u)$.
With these notations,
     \begin{equation*}
          \partial_{x}h\left(1,1\right)=\frac{n_{0}}{\pi}\int_{-1}^{1}
          U_{m_{0}-1}\left(-u\right)\sqrt{\frac{1-u}{1+u}}\text{d}u+
          \frac{p_{10}}{\pi} \int_{-1}^{1}Q_{n_{0}-2}\left(u\right)
          U_{m_{0}-1}\left(-u\right)\sqrt{1-u^{2}}\text{d}u.
     \end{equation*}
The second term in the sum above is null. Indeed, being the
$(m_{0}-1)$-th orthogonal polynomial associated to the weight
$1_{]-1,1[}(u)(1-u^{2})^{1/2}$, $U_{m_{0}-1}$ is such that
$\int_{-1}^{1}U_{m_{0}-1}(u)
P(u)(1-u^{2})^{1/2}\text{d}u=0$ for all polynomial $P$ whose
the degree is less or equal than $m_{0}-2$, that is actually the
case for $Q_{n_{0}-2}$ since we have supposed that $n_{0}\leq
m_{0}$.

As for the first term in the sum above, we show, using induction
and the recurrence relationship verified by the
Chebyshev polynomials, namely $U_{m_{0}+1}(u)=
2uU_{m_{0}}(u)-U_{m_{0}-1}(u)$, see~\cite{SZ},
that for all $m_{0}\in \mathbb{N}^{*}$,
$\int_{-1}^{1}U_{m_{0}-1}(-u)((1-u)/(1+u))^{1/2}\text{d}u=\pi$.
\end{proof}

\begin{prop}\label{singularity_h_square}
Suppose that $M_{y}>0$ and $M_x \geq 0$. The function $h\left(x,1\right)$ admits
a singularity of an algebraic type at $x=x_{3}\left(1\right)$, where
its development is~:
     \begin{equation*}
          h\left(x,1\right)=f\left(x\right)
          +\sqrt{1-x/x_{3}\left(1\right)}g\left(x\right),
     \end{equation*}
where $f$ and $g$ are holomorphic in
the neighborhood of $x_{3}(1)$, and could
be made explicit from the proof~; in particular,
\begin{equation*}
          g\left(x_{3}\left(1\right)\right)=-\left(x_{3}\left(1\right)^{n_{0}}-
          x_{2}\left(1\right)^{n_{0}}\right)
          \left(p_{10}\left(x_{3}\left(1\right)-x_{2}\left(1\right)\right)\right)^{1/2} m_{0}
          \left(\frac{p_{0-1}}{p_{01}}\right)^{m_{0}/2}\frac{1}
          {\left(p_{01}p_{0-1}\right)^{1/4}}.
\end{equation*}
\end{prop}

\begin{proof}
Using the equality $1/(t-x)=1/(t-x_{3}(1))+(x-x_{3}(1))/((t-x)(t-x_{3}(1)))$
in~(\ref{explicit_h(x,z)_fourth}) and setting
temporarily $l(t)=(t^{n_{0}}-(r^{2}/t)^{n_{0}})\mu_{m_{0}}(t,1)(-p_{10}^{2}(t-x_{1}(1))
(t-x_{2}(1))(t-x_{3}(1)))^{1/2}/t$ we obtain~:
     \begin{equation*}
          h\left(x,1\right)=\frac{x}{x_{3}\left(1\right)}
          h\left(x_{3}\left(1\right),1\right)+
          \frac{x\left(x-x_{3}\left(1\right)\right)}{\pi}
          \int_{x_{3}\left(1\right)}^{x_{4}\left(1\right)}
          \frac{l\left(t\right)}{\left(t-x\right)\sqrt{t-x_{3}\left(1\right)}}
          \text{d}t+P\left(x\right),
     \end{equation*}
where $P(x)$ is a polynomial, null at $x_{3}(1)$,
obtained from $xP_{\infty}(x\mapsto x^{n_{0}-1}Y_{0}(x,1)^{m_{0}})(x)$.
But we can easily find the singularities of the following
Cauchy type integral, see~\cite{LU}~:
     \begin{equation*}
          \int_{x_{3}\left(1\right)}^{x_{4}\left(1\right)}
          \frac{1}{\left(t-x\right)\sqrt{t-x_{3}\left(1\right)}}\text{d}t=
          \frac{\pi}{\sqrt{x_{3}\left(1\right)-x}}\left(1+
          \left(x-x_{3}\left(1\right)\right)u\left(x\right)\right),
     \end{equation*}
where $u$ is a function
holomorphic in the neighborhood of $x_{3}(1)$.
Making an expansion of $l(t)-l(x_{3}(1))$ in the
neighborhood of $x_{3}(1)$ and with a repeated use
of $1/(t-x)=1/(t-x_{3}(1))+(x-x_{3}(1))/((t-x)(t-x_{3}(1)))$, we get~:
     \begin{equation*}
          \int_{x_{3}\left(1\right)}^{x_{4}\left(1\right)}
          \frac{l\left(t\right)-l\left(x_{3}\left(1\right)\right)}
          {\left(t-x\right)\sqrt{t-x_{3}\left(1\right)}}\text{d}t=
          c+v\left(x\right)\sqrt{x_{3}\left(1\right)-x},
     \end{equation*}
where $c$ is some constant, $v$ some function
holomorphic at $x_{3}(1)$.
Thus, Proposition~\ref{singularity_h_square}
will be proved as soon as we will have made
explicit $g(x_{3}(1))$.
Before any simplifications,
we have $g(x_{3}(1))=l(x_{3}(1))x_{3}(1)^{3/2}$.
To simplify this quantity,
note that $(x_{3}(1)-x_{1}(1))(x_{4}(1)-x_{3}(1))=4(p_{01}p_{0-1})^{1/2}x_{3}(1)/p_{10}$
and that $\mu_{m_{0}}(x_{3}(1),1)=(p_{0-1}/p_{01})^{(m_{0}-1)/2}m_{0}/(2p_{01}x_{3}(1))$,
where the announced value of $g(x_{3}(1))$ comes from.
\end{proof}

Propositions~\ref{singularity_h_ln}
and~\ref{singularity_h_square} allow to derive easily the
asymptotic of the absorption probabilities~:

\begin{prop}\label{singularity_h_square_asymptotic}
Suppose that $M_{y}=p_{01}-p_{0-1}=0$.
The probability of being absorbed at $\left(i,0\right)$ admits the
following asymptotic as $i\to +\infty $~:
     \begin{equation*}
          \mathbb{P}_{\left(n_{0},m_{0}\right)}
          \left(\textnormal{to be absorbed at}\left(i,0\right)\right)
          \sim \frac{4}{\pi}\sqrt{\frac{p_{10}}{p_{01}}}n_{0}m_{0}\frac{1}{i^{3}}.
     \end{equation*}
\end{prop}

%
%
%
%

\begin{prop}\label{singularity_h_ln_asymptotic}
Suppose that $M_{y}=p_{01}-p_{0-1}>0$.
The probability of being absorbed at $\left(i,0\right)$ admits the
following asymptotic as $i\to +\infty $~:
     \begin{equation*}
          \mathbb{P}_{\left(n_{0},m_{0}\right)}
          \left(\textnormal{to be absorbed at}\left(i,0\right)\right)
          \sim \frac{\sqrt{p_{10}\left(x_{3}\left(1\right)-
          x_{2}\left(1\right)\right)}}
          {2\sqrt{\pi}\left(p_{01}p_{0-1}\right)^{1/4}}
          m_{0}\left( \frac{p_{0-1}}{p_{01}} \right)^{m_{0}/2}
          \left(x_{3}^{n_{0}}-x_{2}^{n_{0}}\right)
          \frac{1}{i^{3/2}x_{3}\left(1\right)^{i}}.
     \end{equation*}
\end{prop}

%
%
%
%

\begin{proof}
These two propositions are corollaries from 
Propositions~\ref{singularity_h_ln} and~\ref{singularity_h_square},
following the principles giving
the way to obtain the asymptotic
of the coefficients of a Taylor series at zero
starting from the knowledge of
the first singularity,
principles explained in the proof of
Proposition~\ref{main_result_simple_walk_drift_zero}
for a logarithmic
singularity, in the one of
Proposition~\ref{zek}
for an algebraic singularity.
\end{proof}

\subsection{Green functions associated to some sets in the case of a drift zero}

In this part, we suppose that the two drifts $M_{x}$ and $M_{y}$ are
equal to zero. Define, for $a,k\in \mathbb{Z}_{+}$,
$\Gamma_{a,k}=\{(i,j)\in (\mathbb{Z}_{+})^{2} :
i-1+a(j-1)=k\}$ and denote by $G_{\Gamma_{a,k}}$ the Green function
associated to $\Gamma_{a,k}$, in other words the mean number of
visits of the walk in $\Gamma_{a,k}$. Note that $G_{\Gamma_{a,k}}$
is connected with the Green functions $G_{i,j}$ via
$G_{\Gamma_{a,k}}=\sum_{i-1+a(j-1)=k}G_{i,j}$
\begin{prop}\label{asymptotic_G_Gamma}
The following asymptotic holds as $k\to +\infty $~:
     \begin{equation*}
          G_{\Gamma_{a,k}}\sim \frac{2n_{0}m_{0}}
          {\sqrt{p_{10}p_{01}}k}.
     \end{equation*}
\end{prop}

\begin{proof}
We start by remarking that for $a\in \mathbb{Z}_{+}$,
     \begin{equation*}
          G\left(x,x^{a},1\right)=\sum_{i,j\geq 1}G_{i,j}x^{i-1+a\left(j-1\right)}
          =\sum_{k=0}^{+\infty}x^{k}\sum_{\Gamma_{a,k}}G_{i,j}.
     \end{equation*}
Besides, Equation~(\ref{functional_equation}) gives $G(x,x^{a},1)=
(h(x,1)+\tilde{h}(x^{a},1)-x^{n_{0}+am_{0}})/Q(x,x^{a},1)$.
Also, applied to $h$ and $\tilde{h}$, Proposition~\ref{singularity_h_ln} leads to~:
     \begin{equation*}
          h\left(x,1\right)+\widetilde{h}\left(x^{a},1\right)-x^{n_{0}+am_{0}}=
          \frac{-2n_{0}m_{0}}{\pi}\left(\sqrt{\frac{p_{10}}{p_{01}}}+a^{2}
          \sqrt{\frac{p_{01}}{p_{10}}}\right)\ln\left(1-x\right)
          \left(1+\left(x-1\right)l_{1}\left(x\right)\right),
     \end{equation*}
where $l_{1}$ is holomorphic at $x=1$.
Moreover, an easy calculation yields
$Q\left(x,x,1\right)=x\left(x-1\right)^{2}/2$~; more generally, for any $a>0$,
$Q\left(x,{x}^{a},1\right)=\left(x-1\right)^{2}P_{a}\left(x\right)$ where
     \begin{equation*}
          P_{a}\left(x\right)=p_{01}x\left(\sum_{k=1}^{a-1}k
          \left(x^{k-1}+x^{2a-1-k}\right)+
          \left(a+\frac{p_{10}}{p_{01}}\right)x^{a-1}\right),
     \end{equation*}
In particular, $P_{a}\left(1\right)=p_{01}a^{2}+p_{10}$. Thus, we obtain that
$G(x,x^{a},1)=-c\ln(1-x)(1+(x-1)l_{2}(x))$ where
$l_{1}$ is holomorphic at $x=1$ and~:
     \begin{equation*}
          c=\frac{2n_{0}m_{0}}{\pi P_{a}\left(1\right)}
          \left(\sqrt{\frac{p_{10}}{p_{01}}}+a^{2}
          \sqrt{\frac{p_{01}}{p_{10}}}\right)=
          \frac{2n_{0}m_{0}}{\sqrt{p_{10}p_{01}}}.
     \end{equation*}
Then, Proposition~\ref{asymptotic_G_Gamma} follows
from the from now on usual way
to obtain the asymptotic of coefficients of
a function
starting from the study of its singularities,
see the proof of
Proposition~\ref{main_result_simple_walk_drift_zero}
for more details.
\end{proof}

We can remark two facts about this asymptotic. First,
as the result of Proposition~\ref{asymptotic_G_Gamma} shows,
the asymptotic of $G_{\Gamma_{a,k}}$ does not depend on $a>0$.
Secondly, this result is in fact also true for $a=0$. To
show this fact, we have
to adapt a little the proof since the explicit expression
of $P_{a}\left(x\right)$ is no more valid~; to
overcome that, we just have to use the equality
$Q(x,1,1)=p_{10}(x-1)^{2}$, the proof
is then the same as before.

\subsection{Probability of being absorbed}\label{Probability_being_absorbed}

In this subsection, we give a nice explicit expression of the
probability for the walk to be absorbed on the boundary.
This explicit formulation, obtained in
Proposition~\ref{probability_being_absorb},
allows us to find again the well known fact
that when at least one of the two drifts~(\ref{drift})
is zero, then the walk is almost surely
absorbed.


\begin{prop}
\label{probability_being_absorb}
The probability of being absorbed is equal to~:
     \begin{equation*}
          h\left(1,1\right)+\widetilde{h}\left(1,1\right)=
          1-\left(1-\left(\frac{p_{-10}}{p_{10}}\right)^{n_{0}}\right)
          \left(1-\left(\frac{p_{0-1}}{p_{01}}\right)^{m_{0}}\right).
     \end{equation*}
\end{prop}

\begin{proof}
We will use the
equality~(\ref{link_explicit_constant})
of Lemma~\ref{prop_link} below.
In fact, the right member
of~(\ref{link_explicit_constant}) can be simplified,
but according to the location of $y$ in the complex plane,
the simplification will not be the same. Suppose for instance
that $|y|>(p_{0-1}/p_{01})^{1/2}(=\tilde{r})$, then
Lemma~\ref{properties_X_Y} gives that
$Y_{0}(X_{0}(y,z),z)=Y_{0}(X_{1}(y,z),z)=p_{0-1}/(p_{01}y)$,
so that~(\ref{link_explicit_constant}) becomes~:
     \begin{equation}\label{link_explicit_constant_applied}
          h\left(X_{1}\left(y,z\right),z\right)+\widetilde{h}\left(y,z\right)
          -X_{1}\left(y,z\right)^{n_{0}}y^{m_{0}}=
          \left(y^{m_{0}}-\left(\frac{\widetilde{r}^{2}}{y}\right)^{m_{0}}\right)
          \left(X_{0}\left(y,z\right)^{n_{0}}-\left(\frac{r^{2}}
          {X_{0}\left(y,z\right)}\right)^{n_{0}}\right).
     \end{equation}
Then, Proposition~\ref{probability_being_absorb} follows
immediately from~(\ref{link_explicit_constant_applied}),
taking $y=1\geq \tilde{r}$ and using the facts
that $X_{0}(1,1)=r^{2}$ and $X_{1}(1,1)=1$,
seen in Lemma~\ref{properties_X_Y}.
\end{proof}

\begin{lem}\label{prop_link}
Suppose that $z\in ]0,z_{1}]$ and
$y\in \mathbb{C}\setminus [x_{1}(z),x_{2}(z)]\cup [x_{3}(z),x_{4}(z)]$.
The functions $h$ and $\tilde{h}$
are connected by~:
     \begin{eqnarray}
          \widetilde{h}\left(y,z\right)&=&X_{0}\left(y,z\right)^{n_{0}}
          y^{m_{0}}-h\left(X_{0}\left(y,z\right),z\right),\label{link_continuation}\\
          \widetilde{h}\left(y,z\right)&=&X_{0}\left(y,z\right)^{n_{0}}
          y^{m_{0}}+X_{1}\left(y,z\right)^{n_{0}}
          Y_{0}\left(X_{1}\left(y,z\right),z\right)^{m_{0}}\label{link_explicit_constant}\\
          &-&X_{0}\left(y,z\right)^{n_{0}}
          Y_{0}\left(X_{0}\left(y,z\right),z\right)^{n_{0}}
          -h\left(X_{1}\left(y,z\right),z\right).\nonumber
     \end{eqnarray}
\end{lem}

\begin{rem}
{\rm (i) The equalities~(\ref{link_continuation})
and~(\ref{link_explicit_constant})
could be obtained with the
procedure of continuation of the functions
$h$ and $\tilde{h}$
explained in~\cite{FIM},
procedure briefly recalled in the proof of 
Proposition~\ref{prop_continuation_Delta_zero}.
Here, we have chosen to not firm up all
the details of this procedure, and we show
how we can find again these equalities
using only the explicit expressions of $h$ and
$\tilde{h}$. (ii) Equality~(\ref{link_continuation})
could be used to
continue $\tilde{h}$, since, as it is proved
in Lemma~\ref{properties_X_Y},
$|X_{0}(y,z)|\leq (p_{-10}/p_{10})^{1/2}\leq 1$,
and thus $h(X_{0}(y,z),z)$ is well defined.
(iii) On the other hand,
it is quite possible that $|X_{1}(y,z)|\geq 1$, see
once again Lemma~\ref{properties_X_Y}, so that
in~(\ref{link_explicit_constant}), $h(X_{1}(y,z),z)$
has to be defined using its analytic continuation,
established here in Proposition~\ref{proposition_continuation}.
(iv) If we make the difference of
equations~(\ref{link_continuation})
and~(\ref{link_explicit_constant}),
we find again the boundary condition
that verifies $h$,
see~(\ref{SR_problem_h}) of
Subsection~\ref{Riemann_Carleman_problem}.}
\end{rem}

\noindent{\it Proof of Lemma~\ref{prop_link}.}
Lemma~\ref{prop_link} will follow from
a suitable change of variable
in the integral expressions of
$h$ and $\tilde{h}$
obtained in Proposition~\ref{explicit_h(x,z)_second}.
The change
of variable $t=Y_{0}(u,z)$, or,
equivalently here, $u=X_{0}(t,z)$, gives
     \begin{eqnarray}
          & &\frac{y}{\pi}\int_{y_{1}\left(z\right)}^{y_{2}\left(z\right)}
          t^{m_{0}}\widetilde{\mu}_{n_{0}}\left(t,z\right)\left(
          \frac{1}{t\left(t-y\right)}+\frac{1}{t y-\widetilde{r}^{2}}\right)
          \sqrt{-\widetilde{d}\left(t,z\right)}\text{d}t=\nonumber \\
          & &\frac{y}{2\pi i}\int_{\mathcal{C}\left(0,r\right)}
          u^{n_{0}}Y_{0}\left(u,z\right)^{m_{0}-1}\left(
          \frac{1}{Y_{0}\left(u,z\right)-y}+\frac{Y_{0}\left(u,z\right)}
          {Y_{0}\left(u,z\right) y-\widetilde{r}^{2}}\right)
          \partial_{u}Y_{0}\left(u,z\right)\text{d}u\label{equality_change_t=Y0},
     \end{eqnarray}
since $X_{0}\big(\big[\overrightarrow{\underleftarrow{y_{1}(z),y_{2}(z)}}\big],z\big)
=\mathcal{C}(0,r)$, as proved
in Subsection~\ref{Riemann_Carleman_problem}.
In addition, an immediate
consequence of the definition of $Q(u,y,z)$
is the equality
$(Y_{0}(u,z)-y)
(Y_{1}(u,z)-y)=(u-X_{0}(y,z))
(u-X_{1}(y,z))p_{10}y/(p_{01}u)$~; also, since
$Y_{0}(u,z)Y_{1}(u,z)=\tilde{r}^{2}$,
$(Y_{0}(u,z) y -\tilde{r}^2)(Y_{1}(u,z) y -\tilde{r}^2)
=\tilde{r}^2(Y_{0}(u,z)-y)
(Y_{1}(u,z)-y)$. Then, using that
$\partial_{u}Y_{0}(u,z) = z p_{10}(r^{2}-u^{2})
Y_{0}(u,z)/ (u d(u,z)^{1/2})$, we find
that~(\ref{equality_change_t=Y0}) is equal to~:
\begin{equation*}
\frac{y}{2\pi i}\int_{\mathcal{C}\left(0,r\right)}
          \frac{u^{n_{0}}Y_{0}\left(u,z\right)^{m_{0}}\left(r^{2}-u^{2}\right)}
          {u\left(u-X_{0}\left(y,z\right)\right)
          \left(u-X_{1}\left(y,z\right)\right)}\text{d}u,
\end{equation*}
which, in turn,
thanks to the equality
$(r^{2}-u^{2})/(u(u-x)(u-r^{2}/x))=-x(1/(u x-r^{2})+1/(u(u-x)))$,
is equal to~:
     \begin{eqnarray*}
          & &\frac{-X_{0}\left(y,z\right)}{2\pi i}\int_{\mathcal{C}\left(0,r\right)}
          u^{n_{0}}Y_{0}\left(u,z\right)^{m_{0}}\left(\frac{1}{uX_{0}\left(y,z\right)-r^{2}}
          +\frac{1}{u\left(u-X_{0}\left(y,z\right)\right)}\right)\text{d}u\\
          &=&\frac{-X_{1}\left(y,z\right)}{2\pi i}\int_{\mathcal{C}\left(0,r\right)}
          u^{n_{0}}Y_{0}\left(u,z\right)^{m_{0}}\left(\frac{1}{u X_{1}\left(y,z\right)-r^{2}}
          +\frac{1}{u\left(u-X_{1}\left(y,z\right)\right)}\right)\text{d}u.
     \end{eqnarray*}
The first equality above
will lead to~(\ref{link_continuation}),
the second
to~(\ref{link_explicit_constant}). Indeed,
we use the residue theorem on the contour drawn
in Figure~\ref{pacman} for both
integrals above, we obtain
a residue part, equal in both cases
to $X_{0}(y,z)^{n_{0}}y^{m_{0}}$
and also a algebraic part,
that is to say an integral
on $[x_{1}(z),x_{2}(z)]$.
Then, using the explicit expressions
of $h$ and $\tilde{h}$ written
in Proposition~\ref{explicit_h(x,z)_second}, we conclude.
We have omitted some details because they are similar to
those present in the proofs of
the Propositions~\ref{explicit_h(x,z)}
and~\ref{explicit_h(x,z)_second}.
\hfill $\square $

\section{Asymptotic of Green functions. Martin boundary.}\label{Martin_boundary}

In this section, we will be interested in the asymptotic of
Green functions $G_{i,j}^{n_{0},m_{0}}$~; we recall that
     \begin{equation*}
          G_{i,j}^{n_{0},m_{0}}=\mathbb{E}_{\left(n_{0} , m_{0}\right)}
          \left[\sum_{k=0}^{+\infty }1_{\left\{\left(X\left(k\right),
          Y\left(k\right)\right)=\left(i,j \right)\right\}}\right].
     \end{equation*}

\begin{prop}\label{Martin_boundary_drift_zero}
Suppose that $M_{x}=M_{y}=0$.
The Green functions $G_{i,j}^{n_{0},m_{0}}$ admit the following
asymptotic when $i,j\to +\infty$, $j/i\to \tan(\gamma )\in
[0,+\infty ]$~:
     \begin{equation*}
          G_{i,j}^{n_{0},m_{0}}\sim
          \frac{4\sqrt{p_{01}p_{10}}}{\pi}
          n_{0}m_{0}\frac{i j}{\left(p_{01}i^{2}
          +p_{10}j^{2}\right)^{2}}.
     \end{equation*}
\end{prop}

\begin{proof}
We will prove Proposition~\ref{Martin_boundary_drift_zero}
in the case of $\gamma \in [0,\pi/2[$. The result
remains the same if we exchange $i$ in $j$ and 
simultaneously $p_{01}$ in $p_{10}$,
so that from the result corresponding to
$j/i\to 0$ we easily deduce
the one for $j/i\to \infty $.

In Subsection~\ref{Introduction} we have already
seen that $(x,y)\mapsto G(x,y,1)$ is holomorphic
in $\mathcal{D}(0,1)^{2}$~;
as a consequence the Cauchy formulas allow to write
$G_{i,j}^{n_{0},m_{0}}$ as the following double integrals~:
     \begin{equation*}
          G_{i,j}^{n_{0},m_{0}}=\frac{1}{\left(2\pi \imath \right)^{2}}
          \iint_{\substack{\left|x\right|=1-\epsilon  \\\left|y\right|=1-\epsilon }}
          \frac{G\left(x,y,1\right)}{x^{i}y^{j}}\text{d}x\text{d}y
          = \frac{1}{\left(2\pi \imath \right)^{2}}
          \iint_{\substack{\left|x\right|=1-\epsilon  \\\left|y\right|=1-\epsilon }}
          \frac{h\left(x,1\right)+\widetilde{h}
          \left(y,1\right)-x^{n_{0}}y^{m_{0}}}
          {Q\left(x,y,1\right)x^{i}y^{j}}\text{d}x\text{d}y,
     \end{equation*}
where $\epsilon \in ]0,1[$ and $\imath ^{2}=-1$. The second equality above comes from
Equation~(\ref{functional_equation}) where we have taken $z=1$.
In this way, we can write $G_{i,j}$ as the sum
$G_{i,j}=G_{i,j,1}(\epsilon)+G_{i,j,2}(\epsilon)+G_{i,j,3}(\epsilon)$ where~:
     \begin{eqnarray*}
          G_{i,j,1}(\epsilon) &=& \frac{1}{\left(2\pi \imath \right)^{2}}
          \iint_{\left|x\right|=\left|y\right|=1-\epsilon}
          \frac{h\left(x,1\right)-h\left(X_{1}\left(y,1\right),1\right)}
          {Q\left(x,y,1\right)x^{i}y^{j}}\text{d}x\text{d}y, \\
          G_{i,j,2}(\epsilon) &=& \frac{1}{\left(2\pi \imath \right)^{2}}
          \iint_{\left|x\right|=\left|y\right|=1-\epsilon}
          \frac{X_{1}\left(y,1\right)^{n_{0}}
          y^{m_{0}}-x^{n_{0}}y^{m_{0}}}
          {Q\left(x,y,1\right)x^{i}y^{j}}\text{d}x\text{d}y, \\
          G_{i,j,3}(\epsilon)  &=& \frac{1}{\left(2\pi \imath \right)^{2}}
          \int_{\left|y\right|=1-\epsilon}
          \frac{h\left(X_{1}\left(y,1\right),1\right)+
          \widetilde{h}\left(y,1\right)-X_{1}\left(y,1\right)^{n_{0}}
          y^{m_{0}}}{y^{j}}\int_{\left|x\right|=1-\epsilon}
          \frac{1}{x^{i}Q\left(x,y,1\right)}\text{d}x\text{d}y,
     \end{eqnarray*}
and we will successively study these three integrals.
But before going into details, let us explain
the ideas of the proof and also make two technical remarks.

We will first transform the double integrals
defining each of the $G_{i,j,k}(\epsilon)$ into single integrals,
to that purpose we will apply at the $G_{i,j,k}(\epsilon)$
the residue theorem and use the remark~{\rm (i)}
below. Then, we will take the limit of these quantities
as $\epsilon$ goes to zero~; in this way we will obtain that for all
$i$ and $j$, $G_{i,j,1}(\epsilon) \to G_{i,j,1}$,
$G_{i,j,2}(\epsilon) \to 0$ and 
$G_{i,j,3}(\epsilon) \to G_{i,j,3}$ as $\epsilon\to 0$,
$G_{i,j,1}$ and $G_{i,j,3}$ being defined 
in~(\ref{def_G_i_j_1}) and~(\ref{def_G_i_j_3}).

It will therefore remain to study
$G_{i,j,1}$ and $G_{i,j,3}$, that are
integrals on the unit circle.
It can be easily shown that the 
modulus of functions $x^{i}Y_{1}(x,1)^{j}$ and
$X_{1}(y,1)^{i}y^{j}$, that appear 
in~(\ref{def_G_i_j_1}) and~(\ref{def_G_i_j_3}),
have a strict maximum value, equal to $1$ and reached for $x=1$ and $y=1$ respectively. 
Moreover, the exponents $i$ and $j$ will go to infinity, 
so that applying a steepest descent method could be interesting~; 
not exactly the usual steepest descent method --that can be 
found e.g.\ in~\cite{FED}-- for the notable reason
that the functions considered in
integrals~(\ref{def_G_i_j_1}) and~(\ref{def_G_i_j_3}) 
are not holomorphic at 
the ``saddle point'' $1$ but have there singularities. 
Indeed, we recall from 
Subsections~\ref{The_algebraic_curve_Q}
and~\ref{Explicit_form_and_asymptotic} that
$X$, $Y$, $h$ and $\tilde{h}$ are not holomorphic
at $1$~: $X$ (resp.\ $Y$) is not holomorphic
in the neighborhood of $[y_{1}(1),y_{4}(1)]$
(resp.\ $[x_{1}(1),x_{4}(1)]$) and 
$h$, $\tilde{h}$ have a logarithmic singularity at $1$.

Nevertheless, we will be able to find
a steepest descent path for the function $\chi_{j/i}(x)=\ln(xY_{1}(x,1)^{j/i})$
on both sides of the $x$-axis, which means that we will construct
a function $x_{j/i}$, defined on a neighborhood of
$0$ and having a singularity at $0$, such that
$\chi_{j/i}(x_{j/i}(t))=|t|$~;
likewise we will find a steepest descent path for the function 
$\tilde{\chi}_{j/i}(y)= \ln(X_{1}(y,1)y^{j/i})$ on both sides 
of the $y$-axis.

Next, using Cauchy theorem,
we will move the contours of 
integrals~(\ref{def_G_i_j_1}) and~(\ref{def_G_i_j_3}),
equal initially to the unit circle, into contours
that coincide in the neighborhood of $1$
with the steepest descent path and that elsewhere
remain outside some proper level lines
of $\chi_{j/i}$ and $\tilde{\chi}_{j/i}$.

The way to find the asymptotic of $G_{i,j,1}$
and $G_{i,j,3}$ is then classical.
We start by splitting the integrals defining them into two
parts, one on the steepest descent path, the other one
on the remaining part of the contour and next we show 
that the contribution of this remaining part is,
for $G_{i,j,1}$ and for $G_{i,j,3}$, exponentially
negligible, which means that we can find two positive constants,
say $c$ and $\rho $,
such that the quantity $c\exp(-\rho i)$ is an upper bound
of these integrals~;
we will even show that we can take $c$ and $\rho $ independent
of $j/i$, provided that $j/i$ remains in some compact.
As for the integral on the steepest descent path,
we will show that it is identically zero in case of
$G_{i,j,1}$, thanks to an interesting 
relationship between the steepest descent contours 
proved in remark~{\rm (ii)}~;
in case of $G_{i,j,3}$ we will prove that its contribution
leads to the result announced in Proposition~\ref{Martin_boundary_drift_zero},
making a precise study of the behavior of the integrand in
the neighborhood of $1$.

Before beginning the study of the $G_{i,j,k}(\epsilon)$, $k=1,2,3$,
we make the two remarks mentioned above and 
that will be useful there.
\medskip

{\rm (i)}  For any $\epsilon >0$ (and $<1-x_{1}\left(1\right)$)
           and any $\left|y\right|=1-\epsilon $,
           $\left|X_{1}\left(y,1\right)\right|\geq 1$ and
           $\left|X_{0}\left(y,1\right)\right|\leq 1$, as we already know
           from Lemma~\ref{properties_X_Y}. Moreover, there exists
           a function $\theta_{0}=\theta_{0}\left(\epsilon\right)$,
           continuous and going to zero as $\epsilon $ goes to zero, such that~:
           \begin{itemize}
                \item if $y=\left(1-\epsilon\right)\exp\left(\imath \theta\right)$ with
                      $-\theta_{0}\left(\epsilon\right)< \theta <
                      \theta_{0}\left(\epsilon\right)$ then
                      $\left|X_{0}\left(y,1\right)\right|> 1-\epsilon $,
                \item if $y=\left(1-\epsilon\right)\exp\left(\imath \theta\right)$ with
                      $\theta \in \left]\theta_{0}\left(\epsilon\right),
                      2\pi -\theta_{0}\left(\epsilon\right)\right[$ then
                      $\left|X_{0}\left(y,1\right)\right|< 1-\epsilon $.
           \end{itemize}
           Of course, we can also define a function $\tilde{\theta}(\epsilon)$
           that is the analogous of $\theta(\epsilon)$ for the function $Y$.

\medskip
{\rm (ii)}
          Consider the two functions $\chi_{j/i}(x)=
          \ln(xY_{1}(x,1)^{j/i})$ and $\tilde{\chi}_{j/i}(y)=
          \ln(X_{1}(y)y^{j/i})$. One can find $\eta>0$, independent of $j/i\in [0,M]$, 
          $M>0$ being fixed but as large as wished, such that
          $\chi_{j/i}$ and $\tilde{\chi}_{j/i}$ are holomorphic in
          $\mathcal{D}(1,\eta)\cap \{s\in \mathbb{C} : \text{Im}(s)>0\}$
          and in $\mathcal{D}(1,\eta)\cap \{s\in \mathbb{C} : \text{Im}(s)<0\}$~;
          moreover, they are continuable at $1$, where they take the value $0$.
          These continuations are, at $1$, continuous, but not differentiable, let alone
          holomorphic. Consider now the functions $x_{j/i}(t)$ and $y_{j/i}(t)$
          defined by
               \begin{equation}\label{def_chi_chi_tilde}
                     \chi_{j/i}\left(x_{j/i}\left(t\right)\right)=\left|t\right|,\hspace{5mm}
                     \widetilde{\chi}_{j/i}\left(y_{j/i}\left(t\right)\right)=\left|t\right|,
               \end{equation}
          and by $\text{sign}(\text{Im}(x_{j/i}(t)))=\text{sign}(t)$
          and $\text{sign}(\text{Im}(y_{j/i}(t)))=\text{sign}(t)$.
          These last conditions concerning the sign of the imaginary parts allow in fact to define
          $x_{j/i}$ and $y_{j/i}$ not ambiguously, indeed, Equation~(\ref{def_chi_chi_tilde})
          doesn't suffice to determine them, since $|t|=|-t|$.

          We will now be interested in the properties of $x_{j/i}$ and $y_{j/i}$.

          As a function of a real variable, $t\mapsto |t|$ has clearly
          no series expansion in the neighborhood of $0$,
          but is equal, on both sides of $0$, to functions which have a series expansion at $0$,
          namely $t\mapsto t$ for $t>0$ and $t\mapsto -t$ for $t<0$. The same happens
          for $x_{j/i}$ and $y_{j/i}$~: they are not holomorphic at $0$ but are equal to
          holomorphic functions on either side of $0$. We will write, for
          $t>0$ (resp.\ $t<0$), $x_{j/i}(t)=x_{j/i}^{+}(t)=x_{j/i}(0^{+})+x_{j/i}'(0^{+})t+\cdots $ and
          $y_{j/i}(t)=y_{j/i}^{+}(t)=y_{j/i}(0^{+})+y_{j/i}'(0^{+})t+\cdots $
          (resp.\ $x_{j/i}(t)=x_{j/i}^{-}(t)=x_{j/i}(0^{-})+x_{j/i}'(0^{-})t+\cdots $ and
          $y_{j/i}(t)=y_{j/i}^{-}(t)=y_{j/i}(0^{-})+y_{j/i}'(0^{-})t+\cdots $) to emphasize
          the fact that the equalities are true only on the right of $0$ 
          (resp.\ on the left of $0$). Of course, we find the coefficients of these expansions
          $x_{j/i}^{(k)}(0^{\pm})$ and $y_{j/i}^{(k)}(0^{\pm})$ by inverting the
          relationships~(\ref{def_chi_chi_tilde}). For instance,
               \begin{equation}\label{diff_y_j/i_0}
                    x_{j/i}'\left(0^{\pm}\right)=\frac{\pm 1+\frac{j}{i}
                    \sqrt{\frac{p_{10}}{p_{01}}}\imath}{1+\left(\frac{j}{i}
                    \right)^{2}\frac{p_{10}}{p_{01}}},\hspace{5mm}
                    y_{j/i}'\left(0^{\pm}\right)=\frac{\pm \frac{j}{i}+
                    \sqrt{\frac{p_{01}}{p_{10}}}\imath}{\left(\frac{j}{i}
                    \right)^{2}+\frac{p_{10}}{p_{01}}}.
               \end{equation}
          \textit{A priori},~(\ref{def_chi_chi_tilde}) is true on
          $]-\delta_{j/i},\delta_{j/i}[$, where $\delta_{j/i}>0$
          depends on $j/i$. An important fact is that we can chose
          $\delta$ sufficiently small such that~(\ref{def_chi_chi_tilde})
          is true for all $t\in ]-\delta,\delta[$ and for all
          $j/i\in [0,M]$.
          This uniformity property is quite important since it allows
          to make $j/i\to \tan(\gamma)$ without requiring $j/i=\tan(\gamma)$, what 
          would be too much restrictive~; in particular this is partly that property
          that allows the calculation of the asymptotic of the Green functions
          in case $j/i\to 0$.
          To prove it, we remark that it suffices to find a lower bound,
          positive and independent of $j/i\in [0,M]$,
          of the radius of convergence of functions $x_{j/i}^{\pm}$ and $y_{j/i}^{\pm}$~;
          this can be done by finding an upper bound independent of $j/i\in [0,M]$
          of the coefficients of the Taylor series at $0$ of $x_{j/i}^{\pm}$ and $y_{j/i}^{\pm}$.
          To find this upper bound,
          we can use the so called Bürman-Lagrange formula
          (see e.g.\ \cite{Chat}), which allows to write the coefficients of the Taylor
          series of an inverse function as integrals in terms of the direct function.
          Then a tedious calculation allows to find $\delta>0$ such that
          $\sup_{k\in \mathbb{N}}\sup_{j/i\in[0,M]}|x_{j/i}^{(k)}(0^{\pm})|\delta^{k}<\infty $ and
          $\sup_{k\in \mathbb{N}}\sup_{j/i\in[0,M]}|y_{j/i}^{(k)}(0^{\pm})|\delta^{k}<\infty $,
          from which the uniformity property comes.

          In addition, $x_{j/i}$ and $y_{j/i}$ are strongly
          connected since
          these two parameterizations are tied together by
          $Y_{1}(x_{j/i}(t),1)=y_{j/i}(-t)$ and $X_{1}(y_{j/i}(t),1)=x_{j/i}(-t)$.
          To prove this fact, note that a consequence of the definition~(\ref{def_chi_chi_tilde})
          and also of the relationships concerning
          the composed functions $X_{k}\circ Y_{l}$, where $k,l\in \{0,1\}$, given in
          Lemma~\ref{properties_X_Y}
          is that $Y_{1}(x_{j/i}(t),1)\in\{y_{j/i}(-t),y_{j/i}(t)\}$. Then, it suffices
          to calculate the sign of imaginary part of the previous quantities to identify
          which of $y_{j/i}(-t),y_{j/i}(t)$ is equal to $Y_{1}(x_{j/i}(t),1)$~;
          that is, in this case, $y_{j/i}(-t)$.

          To sum up, we have built, for all $j/i\in [0,M]$, 
          two contours that coincide in the neighborhood of $1$ with
          the steepest descent paths $x_{j/i}(]-\delta,\delta[)$ and
          $y_{j/i}(]-\delta,\delta[)$, that elsewhere remain outside some
          suitable level lines, that are symmetrical w.r.t.\ the $x$-axis and that are smooths, except at $1$.
          As an example we have represented in Figure~\ref{Four_contours} four of these contours.
          The first two correspond to $j/i\to 0$, the last two are associated to 
          some $j/i$ fixed in $]0,\infty[$. Note that to draw the first two contours, we have used the
          equalities (which are consequences of~(\ref{def_chi_chi_tilde}))~:
          \begin{equation*}
               x_{0}(]-\delta,\delta[) =\left[\overleftarrow{
               \underrightarrow{1,\exp(\delta)}}\right[,
               \hspace{5mm}y_{0 }(]-\delta,\delta[)=
               Y_{1}\left(\left[\overrightarrow{
               \underleftarrow{1,\exp(\delta)}}\right[,1\right),
          \end{equation*}
          $y_{0}(]-\delta,\delta[)$ is thus an arc of
          circle (recall the definition~(\ref{def_curves_L_M}) of the curve
          $\mathcal{L}_{z}$ and the fact that it is just a circle).

\begin{figure}[!h]
\begin{center} 
\begin{picture}(420.00,80.00)
\includegraphics{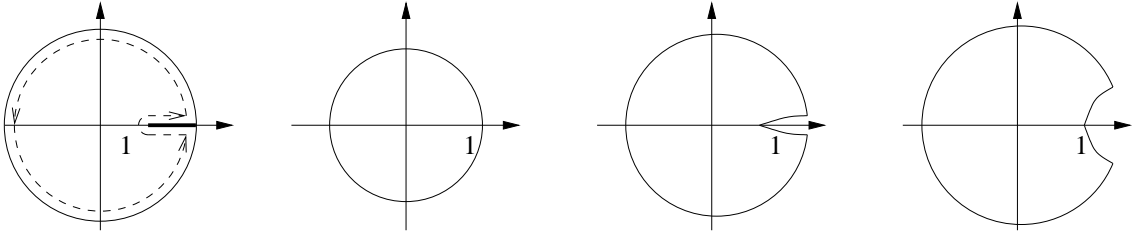}
\end{picture}
\end{center}
\caption{From left to right, the contours 
associated to $\chi_{0}$, $\tilde{\chi}_{0}$, $\chi_{j/i}$ and 
$\tilde{\chi}_{j/i}$ for some $j/i$ in $]0,+\infty [$.}
\label{Four_contours}
\end{figure}

\paragraph{Study of $G_{i,j,1}(\epsilon)$.}
The definition of $G_{i,j,1}(\epsilon)$ immediately leads to~:
     \begin{eqnarray*}
          G_{i,j,1}\left(\epsilon \right)&=&\frac{1}{\left(2\pi \imath \right)^{2}}
          \int_{\left|x\right|=1-\epsilon}
          \frac{h\left(x,1\right)-h\left(1,1\right)}
          {x^{i}}\left(\int_{\left|y\right|=1-\epsilon}
          \frac{\text{d}y}{Q\left(x,y,1\right)y^{j}}\right)\text{d}x\\
          &-&\frac{1}{\left(2\pi \imath \right)^{2}}
          \int_{\left|y\right|=1-\epsilon}
          \frac{h\left(X_{1}\left(y,1\right),1\right)-h\left(1,1\right)}
          {y^{j}}\left(\int_{\left|x\right|=1-\epsilon}
          \frac{\text{d}x}{Q\left(x,y,1\right)x^{i}}\right)\text{d}y.
     \end{eqnarray*}
Then, we apply the residue theorem at infinity to the two integrals
$\int \text{d}x/(Q(x,y,1)x^{i})$ and $\int \text{d}y/(Q(x,y,1)y^{j})$
on the contours
$|x|=1-\epsilon$  and $|y|=1-\epsilon$.
Since $Q(x,y,1)=\tilde{a}(y,1)(x-X_{0}(y,1))(x-X_{1}(y,1))
=a(x,1)(y-Y_{0}(x,1))(y-Y_{1}(x,1))$ we have to know the
position of $X_{i}(y,1)$ and $Y_{i}(x,1)$ w.r.t.\ the circle
$\mathcal{C}(0,1-\epsilon )$, but this is exactly the object
of the remark~{\rm (i)}, so we can write~:
     \begin{eqnarray*}
          G_{i,j,1}\left(\epsilon \right)=&-&
          \frac{1}{2\pi \imath}\int_{\substack{ \left|\theta\right|<
          \widetilde{\theta}_{0}\left(\epsilon\right) \\ x=\left(1-\epsilon\right)
          \exp\left(\imath\theta\right)}}
          \frac{h\left(x,1\right)-h\left(1,1\right)}
          {a\left(x,1\right)\left(Y_{0}\left(x,1\right)-Y_{1}\left(x,1\right)
          \right)x^{i}Y_{0}\left(x,1\right)^{j}}\text{d}x\\
          &+&\frac{1}{2\pi \imath}\int_{\substack{\left|\theta\right|<{\theta}_{0}
          \left(\epsilon\right)\\y=\left(1-\epsilon\right)
          \exp\left(\imath\theta\right)}}
          \frac{h\left(X_{1}\left(y,1\right),1\right)-h\left(1,1\right)}
          {\widetilde{a}\left(y,1\right)\left(X_{0}\left(y,1\right)-
          X_{1}\left(y,1\right)\right)X_{0}\left(y,1\right)^{i}y^{j}}\text{d}y\\
          &-&\frac{1}{2\pi \imath}\int_{\left|x\right|=1-\epsilon}
          \frac{h\left(x,1\right)-h\left(1,1\right)}
          {a\left(x,1\right)\left(Y_{0}\left(x,1\right)-Y_{1}\left(x,1\right)
          \right)x^{i}Y_{1}\left(x,1\right)^{j}}\text{d}x\\
          &+&\frac{1}{2\pi \imath}\int_{\left|y\right|=1-\epsilon}
          \frac{h\left(X_{1}\left(y,1\right),1\right)-h\left(1,1\right)}
          {\widetilde{a}\left(y,1\right)\left(X_{0}\left(y,1\right)-
          X_{1}\left(y,1\right)\right)X_{1}\left(y,1\right)^{i}y^{j}}\text{d}y.
     \end{eqnarray*}
Due to the occurrence of $h(1,1)$,
all the integrands above are integrable
in the neighborhood of $1$. Moreover,
$\theta_{0}(0)=\tilde{\theta}_{0}(0)=0$
so that the first two integrals
have a limit equal to zero as $\epsilon$ goes to zero. Thus,
after $\epsilon$ has gone to zero, we obtain that 
$G_{i,j,1}=\lim_{\epsilon \to 0}G_{i,j,1}(\epsilon)$ is equal to~:
     \begin{equation}\label{def_G_i_j_1}
          \frac{1}{2\pi \imath }\int_{\left|y\right|=1}
          \frac{\left(h\left(X_{1}\left(y,1\right),1\right)-h\left(1,1\right)\right)\text{d}y}
          {\widetilde{a}\left(y,1\right)\left(X_{0}\left(y,1\right)-
          X_{1}\left(y,1\right)\right)X_{1}\left(y,1\right)^{i}y^{j}}-
          \frac{1}{2\pi \imath }\int_{\left|x\right|=1}
          \frac{\left(h\left(x,1\right)-h\left(1,1\right)\right)\text{d}x}
          {a\left(x,1\right)\left(Y_{0}\left(x,1\right)-Y_{1}\left(x,1\right)
          \right)x^{i}Y_{1}\left(x,1\right)^{j}}.
     \end{equation}
We will now use the Cauchy theorem
to move the contours
$|x|=1$ and $|y|=1$. Since $x\mapsto Y_{i}(x,1)$ and $y\mapsto X_{i}(y,1)$, 
$i=0,1$, are holomorphic on $\mathbb{C}\setminus [x_{1}(1),x_{4}(1)]$
and $\mathbb{C}\setminus [y_{1}(1),y_{4}(1)]$
and $x\mapsto h(x,1)$ on $\mathbb{C}\setminus [1,x_{4}(1)]$,
we can move the contours $\mathcal{C}(0,1)$ into new ones,
that coincide in the neighborhood of $1$
with the steepest descent paths constructed previously,
and whose remaining part
lies in $|\chi_{j/i}(x)|>\rho>0$.
In other words, these new contours
--that we call $\mathcal{C}_{j/i,x}$ and $\mathcal{C}_{j/i,y}$-- verify~:
     \begin{itemize}
          \item $\mathcal{C}_{j/i,x}$ (resp.\ $\mathcal{C}_{j/i,y}$)
                contains $x_{j/i}(]-\delta,\delta [)$ (resp.\ $y_{j/i}(]-\delta,\delta [)$)
                where $x_{j/i}$, $y_{j/i}$ and $\delta $ are defined in {\rm (ii)},
          \item for all $x\in \mathcal{C}_{j/i,x}\setminus x_{j/i}(]-\delta,\delta [)$
                (resp.\ $y\in \mathcal{C}_{j/i,y}\setminus y_{j/i}(]-\delta,\delta [)$),
                $|\chi_{j/i}(x)|>\rho(\delta)$
                (resp.\ $|\tilde{\chi}_{j/i}(x)|>\tilde{\rho}(\delta)$)
                where $\rho$ and $\tilde{\rho}$ are two positive functions.
     \end{itemize}
The fact that $\rho$ and $\tilde{\rho}$
can be chosen independently of
$j/i\in [0,M]$ follows from the continuity
of the modulus of $\chi_{j/i}$ and $\tilde{\chi}_{j/i}$
relatively to $j/i \in [0,M]$.

So, by a standard argument,
the integrals on $\mathcal{C}_{j/i,x}\setminus x_{j/i}(]-\delta,\delta[)$ and
$\mathcal{C}_{j/i,y}\setminus y_{j/i}(]-\delta,\delta[)$ are
exponentially negligible
. Indeed, setting $E_{1}=\{x\in \mathbb{C} : |x-x_{1}(1)|>\alpha , |x-x_{4}(1)|>\alpha , |x|<1/\alpha \}$ and
$E_{2}=\{y\in \mathbb{C} : |y-y_{1}(1)|>\alpha , |y-y_{4}(1)|>\alpha , |y|<1/\alpha \}$,
a consequence of expressions of $h$ and $\tilde{h}$
obtained in Proposition~\ref{explicit_h(x,z)_third} and
of the expressions of $Y_{i}$ and $X_{i}$ (see Subsection~\ref{The_algebraic_curve_Q}) 
is that
\begin{equation*}
     K=\sup_{x\in E_{1}} \left|\frac{h\left(x,1\right)-h\left(1,1\right)}
     {a\left(x,1\right)\left(Y_{0}\left(x,1\right)-Y_{1}\left(x,1\right)
     \right)}\right|+\sup_{y\in E_{2}} \left|\frac{h\left(X_{1}\left(y,1\right),1\right)-
     h\left(1,1\right)}{\widetilde{a}\left(y,1\right)\left(X_{0}\left(y,1\right)-
     X_{1}\left(y,1\right)\right)}\right|<\infty .
\end{equation*}
In addition, the contours $\mathcal{C}_{j/i,x}$ and
$\mathcal{C}_{j/i,y}$ can clearly be chosen such that
$\cup_{j/i\in [0,M]} \mathcal{C}_{j/i,x}\subset E_{1}$
and $\cup_{j/i\in [0,M]} \mathcal{C}_{j/i,y}\subset E_{2}$
and also such that
$l=\sup_{j/i\in [0,M]} 
(\text{length}(\mathcal{C}_{j/i,x})+\text{length}(\mathcal{C}_{j/i,y}))<\infty $. So,
\begin{eqnarray*}
          &&\left|\frac{1}{2\pi \imath }\int_{\mathcal{C}_{j/i,y}\setminus y_{j/i}(]-\delta,\delta[)}
          \frac{h\left(X_{1}\left(y,1\right),1\right)-h\left(1,1\right)}
          {\widetilde{a}\left(y,1\right)\left(X_{0}\left(y,1\right)-
          X_{1}\left(y,1\right)\right)X_{1}\left(y,1\right)^{i}y^{j}}\text{d}y\right.\\
          &&-\left.\frac{1}{2\pi \imath }\int_{\mathcal{C}_{j/i,x}\setminus x_{j/i}(]-\delta,\delta[)}
          \frac{h\left(x,1\right)-h\left(1,1\right)}
          {a\left(x,1\right)\left(Y_{0}\left(x,1\right)-Y_{1}\left(x,1\right)
          \right)x^{i}Y_{1}\left(x,1\right)^{j}}\text{d}x\right|\leq \frac{K l}{2\pi} 
          e^{-i \min \left(\rho\left(\delta\right),\widetilde{\rho}\left(\delta\right) \right)}.
     \end{eqnarray*}
It remains to consider~:
     \begin{eqnarray*}
          & &\frac{1}{2\pi \imath }\int_{-\delta}^{\delta}
          \frac{h\left(X_{1}\left(y_{j/i}\left(t\right),1\right),1\right)
          -h\left(1,1\right)}
          {\left.\left\{\widetilde{a}\left(y,1\right)\left(X_{0}\left(y,1\right)-
          X_{1}\left(y,1\right)\right)\right\}\right|_{y=y_{j/i}\left(t\right)}}
          \exp\left(-i\left|t\right|\right)y_{j/i}'\left(t\right)\text{d}t\\
          &-&\frac{1}{2\pi \imath }\int_{-\delta}^{\delta}
          \frac{h\left(x_{j/i}\left(t\right),1\right)-h\left(1,1\right)}
          {\left.\left\{a\left(x,1\right)\left(Y_{0}\left(x,1\right)-
          Y_{1}\left(x,1\right)\right)\right\}\right|_{x=x_{j/i}\left(t\right)}}
          \exp\left(-i\left|t\right|\right)x_{j/i}'\left(t\right)\text{d}t.
     \end{eqnarray*}
Actually, this quantity is equal to zero. To prove that,
we proceed to the three following manipulations
in the first integral above~: (1)~we do the change of variable
$t\mapsto -t$ (2)~we use the fact (mentioned in~{\rm (ii)})
that $X_{1}(y_{j/i}(-t),1)=x_{j/i}(t)$ (3)~we use the
equality $\{\tilde{a}(y,1)(X_{0}(y,1)-X_{1}(y,1))\}|_{y=y_{j/i}t)}
=-Y_{1}'(x_{j/i}(t),1)\{a(x,1)(Y_{0}(x,1)-Y_{1}(x,1))\}|_{x=x_{j/i}(t)}$.
Then, we immediately obtain that the difference
of the two integrals is equal to zero.
It remains to prove~(3)~: if we differentiate the
equality $Q(x,Y_{1}(x),1)=0$, we obtain~:
     \begin{equation*}
          \left.\left(2\widetilde{a}\left(y,1\right)x+\widetilde{b}\left(y,1\right)
          \right)\right|_{y=Y_{1}\left(x,1\right)}+Y_{1}'\left(x,1\right)
          \left(2a\left(x,1\right)Y_{1}\left(x,1\right)+b\left(x,1\right)\right)=0.
     \end{equation*}
Then, taking $x=x_{j/i}(t)$ and using that $X_{1}(y_{j/i}(-t),1)=x_{j/i}(t)$ and
$Y_{1}(x_{j/i}(-t),1)=y_{j/i}(t)$, we obtain~(3).

\paragraph{Study of $G_{i,j,2}(\epsilon)$.}
As for $G_{i,j,1}(\epsilon)$ we start by splitting $G_{i,j,2}(\epsilon)$ in two terms~:
     \begin{eqnarray*}
          G_{i,j,2}\left(\epsilon\right)&=&\frac{1}{\left(2\pi \imath \right)^{2}}
          \int_{\left|y\right|=1-\epsilon}
          \frac{X_{1}\left(y,1\right)^{n_{0}}}
          {y^{j-m_{0}}}\left(\int_{\left|x\right|=1-\epsilon}
          \frac{\text{d}x}{Q\left(x,y,1\right)x^{i\phantom{-n_{0}}}}\right)\text{d}y\\
          &-&\frac{1}{\left(2\pi \imath \right)^{2}}
          \int_{\left|y\right|=1-\epsilon}
          \frac{\phantom{X_{1}}1\phantom{\left(y,1\right)^{n}}}
          {y^{j-m_{0}}}\left(\int_{\left|x\right|=1-\epsilon}
          \frac{\text{d}x}{Q\left(x,y,1\right)x^{i-n_{0}}}\right)\text{d}y.
     \end{eqnarray*}
Then, we use that
$Q(x,y,1)=\tilde{a}(y,1)(x-X_{0}(y,1))(x-X_{1}(y,1))$ and
we apply the residue theorem at infinity to the integrals
$\int \text{d}x/(Q(x,y,1)x^{i})$ and $\int \text{d}x/(Q(x,y,1)x^{i-n_{0}})$.
Using the properties of the modulus of $X_{0}$ and $X_{1}$ described
in~{\rm (i)}, we find~:
     \begin{equation*}
          G_{i,j,2}\left(\epsilon\right)=\frac{1}{2\pi \imath }
          \int_{\substack{\left|\theta\right|<{\theta}_{0}
          \left(\epsilon\right)\\y=\left(1-\epsilon\right)
          \exp\left(\imath \theta\right)}}
          \frac{1}{\widetilde{a}\left(y,1\right)X_{0}\left(y,1\right)^{i}{y}^{j-m_{0}}}
          \frac{X_{1}\left(y,1\right)^{n_{0}}-X_{0}\left(y,1\right)^{n_{0}}}
          {X_{1}\left(y,1\right)-X_{0}\left(y,1\right)}\text{d}y.
     \end{equation*}
First, the integrand is integrable on the contour considered, and secondly,
once again thanks~{\rm (i)}, $\theta_{0}(\epsilon)\to 0$
when $\epsilon \to 0$~; so 
$G_{i,j,2}(\epsilon)\to 0$ as $\epsilon \to 0$.

\paragraph{Study of $G_{i,j,3}(\epsilon)$.}

As for $G_{i,j,1}(\epsilon)$ and $G_{i,j,2}(\epsilon)$ we write
$Q(x,y,1)=\tilde{a}(y,1)(x-X_{0}(y,1))(x-X_{1}(y,1))$ and
we apply the residue theorem at infinity. Next, we let $\epsilon \to 0$~;
we obtain~:
     \begin{equation}\label{def_G_i_j_3}
          G_{i,j,3}=\lim_{\epsilon \to 0}G_{i,j,3}(\epsilon)=\frac{1}{2\pi \imath }
          \int_{\left|y\right|=1}
          \frac{h\left(X_{1}\left(y,1\right),1\right)+
          \widetilde{h}\left(y,1\right)-X_{1}\left(y,1\right)^{n_{0}}
          y^{m_{0}}}{\widetilde{a}\left(y,1\right)
          \left(X_{0}\left(y,1\right)-X_{1}\left(y,1\right)\right)
          X_{1}\left(y,1\right)^{i}y^{j}}\text{d}y.
     \end{equation}
We now move the contour from the unit circle
into $\mathcal{C}_{j/i,y}$, using
the same arguments as in the paragraph concerning
the study of $G_{i,j,1}$. For the same
reasons as before, the integral
on $\mathcal{C}_{j/i,y}\setminus y_{j/i}(]-\delta,\delta[)$
is exponentially negligible,
so that from now we consider that the integral defining
$G_{i,j,3}$ is on the contour $y_{j/i}(]-\delta,\delta[)$.

Now, we notice that from remark~{\rm (ii)} we can deduce that
$y_{j/i}(]-\delta,\delta[)\subset \{s\in \mathbb{C} : |s|>1\}$,
so that in accordance with~(\ref{link_explicit_constant_applied}) we can write,
for all $y\in y_{j/i}(]-\delta,\delta[)$~:
     \begin{equation*}
          h\left(X_{1}\left(y,1\right),1\right)+
          \widetilde{h}\left(y,1\right)-X_{1}\left(y,1\right)^{n_{0}}
          y^{m_{0}}=-\left(y^{m_{0}}-y^{-m_{0}}\right)
          \left(X_{1}\left(y,1\right)^{n_{0}}-X_{1}\left(y,1\right)^{-n_{0}}\right).
     \end{equation*}
For this reason,
     \begin{equation*}
          \frac{h\left(X_{1}\left(y,1\right),1\right)+
          \widetilde{h}\left(y,1\right)-X_{1}\left(y,1\right)^{n_{0}}
          y^{m_{0}}}{\widetilde{a}\left(y,1\right)\left(X_{0}\left(y,1\right)
          -X_{1}\left(y,1\right)\right)}=\frac{\left(y-1\right)
          \left(X_{1}\left(y,1\right)-1\right)}{\sqrt{\widetilde{d}\left(y,1\right)}
          X_{1}\left(y,1\right)^{n_{0}}y^{m_{0}}}\sum_{i=0}^{2m_{0}-1}y^{i}
          \sum_{i=0}^{2n_{0}-1}X_{1}\left(y,1\right)^{i}.
     \end{equation*}
Since $X_{1}(y,1)-1=\tilde{d}(y,1)^{1/2}/(2\tilde{a}(y,1))-p_{01}(y-1)^{2}/(2\tilde{a}(y,1))$,
we have the following expansion~:
     \begin{equation*}
          \frac{\left(y-1\right)
          \left(X_{1}\left(y,1\right)-1\right)}{\sqrt{\widetilde{d}\left(y,1\right)}
          X_{1}\left(y,1\right)^{n_{0}}y^{m_{0}}}\sum_{i=0}^{2m_{0}-1}y^{i}
          \sum_{i=0}^{2m_{0}-1}X_{1}\left(y,1\right)^{i}=\frac{2n_{0}m_{0}}{p_{10}}
          \left(y-1\right)+c_{2}\left(y-1\right)^{2}+c_{3}\left(y-1\right)^{3}+\cdots ,
     \end{equation*}
where the coefficients $c_{2},c_{3},\ldots $ depend on the half
plane (upper or lower) where the 
previous expansion is written.
We could of course male explicit these coefficients $c_{2},c_{3},\ldots $ but we don't need it.
Indeed, using the expansion of $y_{j/i}$ at the left and the right of $0$,
and Laplace's method, we see that the integrals
     \begin{equation*}
          \int_{-\delta}^{\delta }\left(y_{j/i}\left(t\right)-1\right)^{k}
          \exp\left(-i \left|t\right|\right)y_{j/i}'\left(t\right)\text{d}t
     \end{equation*}
are for $k\geq 2$  polynomially negligible with respect to the same
integral where $k=1$. Therefore, to find the asymptotic, we have
only to consider the above integral for $k=1$. We find that
$\int_{-\delta}^{\delta }(y_{j/i}(t)-1) \exp(-i
|t|)y_{j/i}'(t)\text{d}t\sim (y_{j/i}'(0+)^{2}-y_{j/i}'(0-)^{2})
\int_{0}^{\delta}t\exp(-i t)\text{d}t$ as $i$ goes to infinity~; then
with~(\ref{diff_y_j/i_0}) we get~:
     \begin{equation*}
          y_{j/i}'\left(0+\right)^{2}-y_{j/i}'\left(0-\right)^{2}=
          4\imath p_{01}^{1/2}p_{10}^{3/2}\frac{i j}
          {\left(p_{10}j^{2}+p_{01}i^{2}\right)^{2}},
     \end{equation*}
from which Proposition~\ref{Martin_boundary_drift_zero}
follows immediately.
\end{proof}

Let us now turn to the case $M_{x}>0$ and $M_{y}>0$.
In~\cite{KV}, the authors obtain the Green
functions' asymptotic
for the random walks in $(\mathbb{Z}_{+})^2$
with the same jump probabilities in the interior of the quadrant
as ours but with non zero jumps from the axes to the interior.
The analysis of this problem in our case can be carried out
by the same methods. Therefore we just claim

\begin{prop}\label{Martin_boundary_drift_non_zero}
Denote by $\left(s_{3}\left(\gamma\right),
t_{3}\left(\gamma\right)\right)$
the unique critical point of $\left(x,y\right)\mapsto x
y^{\tan\left(\gamma\right)}$ on
$\{\left(x,y\right)\in \mathbb{C}^{2},Q\left(x,y,1\right)=0,x
\in \mathbb{R}_{+}, y\in \mathbb{R}_{+}\}$.
The Green functions admit the following asymptotic when 
$i,j\to \infty $ and $j/i\to
\tan\left(\gamma\right)\in \left]0,+\infty\right[$~:
     \begin{equation}\label{Martin_boundary_drift_non_zero_ij>0}
          G_{i,j}^{n_{0},m_{0}}\sim \frac{\sqrt{s_{3}\left(\gamma\right)^{
          \frac{1}{\tan\left(\gamma\right)}}
          t_{3}\left(\gamma\right)}\left(s_{3}\left(\gamma\right)^{n_{0}}t_{3}\left(\gamma\right)^{m_{0}}
          -h\left(s_{3}\left(\gamma\right),1\right)-\widetilde{h}\left(t_{3}\left(\gamma\right),1\right)\right)}
          {\sqrt{2\pi}\big(2 a\left(s_{3}\left(\gamma\right)\right)
          t_{3}\left(\gamma\right)+b\left(s_{3}\left(\gamma\right)
          \right)\big)\sqrt{j\frac{\textnormal{d}}{\textnormal{d}x^{2}}
          \left(x^{1/\tan\left(\gamma\right)}Y_1\left(x,1\right)\right)
          \Big|_{x=s_{3}\left(\gamma\right)}}s_{3}\left(\gamma\right)^{i}
          t_{3}\left(\gamma\right)^{j}}.
     \end{equation}
In cases $j/i\to 0$ or $+\infty$, the Green functions admit the asymptotic~:
     \begin{equation}\label{Martin_boundary_drift_non_zero_ij=0}
          \begin{array}{ccc}
          G_{i,j}^{n_{0},m_{0}}&\sim &  \displaystyle
          \left(\left(1-2\sqrt{p_{0-1}p_{01}}\right)^{2}-4p_{-10}p_{10}\right)^{1/2}
          \frac{m_{0}\widetilde{r}^{m_{0}-1}\left(x_{3}\left(1\right)^{n_{0}}-
          x_{2}\left(1\right)^{n_{0}}\right)j}{\sqrt{\pi p_{01}
          \sqrt{\widetilde{d}\left(\widetilde{r},1\right)}}
          i^{3/2}x_{3}\left(1\right)^{i}\widetilde{r}^{j}}  , \ j/i\to 0,\\
          G_{i,j}^{n_{0},m_{0}}&\sim &    \displaystyle
          \left(\left(1-2\sqrt{p_{-10}p_{10}}\right)^{2}-4p_{0-1}p_{01}\right)^{1/2}
          \frac{n_{0}r^{n_{0}-1}\left(y_{3}\left(1\right)^{m_{0}}-
          y_{2}\left(1\right)^{m_{0}}\right)i}{\sqrt{\pi p_{10}
          \sqrt{d\left(r,1\right)}}j^{3/2}r^{i}y_{3}\left(1\right)^{j}}, \ j/i\to +\infty .
          \end{array}
     \end{equation}
\end{prop}

Our previous remarks allow us to be more specific about this result
concerning two things. First, as it was already remarked
in~\cite{MA1} and in~\cite{KV}, $t_{3}(\gamma)\in
[(p_{0-1}/p_{01})^{1/2},y_{3}(1)]$
and $s_{3}(\gamma)=X_{1}(t_{3}(\gamma),1)$. So
with~(\ref{link_explicit_constant_applied}) we obtain that the
constant
$s_{3}(\gamma)^{n_{0}}
t_{3}(\gamma)^{m_{0}}-h(s_{3}(\gamma),1)-\tilde{h}(t_{3}(\gamma),1)$ is equal to $(t_{3}(\gamma)^{m_{0}}-
(\tilde{r}^{2}/t_{3}(\gamma))^{m_{0}})
(s_{3}(\gamma)^{n_{0}}-(r^{2}/s_{3}(\gamma))^{n_{0}})$, which is a
simpler expression. In fact, we can simplify longer, and that is the
second thing that we wanted to add, the critical point
$\left(s_{3}\left(\gamma\right),t_{3}\left(\gamma\right)\right)$ has
the following explicit expression~:
     \begin{equation*}
          s_{3}\left(\gamma\right)=-B\left(\gamma\right)/2+
          \sqrt{\left(B\left(\gamma\right)/2\right)^{2}-1},
          \hspace{5mm}t_{3}\left(\gamma\right)=-\widetilde{B}
          \left(\gamma\right)/2+\sqrt{\left(\widetilde{B}
          \left(\gamma\right)/2\right)^{2}-1},
     \end{equation*}
where $B\left(\gamma\right)=(1-(1-(1-\tan(\gamma)^{2})
(1-4p_{0-1}p_{01}+4p_{-10}p_{10}\tan(\gamma)^{2}))^{1/2})/
(\tan(\gamma)^{2}-1)$ and
$\tilde{B}\left(\gamma\right)=(1-(1-(1-\tan(\gamma)^{-2})
(1-4p_{-10}p_{10}+4p_{0-1}p_{01}\tan(\gamma)^{-2}))^{1/2})/
(\tan(\gamma)^{-2}-1)$.
In particular, note that $s_{3}\left(0\right)=x_{3}\left(1\right)$,
$s_{3}\left(\pi/2\right)=r$, $t_{3}\left(0\right)=\tilde{r}$ and
$t_{3}\left(\pi/2\right)=y_{3}\left(1\right)$.
Note that we obtain the explicit expressions of $s_{3}(\gamma)$
and $t_{3}(\gamma)$ by solving
$(xY_{1}(x,1)^{\tan(\gamma)})'=0$ and
$(X_{1}(y,1)y^{\tan(\gamma)})'=0$. We will discuss
this fact again in a next work.
\medskip
\newline Results of Sections~\ref{h_1_x} and~\ref{Martin_boundary}
allow to describe completely the
Martin boundary of the process,
both in case of a positive drift and in case
of a zero drift.
For definitions, details and applications
of Martin boundary theory, we
refer to~\cite{Dynkin} and~\cite{REV}.

\begin{rem}\label{rem_Martin_boundary}
{\rm Let us now fix any reference state in the interior of the quadrant,
e.g.\ $(1,1)$, and consider the Martin kernel
$k_{(n_{0},m_{0})}(i,j)$.
   It equals to $G_{(i,j)}^{(n_{0},m_{0})}/
G_{(i,j)}^{(1,1)}$ and as well to
$\mathbb{P}_{(n_{0},m_{0})}(\textnormal{to hit}
\,(i,j))/\mathbb{P}_{(1,1)}(\textnormal{to hit} \,(i,j))$.
  We will use the first form in the case
  $i,j>0$ and the second one in the case
  $i=0$ or $j=0$.
 At last we set, for $\gamma \in [0,\pi/2]$,
$k_{(n_{0},m_{0})}(\gamma)= \lim_{i,j\to +\infty, j/i\to
\tan(\gamma)}k_{(n_{0},m_{0})}(i,j)$.

It follows from Proposition~\ref{Martin_boundary_drift_zero} that in
the case $M_{x}=M_{y}=0$, for any sequence $(i,j)$ where both
coordinates are positive and $j/i$ converges to
$\tan(\gamma)\in[0,+\infty]$, the Martin kernel
$k_{(n_{0},m_{0})}(i,j)$ converges to $n_{0}m_{0}$. Moreover, we can
deduce from Proposition~\ref{singularity_h_ln_asymptotic} that in
the case $M_{x}=M_{y}=0$, for any sequence of pairs $(i,j)$ where
one of the coordinates $i$ or $j$ goes to infinity, the other being
$0$, the Martin kernel $k_{(n_{0},m_{0})}(i,j)$ converges also to
$n_{0}m_{0}$. So if the two drifts are equal to zero, $n_{0}m_{0}$
is the unique point of the Martin boundary.

Suppose now that the drifts $M_{x}$ and $M_{y}$ are positive and
show that in this case, the Martin boundary is homeomorphic to $[0,
\pi/2]$. It follows from~(\ref{Martin_boundary_drift_non_zero_ij>0})
of Proposition~\ref{Martin_boundary_drift_non_zero} that for any
sequence $(i,j)$ where both coordinates are positive and $j/i$
converges to $\tan(\gamma)\in]0,+\infty[$, the Martin kernel
$k_{(n_{0},m_{0})}(i,j)$ converges to
$k_{(n_{0},m_{0})}(\gamma)=(t_{3}(\gamma)^{m_{0}}-
(\tilde{r}^{2}/t_{3}(\gamma))^{m_{0}})
(s_{3}(\gamma)^{n_{0}}-(r^{2}/s_{3}(\gamma))^{n_{0}})/((t_{3}(\gamma)-
\tilde{r}^{2}/t_{3}(\gamma))
(s_{3}(\gamma)-r^{2}/s_{3}(\gamma)))$. If now $(i,j)$ is a
sequence whose both coordinates are positive and such that $j/i$
goes to $0$ or $\infty$,
Proposition~\ref{Martin_boundary_drift_non_zero} gives that the
Martin kernel $k_{(n_{0},m_{0})}(i,j)$ converges to
$k_{(n_{0},m_{0})}(0)=m_{0}\tilde{r}^{m_{0}/2-1}
(x_{3}(1)^{n_{0}}-x_{2}(1)^{n_{0}})/(x_{3}(1)-x_{2}(1))$ and
$k_{(n_{0},m_{0})}(\pi/2)= n_{0}r^{n_{0}/2-1}
(y_{3}(1)^{m_{0}}-y_{2}(1)^{m_{0}})/(y_{3}(1)-y_{2}(1))$
respectively. A consequence of
Proposition~\ref{singularity_h_square_asymptotic} is that the Martin
kernels $k_{(n_{0},m_{0})}(i,0)$ and $k_{(n_{0},m_{0})}(0,j)$
converge too, respectively to $m_{0}\tilde{r}^{m_{0}/2-1}
(x_{3}(1)^{n_{0}}-x_{2}(1)^{n_{0}})/(x_{3}(1)-x_{2}(1))$ and
$n_{0}r^{n_{0}/2-1}
(y_{3}(1)^{m_{0}}-y_{2}(1)^{m_{0}})/(y_{3}(1)-y_{2}(1))$ as $i$ and
$j$ go to infinity, respectively. In particular, the Martin kernel
is the same depending on whether $j/i\to 0$ with $j>0$ or $j/i\to 0$
with $j=0$~; also on whether $i/j\to 0$ with $i>0$ or $i/j\to 0$
with $i=0$. At last, using the explicit expression of the critical
point $(s_{3}(\gamma),t_{3}(\gamma))$ given just above
Remark~\ref{rem_Martin_boundary}, we notice that the function
$k_{(n_{0},m_{0})}(\gamma)$ is continuous on $[0, \pi/2]$.
Therefore the Martin boundary is homeomorphic to a segment
$[0, \pi/2]$.}
\end{rem}

\section{Extension of the results}\label{Extension_results}

Suppose now that in addition to $p_{10}$, $p_{-10}$, $p_{01}$,
$p_{0-1}$, we permit the transition probabilities $p_{11}$,
$p_{-11}$, $p_{-1-1}$, $p_{1-1}$ to be non zero as in the hypothesis
(H2) of Section~\ref{Intro}, and suppose that the two drifts
$M_{x}=\sum_{i,j}i p_{i j}$ and $M_{y}=\sum_{i,j}j p_{i j}$ are non
negative. We make the following additional hypothesis~: in the list
$p_{-1-1},p_{-10},p_{-11},p_{01},p_{11},p_{10},p_{1-1},p_{0-1},p_{-1-1},p_{-10}$,
there are no three consecutive zeros. This technical assumption
allows to avoid studying degenerated random walks.

We ask us the same questions as before~: can we still find the
asymptotic of $\mathbb{P}_{(n_{0},m_{0})}($to be absorbed at $(i,0))$,
that of $\mathbb{P}_{(n_{0},m_{0})}(\tau =k)$, that of $G_{i,j}^{n_{0},m_{0}}$~?

To answer these questions we take back the analytic aspects
considered at the beginning of this article and we try to
generalize them~: we can  define an analogous of the polynomial $Q$
presented in~(\ref{def_Q})~: $Q(x,y,z)= x y z( \sum_{i,j}p_{i
j}x^{i} y^{j} -z^{-1} )$, and Equation~(\ref{functional_equation})
is still true, if we add to the right member
the function $h_{00}(z)$, equal to the
generating function of the probabilities of being absorbed at
$(0,0)$ at time $n$~: $h_{00}(z)=\sum_{n=0}^{+\infty}
\mathbb{P}_{(n_{0},m_{0})} (\text{to hit}\,(0,0)\,\text{at time}\,
n) z^{n}$. Next, we can also define the functions $X(y,z)$ and
$Y(x,z)$, that verify properties closed to those described in
Lemma~\ref{properties_X_Y}. We can too, as
in~(\ref{def_curves_L_M}), define the curves $\mathcal{L}_{z}$ and
$\mathcal{M}_{z}$, on which the functions $h$ and $\tilde{h}$ verify
again boundary conditions, like~(\ref{SR_problem_h}). There is
however here a striking difference between the walks studied in the
previous sections and the more general walks~: for the first, the
curves $\mathcal{L}_{z}$ and $\mathcal{M}_{z}$ are included in the
closed unit disc (indeed, they are circles centered at zero and with
radius less than one, as we have seen in
Subsection~\ref{Riemann_Carleman_problem}), what is \textit{prima facie} no more
true for the second. In addition, in both cases, the functions $h$
and $\tilde{h}$ are holomorphic in the open unit disc, continuous 
up to its boundary. So, for general walks, 
$h$ and $\tilde{h}$ could be not defined on the curves
on which they satisfy formally the boundary
condition~(\ref{SR_problem_h})~; this is why
for such walks, we have first to
continue $h$ and $\tilde{h}$ into functions holomorphic on sets
whose adherence contains respectively $\mathcal{L}_{z}$ and
$\mathcal{M}_{z}$. We have already mentioned in
Subsection~\ref{Probability_being_absorbed} 
that there exists in~\cite{FIM} a nice
method to construct this continuation elaborated by using
Galois automorphisms, notion that we will explain briefly in
the proof of Proposition~\ref{prop_continuation_Delta_zero}.

But suppose now that we did this continuation and also that somehow
or other we have found the CGF $w_{z}$~; then it is possible,
following the method developed in
Subsection~\ref{Integral_representation}, to find an integral
representation of $h$~:
     \begin{equation*}
          h\left(x,z\right)=x^{n_{0}}Y_{0}\left(x,z\right)^{m_{0}}
          +\frac{1}{\pi}\int_{x_{1}\left(z\right)}^{x_{2}\left(z\right)}
          t^{n_{0}}\mu_{m_{0}}\left(t,z\right)\phi\left(t,x,z\right)
          \sqrt{-d\left(t,z\right)}\text{d}t ,
     \end{equation*}
where $\phi(t,x,z)=w_{z}'(t)/(w_{z}(t)-w_{z}(x))-w_{z}'(t)/(w_{z}(t)-w_{z}(0))$.
Therefore, it is definitely the search and the study of the CGF
that constitute the key points of the generalization,
and that is here that there is an other
fundamental difference between the walks
verifying $p_{10}+p_{-10}+p_{01}+p_{0-1}=1$ and the more general
walks~: as we will see in a
next work, it is still possible, adapting an idea present
in~\cite{FIM}, to prove in the general case 
the existence of the CGF, we will even obtain the explicit
expression of the CGF for the curves $\mathcal{L}_{z}$ and
$\mathcal{M}_{z}$~; but these explicit expressions are strongly
complicated and hardly usable. For all that, in this next work, we will be able to
find the asymptotic of $\mathbb{P}_{(n_{0},m_{0})}(\text{to be absorbed at} \,
(i,0))$ for any walk, as for the one 
$\mathbb{P}_{(n_{0},m_{0})}(\tau = k)$,
we will have to concentrate us on a few particular cases.

But go back to our study and search how to generalize our results
relatively easily. For the walks such that
$p_{10}+p_{-10}+p_{01}+p_{0-1}=1$,
Proposition~\ref{explicit_form_CGF} was quite advantageous, since
the expression of the CGF could not really be more simple~; one can
think that for the other walks for which the curves
$\mathcal{L}_{z}$ and $\mathcal{M}_{z}$ are equal to circles, we
will be able to generalize our main results without real
difficulties. So we ask us which are the walks such that
$\mathcal{L}_{z}$ and $\mathcal{M}_{z}$ are certainly circles. To answer
this question, we have to introduce the quantity $\Delta(z)$,
equal to~:
     \begin{equation*}
          \Delta\left(z\right)=
          \left|\begin{array}{ccc}
          p_{11}&p_{10}&p_{1-1}\\
          p_{01}&-1/z&p_{0-1}\\
          p_{-11}&p_{-10}&p_{-1-1}
          \end{array}\right|.
     \end{equation*}
The answer is then given by the following result, whose proof is
postponed to Subsection~\ref{Proof_lemma}.

\begin{lem}\label{Delta_zero_iff_curves_circles}
Define $z_{1}=\inf\{z\geq 0 : x_{2}(z)=x_{3}(z)\}$ and
let $z$ be in $]0,z_{1}]$. Suppose that $\Delta(z)=0$~; then the curves
$\mathcal{L}_{z}$ and $\mathcal{M}_{z}$, defined
in~(\ref{def_curves_L_M}), are circles, eventually degenerated in straight lines.
\end{lem}

This lemma will allow us to make the suitable hypothesis in
Subsections~\ref{generalization_z=1} and~\ref{generalization_x=1}.

We will first of all, in Part~\ref{generalization_z=1}, be
interested in the asymptotic of $\mathbb{P}_{(n_{0},m_{0})}($to
be absorbed at $\,(i,0))$~; to begin with a drift zero, then with a positive
drift. The study of these quantities is based on the analysis of the
singularities of $h(x,1)$ and $\tilde h(y,1)$,
as in Subsection~\ref{Explicit_form_and_asymptotic}. It makes not
appear the time $z$, which is in fact equal there to $1$, so that we
will suppose in all Part~\ref{generalization_z=1} that
$\Delta(1)=0$. In concrete terms,
this hypothesis means that the three
polynomials $a(x,1)=p_{11}x^{2}+p_{01}x+p_{-11}$,
$b(x,1)=p_{10}x^{2}-x+p_{-10}$ and
$c(x,1)=p_{1-1}x^{2}+p_{0-1}x+p_{-1-1}$ are linearly dependent.

Then, in Subsection~\ref{generalization_x=1}, we will be interested in the
asymptotic of
$\mathbb{P}_{(n_{0},m_{0})}(S=k)$ and
$\mathbb{P}_{(n_{0},m_{0})}(T=k)$ in case of a zero drift,
$S$ and $T$ being the hitting times~(\ref{def_hitting_times})~;
that will be derived from
the study of the functions $h(1,z)$ and $\tilde h(1,z)$
of the variable $z$.
Therefore, we will there suppose that for all $z\in ]0,1]$
--if the drift is zero then $z_{1}=1$--,
$\Delta(z)=0$, or equivalently
that $\Delta(1)=\Delta'(1)=0$.
Here is an interpretation of this hypothesis~:

\begin{lem}\label{Delta_zero_drift_zero_a=c}
Suppose that $M_{x}=M_{y}=0$. Then $\Delta(1)=0$ is
equivalent to $p_{11}+p_{-1-1}=p_{1-1}+p_{-11}$,
which means that
the process has a covariance equal to zero.
Suppose still $M_{x}=M_{y}=0$ and that
$\Delta(1)=0$ and make
the additional assumption
$\Delta'(1)=0$. Then
$a=c$ or $\tilde{a}=\tilde{c}$, which
means that either $p_{i j}=p_{i -j}$ for
all $i,j$ or $p_{i j}=p_{-i j}$ for
all $i,j$.
\end{lem}
This lemma and the forthcoming
Proposition~\ref{prop_continuation_Delta_zero} will be proved in
Subsection~\ref{Proof_lemma}.

We close this introductory part by stating
a result generalizing Proposition~\ref{proposition_continuation}~:
suppose that the drifts $M_{x}$ and $M_{y}$
are non negative and that for some $z\in ]0,z_{1}]$,
$\Delta(z)=0$. Then
it is still possible
to continue $x\mapsto h(x,z)$
and $y\mapsto \tilde{h}(y,z)$
on $\mathbb{C}\setminus
[x_{3}(z),x_{4}(z)]$ and $\mathbb{C}\setminus [y_{3}(z),y_{4}(z)]$.

\begin{prop}\label{prop_continuation_Delta_zero}
Suppose that for some $z\in ]0,z_{1}]$, $\Delta(z)=0$. Then the
functions $x\mapsto h(x,z)$ and $y\mapsto \tilde{h}(y,z)$ are
continuable into functions holomorphic on $\mathbb{C}\setminus
[x_{3}(z),x_{4}(z)]$ and $\mathbb{C}\setminus [y_{3}(z),y_{4}(z)]$
respectively.
\end{prop}

In fact the hypothesis $\Delta(z)=0$ is not necessary but we do
it for two reasons~: first because all the walks we are studying here
verify this assumption for at least one $z\in ]0,z_{1}]$,
second because the proof (done in Subsection~\ref{Proof_lemma})
is quite simpler in this case.

\subsection{Asymptotic in the case
$\Delta\left(1\right)=0$}\label{generalization_z=1}

We have already defined, in the discussion beginning the
Section~\ref{Extension_results}, the polynomials
$a(x,1)=p_{11}x^{2}+p_{01}x+p_{-11}$, $b(x,1)=p_{10}x^{2}-x+p_{-10}$
and $c(x,1)=p_{1-1}x^{2}+p_{0-1}x+p_{-1-1}$. Likewise, we define
$d(x,1)=b(x,1)^{2}-4a(x,1)c(x,1)$ and $x_{i}(1)$, $i\in \{1,\cdots
,4\}$, to be the (real) roots of $d(x,1)$, say that they are
enumerated such that $|x_{1}(1)|<|x_{2}(1)|\leq 1\leq
|x_{3}(1)|<|x_{4}(1)|$. Note that in~\cite{FIM}
the authors show that if $M_{y}=0$ then $x_{2}(1)=1=x_{3}(1)$
whereas if $M_{y}>0$ then $0<x_{2}(1)<1<x_{3}(1)$. We recall that
$h_{i}=\mathbb{P}_{(n_{0},m_{0})}(\textnormal{to be absorbed at}\,(i,0))$.

\begin{prop}\label{singularity_h_ln_asymptotic_generalization}
We suppose here that $M_{y}=\sum_{i,j}j p_{i j}=0$.
The probability of being absorbed at $(i,0)$
admits the following asymptotic~:
     \begin{equation*}
          h_{i} \sim \frac{2n_{0}m_{0}}{\pi}
          \sqrt{\frac{\left(p_{10}^2-4p_{11}p_{1-1}\right)
          \left(x_{4}\left(1\right)-1\right)
          \left(1-x_{1}\left(1\right)\right)}
          {a\left(1,1\right)c\left(1,1\right)}}
          \frac{1}{i^{3}},\hspace{5mm} i\to \infty .
     \end{equation*}
\end{prop}


\begin{prop}\label{singularity_h_square_asymptotic_generalization}
We suppose here that $M_{y}=\sum_{i,j}i p_{i j}>0$. The probability
of being absorbed at $(i,0)$
admits the following asymptotic~:
     \begin{equation*}
          h_{i}\sim
          \sqrt{\frac{-x_{3}\left(1\right)d'\left(x_{3}\left(1\right),1\right)}
          {a\left(x_{3}\left(1\right),1\right)c\left(x_{3}\left(1\right),1\right)}}
          \frac{\left(x_{3}\left(1\right)^{n_{0}}-
          x_{2}\left(1\right)^{n_{0}}\right)}{4\sqrt{\pi }}
          m_{0}\left(\frac{c\left(x_{3}\left(1\right),1\right)}
          {a\left(x_{3}\left(1\right),1\right)}\right)^{\frac{m_{0}}{2}}
          \frac{1}{x_{3}\left(1\right)^{i}}\frac{1}{i^{3/2}},\hspace{5mm} i\to \infty .
     \end{equation*}
\end{prop}

\begin{proof}
Propositions~\ref{singularity_h_ln_asymptotic_generalization}
and~\ref{singularity_h_square_asymptotic_generalization}
are generalizations of
Propositions~\ref{singularity_h_square_asymptotic}
and~\ref{singularity_h_ln_asymptotic},
which are themselves corollaries
of the writing of $h(x,1)$ in the
neighborhood of $x_{3}(1)$,
its first singularity.
This expansion of $h(x,1)$ at $x_{3}(1)$
was the object of
Propositions~\ref{singularity_h_ln}
and~\ref{singularity_h_square}, which
are consequences of the explicit expression of $h$,
written in Proposition~\ref{explicit_h(x,z)_third}.
We could follow the same development here~:
\emph{first} finding a generalization
of the explicit expression of $h(x,1)$,
\emph{then} studying this integral function
and its singularity at $x_{3}(1)$,
\emph{at last} deducing the asymptotic of the coefficients of
the Taylor series at $0$.
The technical details look like
those already outlined, notably
in the proofs of
Propositions~\ref{singularity_h_ln}--\ref{singularity_h_ln_asymptotic},
so we omit them~; except the generalization of
the integral representation of $h$, which is quite interesting.
According to Lemma~\ref{Delta_zero_iff_curves_circles}
and since $\Delta(1)=0$, $\mathcal{M}_{1}$
is a circle --suppose non degenerated--,
of center $\gamma$ and radius $\rho$ say.
With these notations, define $\sigma(t)=\gamma+\rho^{2}/(t-\gamma)$ and
suppose that $x_{4}(1)>0$. Then, the following equality holds~:
     \begin{equation}\label{explicit_h(x,1)_fifth}
          h\left(x,1\right) = \frac{x}{\pi}\int_{x_{3}\left(1\right)}
          ^{x_{4}\left(1\right)}\left(t^{n_{0}}-
          \sigma\left(t\right)^{n_{0}}\right)
          \frac{\mu_{m_{0}}\left(t,1\right)}{t\left(t-x\right)}
          \sqrt{-d\left(t,1\right)}\text{d}t+
          xP_{\infty}\left(x\mapsto
          x^{n_{0}-1}Y_{0}\left(x,z\right)^{m_{0}}\right)\left(x\right).
     \end{equation}
If $x_{4}(1)<0$ or if the circle $\mathcal{M}_{1}$
is degenerated, we could even so find
an explicit formulation like~(\ref{explicit_h(x,1)_fifth}).
In any case, it would be useful next, starting
from~(\ref{explicit_h(x,1)_fifth})
or an equivalent, to study the singularity
of $h(x,1)$ at $x_{3}(1)$~; as already said we
don't write the details and refer to the proofs of
Propositions~\ref{singularity_h_ln}--\ref{singularity_h_ln_asymptotic}.
\end{proof}

\subsection{Asymptotic in the case
$\Delta\left(1\right)=\Delta'\left(1\right)=0$}\label{generalization_x=1}


We have already explained in Lemma~\ref{Delta_zero_drift_zero_a=c}
that the hypothesis $\Delta(1)=\Delta'(1)=0$ for all $z$ in
$]0,z_{1}]$ is equivalent, in
case of two zero drifts, to the fact that $a=c$ or
$\tilde{a}=\tilde{c}$. A particular case of
random walks verifying these assumptions is the set of walks such
that $p_{10}+p_{-10}+p_{01}+p_{0-1}=1$, $p_{-10}=p_{10}$,
$p_{0-1}=p_{01}$, studied in the previous sections. The next
proposition consists in generalizing
Proposition~\ref{main_result_simple_walk_drift_zero} in the case of
all walks with drifts zero and verifying
in addition $\Delta(1)=\Delta'(1)=0$.
We recall from~(\ref{def_hitting_times})
that we denote by $S$ and $T$ the hitting times
of the $x$ and $y$-axis.

\begin{prop}\label{prop_asymptotic_hitting_times_real_axis_drift_zero_Delta_zero}
We suppose here that $M_x=M_y=0$ and that $\Delta(1)=\Delta'(1)=0$.
Then the probability of being absorbed at time $k$ on
the $x$-axis admits the following asymptotic~:
     \begin{equation}\label{asymptotic_hitting_times_real_axis_drift_zero_Delta_zero}
          \mathbb{P}_{\left(n_{0},m_{0}\right)}\left(S = k \right)
          \sim
          \frac{n_{0}m_{0}}
          {2 \pi \left(\left(p_{11}+p_{10}+p_{1-1}\right)
          \left(p_{11}+p_{01}+p_{-11}\right)\right)^{1/2}k^{2}}
     \end{equation}
The same asymptotic holds for $\mathbb{P}_{(n_{0},m_{0})}(T = k )$,
the probability of being absorbed at time $k$ on the $y$-axis.
\end{prop}

\begin{proof}
We will now give a sketch of the proof of
Proposition~\ref{prop_asymptotic_hitting_times_real_axis_drift_zero_Delta_zero},
in the case $\tilde{a}=\tilde{c}$,
the case $a=c$ being of course symmetrical.

First, we are interested in the asymptotic of
$\mathbb{P}_{(n_{0},m_{0})}(S = k )$. Since $\Delta(z)=0$ for all
$z\in ]0,z_{1}]$, Lemma~\ref{Delta_zero_iff_curves_circles} gives that
for all $z\in ]0,z_{1}]$ the curve
$\mathcal{M}_{z}$ is a circle~; in fact simply the unit circle
$\mathcal{C}(0,1)$, because $\tilde{a}=\tilde{c}$. The CGF
associated to this curve is thus equal to $t\mapsto
t+1/t$ and the --fundamental-- Proposition~\ref{explicit_h(x,z)_third},
which gives the explicit expression of $h(x,z)$, is still
valid. Then, we can adapt the change of variable $t=t_{2}(u,z)$,
see~(\ref{tttt}), made in Part~\ref{Chebychev_polynomials}, and, with
some additional technical details, we can follow the
Subsection~\ref{Simple_walk_drift_zero} and finally obtain the
asymptotic~(\ref{asymptotic_hitting_times_real_axis_drift_zero_Delta_zero}).

Now, we are interested in the asymptotic of
$\mathbb{P}_{(n_{0},m_{0})}(T = k )$. Once again thanks to
Lemma~\ref{Delta_zero_iff_curves_circles}, we find that
$\mathcal{L}_{z}$ is also a circle --suppose non degenerated--, but this time with a center
$\tilde{\gamma}(z)$ and a radius
$\tilde{\rho}(z)$ that depend on $z$ and that can
be not equal to zero and one respectively.
In particular, we will perhaps have first to continue
$\tilde{h}$ into a holomorphic function up to $\mathcal{L}_{z}$,
using Proposition~\ref{prop_continuation_Delta_zero}.
$\tilde{\gamma}$ and $\tilde{\rho}$ are
defined in substance in the proof of
Lemma~\ref{Delta_zero_iff_curves_circles}~: indeed, we give there an
explicit expression of the circle $\mathcal{L}_{z}$.
In particular, the CGF
associated to $\mathcal{L}_{z}$ is now equal to
$\tilde{w}_{z}(t)=t+\tilde{\rho}(z)^{2}/(t-\tilde{\gamma}(z))$.
A consequence of
these facts is that we have to adapt the
Proposition~\ref{explicit_h(x,z)_third}, giving the explicit
expression of $\tilde{h}(y,z)$, which, as it is, is no more true.
Skipping over the details, we obtain the following integral
formulation for $\tilde{h}(y,z)$~:
     \begin{equation*}
          \widetilde{h}\left(y,z\right) = \frac{y}{\pi}\int_{y_{3}\left(z\right)}
          ^{y_{4}\left(z\right)}\left(t^{m_{0}}-
          \widetilde{\sigma}\left(t,z\right)^{m_{0}}\right)
          \frac{\widetilde{\mu}_{n_{0}}\left(t,z\right)}{t\left(t-y\right)}
          \sqrt{-\widetilde{d}\left(t,z\right)}\text{d}t+
          yP_{\infty}\left(y\mapsto X_{0}\left(y,z\right)^{n_{0}}
          y^{m_{0}-1}\right)\left(y\right) ,
     \end{equation*}
where $P_{\infty}$ is the principal part, defined in
Lemma~\ref{simplification_principal_part} and
$\tilde{\sigma}(t,z)=\tilde{\gamma}(z)+
\tilde{\rho}(z)^{2}/(t-\tilde{\gamma}(z))$. The function
$\tilde{\sigma}$ satisfies many noteworthy relationships~; for instance,
$\tilde{\sigma}$
leaves $\tilde{w}_{z}$ invariant~:
$\tilde{w}_{z}(\tilde{\sigma}(t,z))=\tilde{w}_{z}(t)$. Then, we can one more
time adapt the change of variable $t=t_{2}(u,z)$, see~(\ref{tttt}),
proposed in Part~\ref{Chebychev_polynomials}, and we finally get the
asymptotic of $\mathbb{P}_{(n_{0},m_{0})}(T = k)$.
\end{proof}

%
%
%
%
%

\subsection{Proofs of lemmas}\label{Proof_lemma}

\noindent{\it Proof of Lemma~\ref{Delta_zero_iff_curves_circles}.}
The proof is based on the explicit expressions of $\mathcal{L}_{z}$
and $\mathcal{M}_{z}$~: in~\cite{FIM}, the authors give the way to
find the explicit expressions of these curves in the particular case
$z=1$. We can adapt this argument and obtain the expressions of the
curves for all $z\in ]0,z_{1}]$~; what is new here is the possibility of
expressing the curves $\mathcal{L}_{z}$ and $\mathcal{M}_{z}$ in terms
of three determinants~: $\mathcal{L}_{z}$ is equal to
$\{u+i v \in \mathbb{C} :
q_{z}(u,v)^{2}-q_{1,z}(u,v)q_{2,z}(u,v)=0\}$ where $q_{z}$,
$q_{1,z}$ and $q_{2,z}$ are respectively equal to~:
\begin{equation*}
          \left|\begin{array}{ccc}
          p_{11}&p_{10}&p_{1-1}\\
          1&-2u&u^{2}+v^{2}\\
          p_{-11}&p_{-10}&p_{-1-1}
          \end{array}\right|,\hspace{5mm}
          \left|\begin{array}{ccc}
          1&-2u&u^{2}+v^{2}\\
          p_{01}&-1/z&p_{0-1}\\
          p_{-11}&p_{-10}&p_{-1-1}
          \end{array}\right|,\hspace{5mm}
          \left|\begin{array}{ccc}
          p_{11}&p_{10}&p_{1-1}\\
          p_{01}&-1/z&p_{0-1}\\
          1&-2u&u^{2}+v^{2}
          \end{array}\right|,
     \end{equation*}
and, of course, we could write a similar expression for $\mathcal{M}_{z}$.
Now that we have the expression of these polynomials, we will
establish some relationships between their coefficients~; but before
take the notations~: for $i=1,2$, $\alpha_{i}$, $\beta_{i}$ and
$\gamma_{i}$ will stand for the coefficients
--depending on $z$-- of $q_{i,z}$~:
$q_{i,z}(u,v)=\alpha_{i}-2\beta_{i}u+\gamma_{i}(u^{2}+v^{2})$~; one
obtains obviously the explicit expression of these coefficients by
expanding the determinants above. Then, by a simple calculation, we
verify the three following facts~: first
$\alpha_{1}\gamma_{2}-\alpha_{2}\gamma_{1}=-\Delta(z)/z$, then
$\alpha_{1}\beta_{2}-\alpha_{2}\beta_{1}=-p_{0-1}\Delta(z)$ and at
last $\gamma_{1}\beta_{2}-\gamma_{2}\beta_{1}=p_{01}\Delta(z)$. In
particular, if $\Delta\left(z\right)=0$, then the polynomials
$q_{1,z}$ and $q_{2,z}$ are proportional. In addition to that,
Cramer's formulas give $z^{-1}q_{z}(u,v)=p_{10}q_{1,z}(u,v)
+p_{-10}q_{2,z}(u,v)+2u\Delta(z)$~; so that if $\Delta(z)=0$,
$q_{z}$, $q_{1,z}$ and $q_{2,z}$ are multiple of a same polynomial.
\textit{A priori}, it may quite happen that one or even several of $q_{z}$,
$q_{1,z}$, $q_{2,z}$ are zero~; in fact, we could show that at most
one of these three polynomials can be equal to zero, otherwise the walk
would be degenerated (we recall from the very beginning of
Section~\ref{Extension_results} that we say that a walk is
degenerated if there are three consecutive zeros in the list
$p_{-1-1},p_{-10},p_{-11},p_{01},p_{11},p_{10},p_{1-1},p_{0-1},p_{-1-1},p_{-10}$).

So in any non degenerated case, we can write that $\mathcal{L}_{z}=
\{u+i v \in \mathbb{C} : r_{z}(u,v)=0\}$, where $r_{z}$ stands for
one of the non zero polynomials in the list $q_{z}$, $q_{1,z}$,
$q_{2,z}$. But the curve $\{u+i v \in \mathbb{C} : r_{z}(u,v)=0\}$
is clearly a circle, eventually degenerated in
a straight line, for which we could easily write the center and
the radius, the proof of Lemma~\ref{Delta_zero_iff_curves_circles}
is completed.

To be exhaustive, we give here
the single possibilities for $\mathcal{L}_{z}$ and $\mathcal{M}_{z}$
to be straight lines~:
{\rm (i)}  $\mathcal{L}_{z}$
           is a straight line if and only if
           $p_{01}+p_{1-1}+p_{0-1}+p_{-1-1}=1$,
           in that case $\mathcal{L}_{z}=\{u+i v : 2p_{01}z u =1\}$
           and $\mathcal{M}_{z}=\mathcal{C}(0,(p_{-1-1}/p_{1-1})^{1/2})$,
{\rm (ii)} $\mathcal{M}_{z}$ is a straight line if and only if
           $p_{10}+p_{-11}+p_{-10}+p_{-1-1}=1$,
           in that case $\mathcal{M}_{z}=\{u+i v : 2p_{10}z u =1\}$
           and $\mathcal{L}_{z}=\mathcal{C}(0,(p_{-1-1}/p_{-11})^{1/2})$.
           The proof of these facts consists simply in a
           play with the parameters, so we omit it.
\hfill $\square $

\bigskip

\noindent{\it Proof of Lemma~\ref{Delta_zero_drift_zero_a=c}.}
Start by showing that $\Delta(1)=0$ is equivalent to
$p_{11}+p_{-1-1}=p_{1-1}+p_{-11}$, and suppose first that
$p_{11}+p_{-1-1}=p_{1-1}+p_{-11}$. We have the system of equations (1)
$M_{x}=0$ (2) $M_{y}=0$ (3) $p_{11}+p_{-1-1}=p_{1-1}+p_{-11}$ (4)
$\sum p_{i j}=1$ (5) $p_{i j} \geq 0$. This system has no
single solutions but implies some
relationships between the parameters, which in turn imply, by a
direct calculation, that $\Delta(1)=1$. Likewise, we prove
that $\Delta(1)$ implies that
$p_{11}+p_{-1-1}=p_{1-1}+p_{-11}$.

Suppose now that $\Delta(1)=\Delta'(1)=0$
and consider the system composed from the equations~: (1)
$M_{x}=0$ (2) $M_{y}=0$ (3) $p_{11}p_{-1-1}-p_{1-1}p_{-11}=0$ (4)
$p_{11}+p_{-1-1}=p_{1-1}+p_{-11}$ (5) $\sum p_{i j}=1$ (6) $p_{i j}
\geq 0$. It turns out that this system implies either
$p_{11}=p_{1-1}$, $p_{01}=p_{0-1}$ and $p_{-11}=p_{-1-1}$ (in other
words $a=c$) or $p_{11}=p_{-11}$, $p_{10}=p_{-10}$ and
$p_{1-1}=p_{-1-1}$ (that means $\tilde{a}=\tilde{c}$).
\hfill $\square $

\bigskip

\noindent{\it Proof of Proposition~\ref{prop_continuation_Delta_zero}.}
First, we lift $x\mapsto h(x,z)$ and $y\mapsto \tilde{h}(y,z)$,
initially defined on the unit disc, on the algebraic curve
$\{(x,y)\in \mathbb{C}^{2} : Q(x,y,z)=0\}$~: we obtain so two functions, say $g$ and $\tilde{g}$,
defined on $\{(x,y)\in \mathbb{C}^{2} : Q(x,y,z)=0, |x|<1\}$ and
$\{(x,y)\in \mathbb{C}^{2} : Q(x,y,z)=0, |y|<1\}$ respectively.
Then, we will continue $g$ and $\tilde{g}$ on the whole $Q(x,y,z)=0$,
into functions again denoted by $g$ and $\tilde{g}$, and verifying
$g(\xi(x,y),z)=g(x,y,z)$ and
$\tilde{g}(\eta(x,y),z)=\tilde{g}(x,y,z)$~; we will explain how to
get this continuation in a few lines. Suppose before that we did
successfully this continuation, and see how to conclude~: we will
continue $h$ and $\tilde{h}$ by setting $h(x,z)=g(x,y,z)$ and
$\tilde{h}(x,z)=\tilde{g}(x,y,z)$. The two relationships
$g(\xi(x,y),z)=g(x,y)$ and $\tilde{g}(\eta(x,y),z)=\tilde{g}(x,y,z)$
allow to the continuation to be not ambiguous.

It remains to see how continue $g$ and $\tilde{g}$ from $\{(x,y)\in
\mathbb{C}^{2} : Q(x,y,z)=0, |x|<1\}$ and $\{(x,y)\in \mathbb{C}^{2}
: Q(x,y,z)=0, |y|<1\}$ to the whole $\{(x,y)\in \mathbb{C}^{2} : Q(x,y,z)=0\}$.
Define $\xi
(x,y)=(x,c(x,z)/(a(x,z)y))$, $\eta
(x,y)=(\tilde{c}(y,z)/(\tilde{a}(y,z)x),y)$ ~; they are the Galois
automorphisms attached to the algebraic curve $Q(x,y,z)$,
see~\cite{MA2} and~\cite{FIM}. They are such that if  $Q\left(x,y,z\right)=0$ then
$Q(\xi(x,y),z)=0$ and $Q( \eta(x,y),z)=0$.

The key step is the following~: a worthwhile fact of having supposed
$\Delta(z)=0$ is that the group $\mathcal{H}$, generated by $\xi$
and $\eta$ is of order four~; we can prove this fact by a direct
calculation.

Consider next the following sub-domains of
$\{(x,y)\in \mathbb{C}^{2} : Q(x,y,z)=0\}$~:
$\mathcal{D}_{z,-,-}=\{(x,y)\in \mathbb{C} : Q(x,y,z)=0,|x|\leq
1,|y|\leq 1\}$, $\mathcal{D}_{z,-,+}=\{(x,y)\in \mathbb{C} :
Q(x,y,z)=0,|x|\leq 1,|y|\geq 1\}$, $\mathcal{D}_{z,+,-}=\{(x,y)\in
\mathbb{C} : Q(x,y,z)=0,|x|\geq 1,|y|\leq 1\}$ and
$\mathcal{D}_{z,+,+}=\{(x,y)\in \mathbb{C} : Q(x,y,z)=0,|x|\geq
1,|y|\geq 1\}$, whose the union 
equals $\{(x,y)\in \mathbb{C}^{2} : Q(x,y,z)=0\}$.
\textit{A priori}, $g$ and $\tilde{g}$ are well defined only
in $\mathcal{D}_{z,-,-}\cup \mathcal{D}_{z,-,+}$ and
$\mathcal{D}_{z,-,-}\cup \mathcal{D}_{z,+,-}$ respectively. Then, we
continue $g$ in $\mathcal{D}_{z,+,-}$ using the functional
equation~(\ref{functional_equation})~: we set
$g(x,y)=x^{n_{0}}y^{m_{0}}-\tilde{g}(x,y)-h_{00}(z)$. Likewise, in
$\mathcal{D}_{z,-,+}$, we continue $\tilde{g}$ by setting
$\tilde{g}(x,y)=x^{n_{0}}y^{m_{0}}-g(x,y)-h_{00}(z)$. In
$\mathcal{D}_{z,+,+}$, we set $g(x,y)=g \circ \xi(x,y)$ and
$\tilde{g}(x,y)=\tilde{g}\circ \eta(x,y)$.

Thanks to the properties of the automorphisms $\xi$ and $\eta$
described above, we get the sub-domains $\mathcal{D}_{z,\pm ,\pm }$
as the successive ranges of $\mathcal{D}_{z,- ,- }$ by the
automorphisms $\xi$ and $\eta$~: for instance
$\xi(\mathcal{D}_{z,+,+})=\mathcal{D}_{z,+,-}$ and
$\eta(\mathcal{D}_{z,+,+})=\mathcal{D}_{z,-,+}$, so that
the continuation is done on the whole
$\{(x,y)\in \mathbb{C}^{2} : Q(x,y,z)=0\}$ and is not ambiguous. \hfill $\square $

\footnotesize

\bibliographystyle{alpha}
\bibliography{RW_RASCHEL}

\end{document}